\documentclass[12pt]{amsart}
\usepackage{amsmath, amsthm, amsxtra, amstext, amssymb, latexsym, dsfont, pb-diagram}
\usepackage[a4paper,top=3cm, left=3cm, right=2cm, bottom=2cm]{geometry}
\usepackage{enumerate, fancyhdr,color,comment,ulem}
\usepackage{graphicx}

\newcommand{\limdir}{\mathrm{limdir}}
\newcommand{\cov}{\mathrm{cov}}
\newcommand{\add}{\mathrm{add}}

\newcommand{\afrak}{\mathfrak{a}}
\newcommand{\bfrak}{\mathfrak{b}}
\newcommand{\minadd}{\mathrm{minadd}}
\newcommand{\mincof}{\mathrm{mincof}}
\newcommand{\cof}{\mathrm{cof}}
\newcommand{\non}{\mathrm{non}}
\newcommand{\supcov}{\mathrm{supcov}}
\newcommand{\supadd}{\mathrm{supadd}}

\newcommand{\dfrak}{\mathfrak{d}}
\newcommand{\ufrak}{\mathfrak{u}}
\newcommand{\sfrak}{\mathfrak{s}}
\newcommand{\minnon}{\mathrm{minnon}}
\newcommand{\mincov}{\mathrm{mincov}}
\newcommand{\cfrak}{\mathfrak{c}}
\newcommand{\Rbb}{\mathbb{R}}
\newcommand{\Scal}{\mathcal{S}}

\newcommand{\supcof}{\mathrm{supcof}}
\newcommand{\Ical}{\mathcal{I}}

\newcommand{\Jcal}{\mathcal{J}}

\newcommand{\Abf}{\mathbf{A}}
\newcommand{\Bbf}{\mathbf{B}}

\newcommand{\Hcal}{\mathcal{H}}
\newcommand{\Pcal}{\mathcal{P}}
\newcommand{\Rcal}{\mathcal{R}}
\newcommand{\Ncal}{\mathcal{N}}
\newcommand{\Mcal}{\mathcal{M}}
\newcommand{\SNcal}{\mathcal{SN}}

\newcommand{\Pbb}{\mathbb{P}}
\newcommand{\Cbb}{\mathbb{C}}

\newcommand{\Qbb}{\mathbb{Q}}
\newcommand{\Ebb}{\mathbb{E}}
\newcommand{\Ecal}{\mathcal{E}}
\newcommand{\Dbb}{\mathbb{D}}
\newcommand{\Tbb}{\mathbb{T}}
\newcommand{\LOCbb}{\mathbb{LC}}

\newcommand{\Sbb}{\mathbb{S}}
\newcommand{\id}{\mathrm{id}}
\newcommand{\la}{\langle}
\newcommand{\ra}{\rangle}
\newcommand{\cf}{\mathrm{cf}}

\newcommand{\frestr}{\!\!\upharpoonright\!\!}

\newcommand{\Dbf}{\mathbf{D}}

\newcommand{\ran}{\mathrm{ran}}

\newcommand{\Lc}{\mathbf{Lc}}
\newcommand{\blc}{\mathfrak{b}^{\mathrm{Lc}}}
\newcommand{\dlc}{\mathfrak{d}^{\mathrm{Lc}}}

\newcommand{\aLc}{\mathbf{aLc}}
\newcommand{\balc}{\mathfrak{b}^{\mathrm{aLc}}}
\newcommand{\dalc}{\mathfrak{d}^{\mathrm{aLc}}}
\newcommand{\Ed}{\mathbf{Ed}}

\newcounter{enuAlph}

\newtheorem{theorem}{Theorem}[section]
\newtheorem{lemma}[theorem]{Lemma}

\newtheorem{Corol}[theorem]{Corollary}
\newtheorem{Question}[theorem]{Question}
\newtheorem{Claim}[theorem]{Claim}
\newtheorem{Subclaim}[theorem]{Subclaim}

\newtheorem{thm}[enuAlph]{Theorem}

\theoremstyle{definition}
\newtheorem{definition}[theorem]{Definition}

\newtheorem{Example}[theorem]{Example}

\theoremstyle{remark}
\newtheorem{remark}[theorem]{Remark}

\numberwithin{equation}{section}

\definecolor{OliveGreen}{cmyk}{0.64,0,0.95,0.40}

\begin{document}

\title{On cardinal characteristics of Yorioka ideals}

\author{Miguel A. Cardona}
\address{TU Wien, Faculty of Mathematics and Geoinformation, Institute of Discrete Mathematics and Geometry, Wiedner Hauptstrasse 8--10, A-1040 Vienna, Austria }
\email{miguel.montoya@tuwien.ac.at}
\thanks{This work was supported by: the Austrian Science Fund (FWF) P23875 (both authors), I1272 (second author) and P30666 (both authors); the first author was supported by the International Academic Mobility Scholarship from Universidad Nacional de Colombia; the second author was supported by the grant no. IN201711, Direcci\'on Operativa de Investigaci\'on, Instituci\'on Universitaria Pascual Bravo, and by Grant-in-Aid for Early Career Scientists 18K13448, Japan Society for the Promotion of Science.}

\author{Diego A. Mej\'{\i}a}
\address{Shizuoka University, Faculty of Sciences. 836 Ohya, Suruga-ku, Shizuoka city, Japan 422-8529}
\email{diego.mejia@shizuoka.ac.jp}
\urladdr{http://www.researchgate.com/profile/Diego\_Mejia2}





\begin{abstract}
   Yorioka \cite{Yorioka} introduced a class of ideals (parametrized by reals) on the Cantor space to prove that the relation between the size of the continuum and the cofinality of the strong measure zero ideal on the real line cannot be decided in ZFC. We construct a matrix iteration of ccc posets to force that, for many ideals in that class, their associated cardinal invariants (i.e. additivity, covering, uniformity and cofinality) are pairwise different. In addition, we show that, consistently, the additivity and cofinality of Yorioka ideals does not coincide with the additivity and cofinality (respectively) of the ideal of Lebesgue measure zero subsets of the real line.
\end{abstract}

\maketitle

\normalem

\section{Introduction}\label{SecIntro}

Yorioka \cite{Yorioka} introduced a characterization of $\SNcal$, the \emph{$\sigma$-ideal of strong measure zero subsets of the Cantor space $2^\omega$}, in terms of $\sigma$-ideals $\Ical_f$ parametrized by increasing functions $f\in\omega^\omega$, which we call \emph{Yorioka ideals} (see Definition \ref{DefYorioId}). Concretely, $\SNcal=\bigcap\{\Ical_f:f\in\omega^\omega\textrm{\ increasing}\}$ and $\Ical_f\subseteq\Ncal$ where $\Ncal$ is the \emph{$\sigma$-ideal of Lebesgue-measure zero subsets of $2^\omega$}. Yorioka used this characterization to show that no inequality between $\cof(\SNcal)$ and $\cfrak:=2^{\aleph_0}$ cannot be decided in ZFC, even more, he proved that $\cof(\SNcal)=\dfrak_\kappa$ (the \emph{dominating number on $\kappa^\kappa$}) whenever $\add(\Ical_f)=\cof(\Ical_f)=\kappa$ for all increasing $f$ (\footnote{Yorioka's original proof assumes $\add(\Ical_f)=\cof(\Ical_f)=\dfrak=\cov(\Mcal)=\kappa$ for all increasing $f$, but $\dfrak$ and $\cov(\Mcal)$ can be omitted thanks to the results in Section \ref{SecZFC}.}).

Further research on Yorioka ideals has been continued by Kamo and Osuga \cite{O,kamo-osuga,K1,KO}. In \cite{kamo-osuga} they proved that, in ZFC, $\add(\Ical_f)\leq\bfrak\leq\dfrak\leq\cof(\Ical_f)$ for all increasing $f$ and that, for any fixed $f$, the basic diagram of the cardinal invariants associated to $\Ical_f$ (see Figure \ref{FigaddetcYorio}) is complete in the sense that no other inequality can be proved in ZFC. On the other hand, in \cite{KO} they constructed models by FS (finite support) iterations of ccc posets where infinitely many cardinal invariants of the form $\cov(\Ical_f)$ are pairwise different. Moreover, if there exists a weakly inaccessible cardinal, then there is a ccc poset forcing that there are continuum many pairwise different cardinals of the form $\cov(\Ical_f)$.

\begin{figure}
  \begin{center}
    \includegraphics[width=12cm,height=4cm]{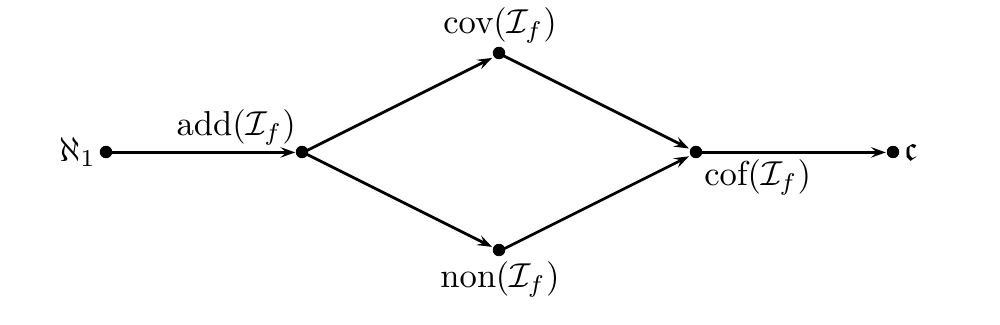}
    \caption{Cardinal invariants associated with $\Ical_f$. Each arrow denotes $\leq$.}
    \label{FigaddetcYorio}
  \end{center}
\end{figure}

To continue this line of research, we aim to obtain further consistency results considering several cardinal invariants associated with Yorioka ideals at the same time, that is, to construct models of ZFC where three or more of such cardinal invariants are pairwise different. Given a family $\Ical$ of subsets of a set $X$, the \emph{cardinal invariants associated with $\Ical$} are the four cardinals $\add(\Ical)$, $\cov(\Ical)$, $\non(\Ical)$ and $\cof(\Ical)$. The main objective of this paper is to prove the following result.

\begin{thm}\label{mainA}
   There is a function $f_0\in\omega^\omega$ and a ccc poset forcing that the four cardinal invariants associated with $\Ical_f$ are pairwise different for each increasing $f\geq^*f_0$.
\end{thm}

Concerning problems of this nature, the consistency of $\add(\Ncal)<\cov(\Ncal)<\non(\Ncal)<\cof(\Ncal)$ with ZFC is a consequence of \cite[Thm. 17]{M}. Quite recently, Goldstern, Kellner and Shelah \cite{GKS} showed the consistency, modulo strongly compact cardinals, of Cicho\'n's diagram separated into 10 different values, in particular, the four cardinal invariants associated with the meager ideal $\Mcal$ on $\mathbb{R}$ are pairwise different is consistent. However, this consistency result alone is unknown without using large cardinals.

Though many inequalities between the cardinal invariants associated with Yorioka ideals and the cardinals in Cicho\'n's diagram are known in ZFC, there are still many open questions. Most of these inequalities had been settled in \cite{kamo-osuga,KO}, however, there are a couple of statements from \cite{K1} whose proofs do not appear anywhere. Namely, $\add(\Ncal)\leq\add(\Ical_g)\leq\add(\Ical_f)$ and $\cof(\Ical_f)\leq\cof(\Ical_g)\leq\cof(\Ncal)$ when $f,g\in\omega^\omega$ are increasing and $f(n+1)-f(n)\leq g(n+1)-g(n)$ for all but finitely many $n<\omega$. We offer our own proofs of these inequalities in Corollaries \ref{YorioaddN} and \ref{2.6}, even more, in Theorem \ref{2.4.1} we show the stronger fact (\footnote{Its proof does not use the fact that $\Ical_f$ is an ideal, and it is simpler than Yorioka's argument to prove that $\Ical_f$ is a $\sigma$-ideal (see \cite[Lemma 3.4]{Yorioka}), though the ideas in both proofs are quite similar.}) that $\add(\Ical_g)$ is above some definable cardinal invariant above $\add(\Ncal)$ (and dually for $\cof(\Ical_f)$). We use this to prove the following new consistency result.

\begin{thm}[Theorem \ref{addN<addIf}]\label{mainB}
   If $f\in\omega^\omega$ is increasing then there is a ccc poset that forces $\add(\Ncal)<\add(\Ical_f)<\cof(\Ical_f)<\cof(\Ncal)$.
\end{thm}

This definable cardinal we use to prove the theorem above is a cardinal characteristic, denoted by $\blc_{b,h}$, that is parametrized by functions $b,h\in\omega^\omega$, which form part of what we call  \emph{localization cardinals} (Definition \ref{Defantiloc}). These cardinals are a generalization of the cardinal invariants Bartoszy\'nski used to characterize $\add(\Ncal)$ and $\cof(\Ncal)$ in terms of a slalom structure (Theorem \ref{Barcharloc}), and they were used by Brendle and the second author in \cite{BrM} to show consistency results about cardinal invariants associated with Rothberger gaps in $F_\sigma$ ideals on $\omega$. In one of these results, it was constructed a ccc poset that forces infinitely many cardinals of the form $\blc_{b,h}$ to be pairwise different (even continuum many modulo a weakly inaccessible). An older precedent is known for $\dlc_{b,h}$, the dual of $\blc_{b,h}$: Goldstern and Shelah \cite{GS93} constructed a proper poset (using creatures) to force that $\aleph_1$-many cardinals of the form $\dlc_{b,h}$ are pairwise different, result that was later improved by Kellner \cite{Kell} for continuum many. In this work, to prove Theorem \ref{mainB}, we show in Theorem \ref{2.4.1} a connection between the cardinals of type $\add(\Ical_f)$, $\cof(\Ical_f)$ and the localization cardinals.

A variation of the localization cardinals that we call \emph{anti-localization cardinals}, denoted by $\balc_{b,h}$ and (its dual) $\dalc_{b,h}$ (for $b,h\in\omega^\omega$), also play an important role in the study of cardinal invariants associated with Yorioka ideals. The cardinal $\dalc_{b,h}$ is known from Miller's \cite{Mi} characterization $\non(\SNcal)=\min\{\dalc_{b,h}:b\in\omega^\omega\}$ (for any $h\geq^*1$). On the other hand, Kellner and Shelah \cite{KellS} constructed a proper poset that forces that, for the types $\balc_{b,h}$ and $\dlc_{b,h}$, continuum many cardinals of each type are pairwise different. In fact, Kamo and Osuga \cite{KO} discovered a connection between the cardinals of type $\cov(\Ical_f)$, $\non(\Ical_f)$ and the anti-localization cardinals (Lemmas \ref{2.7} and \ref{2.8}), which they use to construct a finite support iteration of ccc posets that forces that, for the types $\cof(\Ical_f)$ and $\balc_{b,1}$, infinitely many cardinals of each type are pairwise different (even continuum many when a weakly inaccessible cardinal is assumed).

We use the connections between the cardinal invariants associated with Yorioka ideals and the localization and anti-localization cardinals to prove Theorems \ref{mainA} and \ref{mainB}. Even more, in the consistency result of the first theorem we can additionally include that infinitely many cardinals of each type $\blc_{b,h}$ and $\balc_{b,1}$ are pairwise different.

\begin{thm}[Theorem \ref{main}]\label{mainC}
   There is a function $f_0\in\omega^\omega$ and a ccc poset forcing that
   \begin{enumerate}[(a)]
       \item the four cardinal invariants associated with $\Ical_f$ are pairwise different for each increasing $f\geq^*f_0$,
       \item infinitely many cardinals of the form $\blc_{b,h}$ are pairwise different, and
       \item infinitely many cardinals of the form $\balc_{b,h}$ and $\cov(\Ical_f)$ are pairwise different.
   \end{enumerate}
\end{thm}

Within this result, we merge with Theorem \ref{mainA} the consistency result of infinitely many localization cardinals from \cite{BrM}, and the consistency result of infinitely many anti-localization cardinals and $\cov(\Ical_f)$ from \cite{KO}.

This result is proved by using a FS iteration, so it forces that $\non(\Mcal)\leq\cov(\Mcal)$. Because of this, as ZFC proves that $\balc_{b,h}\leq\non(\Mcal)$ and $\cov(\Mcal)\leq\non(\SNcal)\leq\non(\Ical_f)$ for any $b,h,f$, we cannot expect continuum many values in (b) and (c) above. However, Theorem \ref{main} is stated in such generality that (c) can be obtained for continuum many values when a weakly inaccessible is assumed (of course, in this case (a) can not forced). The same applies for (b) but not simultaneously with (c) because of the limitation in our result that the values for the cardinals in (b) appear below those of (c). A curious fact is that, in our result, any single value from (b) and (c) can be repeated continuum many times for different parameters.

The forcing method we use to prove Theorems \ref{mainA}, \ref{mainB} and \ref{mainC} is the technique of \emph{matrix iterations} to construct two dimensional arrays of posets by FS iterations. This method has been very useful to obtain models where several cardinal characteristics of the continuum are pairwise different. It was introduced for the first time by Blass and Shelah \cite{B1S} to prove the consistency of $\ufrak<\dfrak$ with large continuum (i.e. $\cfrak>\aleph_2$) where $\ufrak$ is the \emph{ultrafilter number}. Later on, this method was improved by Brendle and Fischer \cite{BrF} when they proved the consistency of $\bfrak=\afrak<\sfrak$ and $\aleph_1<\sfrak=\bfrak<\afrak$ with large continuum, the latter assuming the existence of a measurable cardinal in the ground model. The second author \cite{M} induced known preservation results in such type of iterations to construct models where several cardinals in Cicho\'n's diagram are pairwise different. Just a while ago, Dow and Shelah \cite{DowShelah} used this technique to prove that the splitting number $\sfrak$ is consistently singular.



To guarantee that the matrix iteration constructed to prove the theorem forces the desired values for the cardinal invariants, we propose a more general version of the classical preservation theory of Judah and Shelah \cite{JS} and Brendle \cite{Br}, along with the corresponding version for matrix iterations of the second author \cite{M}. This generalization also looks to describe the preservation property that Kamo and Osuga \cite{KO} proposed to control the covering of Yorioka ideals. Concretely, they proposed a preservation property for cardinal invariants of the form $\balc_{b,h^{\id_\omega}}$ in order to decide values of $\cov(\Ical_f)$ thanks to the relations they discovered between both types of cardinals. Although this preservation property is similar to Judah-Shelah's and Brendle's preservation theory, it is not a particular case of it. In view of this, we generalize this classical preservation theory so that the preservation property of Kamo and Osuga becomes a particular case (Example \ref{3.18}). Even more, our theory also covers the preservation property defined in \cite[Sect. 5]{BrM} to decide cardinal invariants related to Rothberger gaps by forcing (Example \ref{ExmBM}). In addition, we add a particular case of our theory to deal directly with the preservation of cardinals of the form $\blc_{b,h}$ (Example \ref{ExLc*}).


This paper is structured as follows. In Section \ref{SecPre} we review the essential notions of this paper: localization and anti-localization cardinals, Yorioka ideals and related forcing notions and their properties. In Section \ref{SecZFC} we show the connection between $\add(\Ncal)$, $\cof(\Ncal)$ and the localization cardinals, and we review some inequalities between the cardinal invariants associated to Yorioka ideals that are known in ZFC. Additionally, we show what happens to the localization and anti-localization cardinals in non-standard cases, that is, in cases where the function $b$ is allowed to take uncountable values. Section \ref{SecPreservation} is devoted to our general preservation theory. In Section \ref{4.2} we prove the main results of this paper. Finally, we present in Section \ref{SecQ} discussions and open questions related to this work.


\section{Notation and preliminaries}\label{SecPre}

Throughout this text, we refer to members of any uncountable Polish space as \textit{reals}. We write $\exists^{\infty}{n<\omega}$ and $\forall^{\infty}{n<\omega}$ to abbreviate `for infinitely many natural numbers' and `for all but
finitely many natural numbers', respectively. For sets $A,B$, denote by $B^A$ the set of functions from $A$ to $B$. For functions $f,g$ from $\omega$ into the ordinals,  $f\leq g$  means that $f(n) \leq g(n)$ for all $n <\omega $. We say that $g$ \textit{(eventually) dominates} $f$, denoted by $f\leq^* g$, if $\forall^\infty_{n<\omega}(f(n)\leq g(n))$. Define $f<g$ and $f<^*g$ similarly. $F\subseteq \omega^\omega$ is \textit{bounded} if it is dominated by a single function in $\omega^\omega$, i.e, there is a $g\in \omega^\omega$ such that $f\leq ^* g$ for all $f\in F$. A not bounded set is called \textit{unbounded}. A set $D\subseteq\omega^\omega$ is \textit{dominating} if every $f\in\omega^{\omega}$  is dominated by some member of $D$.

We extend some known operations in the natural numbers as operations between functions of natural numbers defined point-wise. For example, for $f,g \in \omega^{\omega}$, $f+g\in\omega^\omega$ is defined as $(f+g)(n)=f(n)+g(n)$, likewise for the product and exponentiation. Also, we use natural numbers (and even ordinal numbers) to denote constant functions with domain $\omega$. We denote by $f^+$ and $\log f$ the functions from $\omega$ to $\omega$ defined by:
\[ f^+(n)=\sum_{j\leq n}f(j)\textrm{\ and }  \log f (n)=\min\{ k<\omega : f(n)\leq 2^k\}.\]
For any set $A$, $\id_{A}$ denotes the \textit{identity function on $A$}.

Given a function $b$ with domain $\omega$ such that $b(i)\neq\emptyset$ for all $i<\omega$, $h\in \omega^\omega$ and $n<\omega$, define

\[\Scal_n(b,h)=\prod_{i<n}[b(i)]^{\leq h(i)},\ \Scal_{<\omega}(b,h)=\bigcup_{n<\omega}\Scal_{n}(b,h)\textrm{\ and } \Scal(b,h)=\prod_{n<\omega}[b(n)]^{\leq h(n)}.\]

Depending on the context, for a set $A$ we denote $\Scal_{<\omega}(A,h)=\Scal_{<\omega}(b,h)$ where $b:\omega\to \{A\}$. Similarly, we use $\Scal_n(A,h)$ and $\Scal(A,h)$.

\newcommand{\Seq}{\mathrm{seq}_{<\omega}}

Denote $\prod b:=\prod_{i<\omega}b(i)$ and $\Seq(b):=\bigcup_{n<\omega}\prod_{i<n}b(i)$. As a topological space, we endow $\prod b$ with the product topology where each $b(i)$ has the discrete topology. Note that the sets of the form $[s]:=[s]_b:=\{x\in\prod b: s\subseteq x\}$ form a basis of this topology. In particular, if each $b(i)$ is countable then $\prod b$ is a Polish space, and it is perfect iff $\exists^\infty i<\omega(|b(i)|\geq 2)$.

\subsection{Relational systems and cardinal invariants}\label{SubsecInv}

Many of the classical cardinal invariants can be expressed by relational systems, and inequalities between these cardinals are induced by the \textit{Tukey-Galois order} between the corresponding relational systems. These notions where defined by Votjas \cite{V}.

\begin{definition}\label{1.40.1} A \textit{relational system} is a triple $\Abf=\langle A_-,A_+,\sqsubset\rangle$ consisting of two non-empty sets $A_{+}, A_{-}$ and a binary relation $\sqsubset \subseteq A_{-}\times A_{+}$. If $\Abf=\langle A_-,A_+,\sqsubset\rangle$ is a relational system, define \textit{the dual of $\Abf$} as the relational system $\Abf^{\perp}=\langle A_+,A_-,\not\sqsupset\rangle$.

For $x\in A_-$ and $y\in A_+$, $x\sqsubset y$ is often read $y\sqsubset$-\textit{dominates} $x$. A family $X\subseteq A_-$ is \textit{$\Abf$-bounded} if there is a member of $A_{+}$ that $\sqsubset$-dominates every member of $X$, otherwise we say that the set is \emph{$\Abf$-unbounded}. On the other hand, $Y\subseteq A_+$ is \textit{$\Abf$-dominating} if every member of $A_{-}$ is $\sqsubset$-dominated by some member of $Y$.
Define $\dfrak(\Abf)$ as the smallest size of an $\Abf$-dominating family, and  $\bfrak(\Abf)$ as the smallest size of an $\Abf$-unbounded family.
\end{definition}

Note that $\dfrak(\Abf)=\bfrak(\Abf^{\perp})$ and $\bfrak(\Abf)=\dfrak(\Abf^{\perp})$. Depending on the relational system, $\bfrak(\Abf)$ or $\dfrak(\Abf)$ may not exist, case in which the non-existent cardinal is treated as `something' above all the cardinal numbers. Clearly, $\bfrak(\Abf)$ does not exist iff $\dfrak(\Abf)=1$, and $\dfrak(\Abf)$ does not exist iff $\bfrak(\Abf)=1$.

Many classical cardinal invariants can be expressed through relational systems.

\begin{Example}\label{ExmCichonrelsys}
\begin{enumerate}[(1)]
    \item Let $\Ical$ be a family of subsets of a non-empty set $X$ which is downwards $\subseteq$-closed and $\emptyset\in\Ical$. Clearly, $\la\Ical,\Ical,\subseteq\ra$ and $\la X,\Ical,\in\ra$ are relational systems, $\add(\Ical)=\bfrak\la\Ical,\Ical,\subseteq\ra$, $\cof(\Ical)=\dfrak\la\Ical,\Ical,\subseteq\ra$, $\non(\Ical)=\bfrak\la X,\Ical,\in\ra$, and $\cov(\Ical)=\dfrak\la X,\Ical,\in\ra$.

    \item $\Dbf:=\la\omega^\omega,\omega^\omega,\leq^{*}\ra$ is a relational system, $\bfrak=\bfrak(\Dbf)$ and $\dfrak=\dfrak(\Dbf)$.

    \item Let $b$ be a function with domain $\omega$ such that $b(i)\neq\emptyset$ for all $i<\omega$. Define the relation $\neq^*$ on $\prod b$ as $x\neq^* y$ iff $\forall^\infty i<\omega(x(i)\neq y(i))$. This means that \emph{$x$ and $y$ are eventually different}. Denote its negation by $=^\infty$. Put $\Ed(b):=\la\prod b,\prod b,\neq^*\ra$ ($\Ed$ stands for \emph{eventually different}). It is clear that $\Ed(b)^\perp=\la\prod b,\prod b,=^\infty\ra$.
\end{enumerate}
\end{Example}

Functions $\varphi:\omega\to [\omega]^{<\omega}$ are often called \emph{slaloms}.

\begin{definition}\label{Defantiloc}
\begin{enumerate}[(1)]
    \item For two functions $x$ and $\varphi$ with domain $\omega$, define
     \begin{enumerate}[({1.}1)]
         \item $x\in^{*}\varphi$ by $\forall^{\infty}{n<\omega}(x(n)\in \varphi(n))$, which is read \textit{$\varphi$ localizes $x$};
         \item $x\in^{\infty}\varphi$ by $\exists^{\infty}{n<\omega}(x(n)\in \varphi(n))$. Denote its negation by $\varphi\not\ni^*x$, which is read \emph{$\varphi$ anti-localizes $x$}.
     \end{enumerate}
    \item Let $b$ be a function with domain $\omega$ such that $b(i)\neq\emptyset$ for all $i<\omega$, and let $h\in\omega^\omega$.
    \begin{enumerate}[({2.}1)]
        \item Define $\Lc(b,h):=\la\prod b,\Scal(b,h),\in^*\ra$ ($\Lc$ stands for \emph{localization}), which is a relational system. Put $\blc_{b,h}:=\bfrak(\Lc(b,h))$ and $\dlc_{b,h}:=\dfrak(\Lc(b,h))$, to which we often refer to as \emph{localization cardinals}.
        \item Define $\aLc(b,h):=\la\Scal(b,h),\prod b,\not\ni^*\ra$ ($\aLc$ stands for \emph{anti-localization}), which is a relational system. Note that $\aLc(b,h)^\perp=\la\prod b,\Scal(b,h),\in^\infty\ra$. Define $\balc_{b,h}:=\bfrak(\aLc(b,h))$ and $\dalc_{b,h}:=\dfrak(\aLc(b,h))$, to which we refer to as \emph{anti-localization cardinals}.
    \end{enumerate}
\end{enumerate}
\end{definition}



Goldstern and Shelah \cite{GS93}, and Kellner and Shelah \cite{Kell,KellS} have studied cardinal coefficients of the form $\dlc_{b,h}$ and $\balc_{b,h}$ for $b\in\omega^\omega$ (with a different notation in their work) to provide the first examples of models where continuum many cardinal characteristics of the continuum are pairwise different. On the other hand, Brendle and Mej\'ia \cite{BrM} investigated cardinals of the form $\blc_{b,h}$ for $b\in\omega^\omega$ in relation with gap numbers for $F_\sigma$-ideals. In this paper we also look at $\dalc_{b,h}$, the dual cardinal of $\balc_{b,h}$.

When $b$ is the constant function $\omega$, the localization and anti-localization cardinals provide the following well-known characterizations of classical cardinal invariants.

\begin{theorem}[Bartoszy\'nski {\cite[Thm. 2.3.9]{BJ}}]\label{1.6}\label{Barcharloc}
If $h\in\omega^\omega$ goes to infinity then $\add(\Ncal)=\blc_{\omega,h}$ and
$\cof(\Ncal)=\dlc_{\omega,h}$.
\end{theorem}

\begin{theorem}[Bartoszy\'nski {\cite[Lemmas 2.4.2 and 2.4.8]{BJ}}]\label{1.7}\label{Barcharalc}
    If $h\in\omega^\omega$ and $h\geq^*1$ then
    $\balc_{\omega,h}=\non(\Mcal)$ and $\dalc_{\omega,h}=\cov(\Mcal)$.
\end{theorem}

\newcommand{\leqT}{\preceq_{\mathrm{T}}}
\newcommand{\eqT}{\cong_{\mathrm{T}}}

\begin{remark}\label{notrivLoc-aLoc}
   Fix $b$ and $h$ as in Definition \ref{Defantiloc}. The following items justify that the study of localization and anti-localization cardinals can be reduced to the case when $1\leq h$ and $h(i)<|b(i)|$ for all $i<\omega$.
   \begin{enumerate}[(1)]
       \item If $|b(i)|\leq h(i)$ for all but finitely many $i<\omega$, $\dlc_{b,h}=1$ and $\blc_{b,h}$ is undefined. On the other hand, if $\exists^\infty{i<\omega}(h(i)=0)$ then $\blc_{b,h}=1$ and $\dlc_{b,h}$ is undefined. Hereafter, both localization cardinals for $b,h$ are defined only when $1\leq^*h$ and $\exists^\infty{i<\omega}(h(i)<|b(i)|)$, even more, if $A:=\{i<\omega:1\leq h(i)<|b(i)|\}$ (which is infinite) and $\iota_A:\omega\to A$ is the increasing enumeration of $A$, then $\blc_{b,h}=\blc_{b\circ\iota_A,h\circ\iota_A}$ and $\dlc_{b,h}=\dlc_{b\circ\iota_A,h\circ\iota_A}$ (actually, $\Lc(b,h)\eqT\Lc(b\circ\iota_A,h\circ\iota_A)$, see Definition \ref{1.45.1}).

      \item If $\exists^\infty i<\omega(|b(i)|\leq h(i))$ then $\balc_{b,h}=1$ and $\dalc_{b,h}$ is undefined. On the other hand, if $\forall^\infty i<\omega(h(i)=0)$ then $\dalc_{b,h}=1$ and $\balc_{b,h}$ is undefined. Hence, both anti-localization cardinals for $b,h$ are defined iff $h(i)<|b(i)|$ for all but finitely many $i<\omega$, and $\exists^\infty i<\omega(h(i)\geq 1)$, even more, if $A$ and $\iota_A$ are as in (1) then $\balc_{b,h}=\balc_{b\circ\iota_A,h\circ\iota_A}$ and $\dalc_{b,h}=\dalc_{b\circ\iota_A,h\circ\iota_A}$ (actually, $\aLc(b,h)\eqT\aLc(b\circ\iota_A,h\circ\iota_A)$, see Definition \ref{1.45.1}).
   \end{enumerate}
\end{remark}

\begin{definition}[{\cite[Def. 4.8]{B10}}]\label{1.45.1}
   Let $\Abf=\langle A_-,A_+,\sqsubset\rangle$ and $\Bbf=\langle B_-,B_+,\sqsubset^{\prime}\rangle$ be relational systems. Say that \emph{$\Abf$ is Tukey-Galois below $\Bbf$}, denoted by $\Abf\leqT \Bbf$, if there exist functions $\varphi_{-}:A_-\to B_-$ and $\varphi_{+}:B_+\to A_+$ such that, for all $x\in A_-$ and $b\in B_+$, if $\varphi_{-}(x)\sqsubset^{\prime}b$ then $x\sqsubset\varphi_{+}(b)$. Here, we say that the pair \emph{$(\varphi_-,\varphi_+)$ witnesses $\Abf\leqT \Bbf$}. Say that $\Abf$ and $\Bbf$ are \textit{Tukey-Galois equivalent}, denoted by  $\Abf\eqT\Bbf$, if $\Abf\leqT \Bbf$ and $\Bbf\leqT \Abf$.
\end{definition}

\begin{theorem}[{\cite[Thm. 4.9]{B10}}]\label{1.46.1} Assume $\Abf\leqT \Bbf$ and that this is witnessed by $(\varphi_-,\varphi_+)$.
\begin{enumerate}[(a)]
\item If $D\subseteq B_+$ is $\Bbf$-dominating then $\varphi_{+}[D]$ is $\Abf$-dominating.
\item If $C\subseteq A_-$ is $\Abf$-unbounded then $\varphi_{-}[D]$ is $\Bbf$-unbounded.
\end{enumerate}
In particular, $\dfrak(\Abf)\leq \dfrak(\Bbf)$ and $\bfrak(\Bbf)\leq \bfrak(\Abf)$.
\end{theorem}

\begin{Example}\label{1.46}\label{ExTukeyorder}
\begin{enumerate}[(1)]
    \item Let $\Ical\subseteq \Jcal$ be two downwards $\subseteq$-closed families of subsets of a non-empty set $X$ such that $\emptyset\in\Ical$. Clearly, $\langle X, \Jcal, \in \rangle\leqT\langle X, \Ical, \in \rangle$. In particular, $\cov(\Jcal)\leq\cov(\Ical)$ and $\non(\Ical)\leq \non(\Jcal)$.

    \item Let $b$ be a sequence of length $\omega$ of non-empty sets and $h\in\omega^\omega$ such that $h\geq^{*}1$. Define $\varphi_+=\id_{\prod b}$ and
          $\varphi_{-}: \prod b \to \Scal(b,h)$ where $\varphi_-(x):=s_{x}$ is defined as  $s_x(i):=\{x(i)\}$ if $h(i)\neq 0$, or $s_x(i):=\emptyset$ otherwise.
     Note that, for any $x,y\in\prod{b}$, $s_x\not\ni^*y$ implies $x\neq^*y$, so $\Ed(b)\leqT\aLc(b,h)$. Even more, $\Ed(b)\eqT\aLc(b,1)$. Hence
      $\balc_{b,h}\leq\balc_{b,1}=\bfrak\la\prod b,\prod b,\neq^*\ra$ and $\dalc_{b,h}\geq\dalc_{b,1}=\dfrak\la\prod b,\prod b,\neq^*\ra$.

     In addition, if $b\in\omega^\omega$ then $\aLc(b,1)\eqT\Lc(b,b-1)$, so $\balc_{b,1}=\blc_{b,b-1}$ and $\dalc_{b,1}=\dalc_{b,b-1}$.

    \item Let $b, b'$ be sequences of length $\omega$ of non-empty sets, and let $h,h' \in \omega^\omega$. If $\forall^\infty i<\omega(|b(i)|\leq|b'(i)|)$ and $h'\leq^* h$ then $\Lc(b,h)\leqT\Lc(b',h')$ and $\aLc(b',h')\leqT\aLc(b,h)$. Hence
    \begin{enumerate}[(i)]
        \item $\blc_{b',h'}\leq\blc_{b,h}$ and $\dlc_{b,h}\leq\dlc_{b',h'}$,
        \item $\balc_{b,h}\leq\balc_{b',h'}$ and $\dalc_{b',h'}\leq \dalc_{b,h}$.
    \end{enumerate}
     By Theorem \ref{1.7}, if each $b(i)$ is countable and $h\geq^*1$ then $\balc_{b,h}\leq\non(\Mcal)$ and $\cov(\Mcal)\leq\dalc_{b,h}$. In addition, if $h$ goes to infinity then $\add(\Ncal)\leq\blc_{b,h}$ and $\dlc_{b,h}\leq\cov(\Ncal)$ by Theorem \ref{1.6}.

     \item Let $b$ and $h$ be as in (2). If $\iota:\omega\to\omega$ is a one-to-one function, then $\Lc(b\circ\iota,h\circ\iota)\leqT\Lc(b,h)$ and $\aLc(b\circ\iota,h\circ\iota)\leqT\aLc(b,h)$. Hence $\blc_{b,h}\leq\blc_{b\circ\iota,h\circ\iota}$ and $\dlc_{b\circ\iota,h\circ\iota}\leq\dlc_{b,h}$, likewise for the anti-localization cardinals.

     Even more, if $\omega\smallsetminus\ran\iota$ is finite then $\Lc(b\circ\iota,h\circ\iota)\eqT\Lc(b,h)$ and $\aLc(b\circ\iota,h\circ\iota)\eqT\aLc(b,h)$.

     \item If $b$ and $h$ are as in (2), then $\aLc(b,h)^\perp\leqT\Lc(b,h)$, so $\balc_{b,h}\leq\dlc_{b,h}$ and $\blc_{b,h}\leq\dalc_{b,h}$.
\end{enumerate}
\end{Example}


\begin{definition}[{\cite[Def. 4.10]{B10}}]\label{1.47}
Let $\Abf=\langle A_-,A_+,\sqsubset\rangle$ and $\Bbf=\langle B_-,B_+,\sqsubset^{\prime}\rangle$ be relational systems. Define the following relational systems.
\begin{enumerate}[(1)]
\item The \textit{conjunction} $\Abf\wedge\Bbf=\langle A_-\times B_-,A_+\times B_+,\sqsubset_{\wedge}\rangle$
where the binary relation $\sqsubset_{\wedge}$ is defined as $(x,y)\sqsubset_{\wedge} (a,b)$ iff $x\sqsubset a$ and $y\sqsubset^{\prime} b$.
\item The \textit{sequential composition} $(\Abf;\Bbf)=\langle A_-\times B_-^{A_+},A_+\times B_+,\sqsubset_{;}\rangle$ where the binary relation $(x,f)\sqsubset_{;}(a,b)$ means $x\sqsubset a$ and $f(a)\sqsubset^{\prime} b$.
\end{enumerate}
\end{definition}

The following result describes the effect of the conjunction and sequential composition on the corresponding cardinal invariants.

\begin{theorem}[{\cite[Thm. 4.11]{B10}}]\label{1.48}\label{seqcompinv}
Let $\Abf=\langle A_-,A_+,\sqsubset\rangle$ and $\Bbf=\langle B_-,B_+,\sqsubset^{\prime}\rangle$ be relational systems. Then:
\begin{enumerate}[(a)]
\item $\max\{\dfrak(\Abf),\dfrak(\Bbf)\}\leq\dfrak(\Abf\wedge\Bbf)\leq\dfrak(\Abf)\cdot\dfrak(\Bbf)$ and $\bfrak(\Abf\times\Bbf)=\min\{\bfrak(\Abf),\bfrak(\Bbf)\}$.
\item $\dfrak(\Abf;\Bbf)=\dfrak(\Abf)\cdot\dfrak(\Bbf)$ and $\bfrak(\Abf;\Bbf)=\min\{\bfrak(\Abf),\bfrak(\Bbf)\}$.
\end{enumerate}
\end{theorem}

\subsection{Yorioka ideals}\label{SubsecYideal}

\begin{definition}\label{DefYorionotation}
For $\sigma \in (2^{<\omega})^{\omega}$ define
\[[\sigma]_\infty:=\{x \in 2^{\omega}:\exists^{\infty}{n < \omega}^{}(\sigma(n) \subseteq x)\}=\bigcap_{n<\omega} \bigcup_{m \geqslant n}[\sigma(m)].
\]
Also define  $h_{\sigma}\in\omega^{\omega}$ such that $h_{\sigma}(i)=|\sigma(i)|$ for each $i<\omega$.

Define the relation $\ll$ on $\omega^\omega$ by
$f\ll g$  iff $\forall{k<\omega}\forall^{\infty}{n<\omega}(f(n^k)\leq g(n))$.
\end{definition}

\begin{definition}[Yorioka {\cite{Yorioka}}]\label{DefYorioId}
    For each $f \in \omega^\omega$ define the families
    \[\mathcal{J}_{f}:=\{X\subseteq 2^{\omega}:\exists{\sigma \in (2^{<\omega})^{\omega}}(X \subseteq [\sigma]_\infty\text{\ and }h_{\sigma}=f )\}\textrm{\ and }
    \mathcal{I}_{f}:=\bigcup_{g\gg f}\mathcal{J}_{g}.\]
    Any family of the form $\mathcal{I}_{f}$ with $f$ increasing is called a \textit{Yorioka ideal}.
\end{definition}

It is clear that both $\Jcal_g$ and $\Ical_f$ contain all the finite subsets of $2^{\omega}$ and that they are downwards $\subseteq$-closed. Note that $f\leq^* f'$ implies $\mathcal{J}_{f'}\subseteq \mathcal{J}_{f}$ and $\mathcal{I}_{f'}\subseteq \mathcal{I}_{f}$, so $\cov(\mathcal{I}_{f})\leq \cov(\mathcal{I}_{f'})$ and $\non(\Ical_{f})\geq\non(\Ical_{f'})$, likewise for $\Jcal_f$ and $\Jcal_{f'}$. Moreover $\Ical_f\subseteq\Jcal_f$ and $f\ll g $ implies $\mathcal{J}_{g}\subseteq  \mathcal{I}_{f}$, so $\cov(\mathcal{J}_{f})\leq \text{cov}(\mathcal{I}_{f})\leq \text{cov}(\mathcal{J}_{g})$ and $\non(\mathcal{J}_{g})\leq \text{non}(\mathcal{I}_{f})\leq \text{non}(\mathcal{J}_{f})$. On the other hand, $\Jcal_f\subseteq\Ncal$ iff the series $\sum_{i<\omega}2^{-f(i)}$ converges, and $\SNcal\subseteq\Jcal_f$. Hence $\SNcal\subseteq\Ical_f\subseteq\Ncal$ when $f$ is increasing, so $\cov(\mathcal{N})\leq \cov(\mathcal{I}_f)\leq \cov(\mathcal{SN})$ and $\non(\mathcal{SN})\leq \non(\mathcal{I}_f)\leq \non(\mathcal{N})$.

\begin{theorem}[Yorioka \cite{Yorioka}]\label{Yoriosigmaid}
If $f \in \omega^\omega$ is an increasing function then $\mathcal{I}_{f}$ is a $\sigma$-ideal. Moreover, $\mathcal{SN}=\bigcap\{\Ical_f:f\textrm{\ increasing}\}$.
\end{theorem}

In contrast, Kamo and Osuga \cite{kamo-osuga} proved that $\Jcal_f$ is not closed under unions when $f(i+1)-f(i)\geq 3$ for all but finitely many $i<\omega$.

Denote $\omega^{\uparrow\omega}:=\{d\in\omega^{\omega}:d(0)=0\textrm{\ and }d\textrm{\ is increasing}\}$. For $d \in \omega^{\uparrow \omega}$ and an increasing $f\in\omega^\omega$, define $g_{d}^{f} \in \omega^\omega$ by $g_{d}^{f}(n)=f(n^{k+10})$  when  $  n \in [d(k),d(k+1))$.
Note that $d\leq^* e$ iff $g_{e}^{f}\leq^{*}g_{d}^{f}$, and $d\leq e$ iff $g_{e}^{f}\leq g_{d}^{f}$.

\begin{lemma}[Osuga \cite{O}]\label{lemosuga} For each $d\in\omega^{\uparrow \omega}$, $g^f_{d}$ is increasing and $g_{d}^{f}\gg f$ for each $d \in \omega^{\uparrow \omega}$. Conversely, for each $g\gg f$ there exists a $d \in \omega^{\uparrow \omega}$ such that $g_{d}^{f}\leq^* g$. In particular, if $D\subseteq\omega^{\uparrow\omega}$ is a dominating family, then  $\mathcal{I}_{f}=\bigcup_{d \in D}\mathcal{J}_{g_{d}^{f}}$.
\end{lemma}

\begin{lemma}\label{f+c=f}
If $f\in \omega^\omega$ is increasing and $c\in\omega$ then $f\ll g$ iff $f+c\ll g$ for all $g\in \omega^\omega$. In particular $\Ical_{f}=\Ical_{f+c}$.
\end{lemma}
\begin{proof}
   Assume $f\ll g$. Fix a natural number $k\geq 1$ and choose $m>c+2$ such that $\forall{n\geq m}(f(n^{k+1})\leq g(n))$. As $f$ is increasing, for all $n\geq m$,
   \[f(n^k)+c\leq f(n^{k}+c)\leq f(n^{k}+n)\leq f(n^{k+1})\leq g(n).\]
\end{proof}

\subsection{Forcing}\label{SubsecForcing}

The basics of forcing can be found in \cite{Je2}, \cite{Ke2} and \cite{Ku}. See also \cite{BJ} for further information about Suslin ccc forcing. Unless otherwise stated,  we denote the ground model by $V$. When dealing with an iteration over a model $V$, we denote by $V_\alpha$ the generic extension at the $\alpha$-th stage.

Recall the following stronger versions of the countable chain condition of a poset.

\begin{definition}\label{Defcc}
Let $\Pbb$ be a forcing notion and $\kappa$ an infinite cardinal. 
\begin{enumerate}[(1)]
\item For $n<\omega$, $B\subseteq \Pbb$ is $n$-\textit{linked} if, for every $F\subseteq B$ of size $\leq n$, $\exists q\in \Pbb\forall p\in F(q\leq p)$.
\item $C\subseteq \Pbb$ is  \textit{centered} if it is $n$-linked for every $n<\omega$.
\item $\Pbb$ is $\kappa$-\textit{linked} if $\Pbb=\bigcup_{\alpha<\kappa}P_\alpha$ where each $P_{\alpha}$ is $2$-linked. When $\kappa=\omega$, we say that $\Pbb$ is $\sigma$-\textit{linked}.
\item $\Pbb$ is $\kappa$-\textit{centered} if $\Pbb=\bigcup_{\alpha<\kappa}P_\alpha$ where each $P_{\alpha}$ is centered. When $\kappa=\omega$, we say that $\Pbb$ is $\sigma$-\textit{centered}.
\item \emph{$\Pbb$ has $\kappa$-cc (the $\kappa$-chain condition)} if every antichain in $\Pbb$ has size $<\kappa$. \textit{$\Pbb$ has ccc (the countable chain contidion)} if it has $\aleph_{1}$-cc.
\end{enumerate}
\end{definition}

Any $\kappa$-centered poset is $\kappa$-linked and any $\kappa$-linked poset has $\kappa^{+}$-cc.

The following generalization of the notion of $\sigma$-linkedness is fundamental in this work.

\begin{definition}[Kamo and Osuga {\cite{KO}}]\label{link}
Let $\rho,\pi \in \omega^\omega$. A forcing notion $\Pbb$ is \textit{$(\rho,\pi)$-linked} if there exists a sequence $\langle Q_{n,j}:n<\omega, j<\rho(n)\rangle$ of subsets of $\Pbb$ such that
\begin{enumerate}[(i)]
\item $Q_{n,j}$ is $\pi(n)$-linked for all $n<\omega$ and $j<\rho(n)$, and
\item $\forall p\in \Pbb\forall^{\infty}{n<\omega}\exists{j<\rho(n)}(p\in Q_{n,j})$.
\end{enumerate}
Here, condition (ii) can be replaced by
\begin{enumerate}[(i')]
\setcounter{enumi}{1}
    \item $\forall p\in \Pbb\forall^{\infty}{n<\omega}\exists{j<\rho(n)}\exists q\leq p(q\in Q_{n,j})$
\end{enumerate}
because (i) and (ii') imply that the sequence of $Q'_{n,j}:=\{q\in\Pbb:\exists p\in Q_{n,j}(q\leq p)\}$ ($n<\omega$ and $j<\rho(n)$) satisfies (i) and (ii).
\end{definition}

\begin{lemma}\label{centrlink}
If $\Pbb$ is $\sigma$-centered then $\Pbb$ is $(\rho,\pi)$-linked when $\rho:\omega\to\omega$ goes to $+\infty$.
\end{lemma}
\begin{proof} Suppose that $\Pbb=\bigcup_{n<\omega}P_n$ where each $P_{n}$ is centered. For each $n\in \omega$, define $Q_{n,j}=P_{j}$ for $j<\rho(n)$. It is clear that  $\langle Q_{n,j}:n<\omega, j<h(n)\rangle$  satisfies (i) and (ii) of  Definition \ref{link}.
\end{proof}

\begin{lemma}[{\cite[Lemma 6]{KO}}]\label{linksigmalink}
If $\Pbb$ is $(\rho,\pi)$-linked and $\pi\not\leq^{*}1$ then $\Pbb$ is $\sigma$-linked.
\end{lemma}

To fix some notation, for each set $\Omega$ let $\Cbb_\Omega$ be the finite support product of \emph{Cohen forcing} $\Cbb:=\omega^{<\omega}$ (ordered by end-extension) indexed by $\Omega$; $\Dbb$ is \emph{Hechler forcing}, i.e. the standard $\sigma$-centered poset that adds a dominating real; and $\mathds{1}$ denotes the trivial poset. These three posets are Suslin ccc forcing notions.

The following poset was defined by Kamo and Osuga to increase the cardinal $\balc_{b,h}$.

\begin{definition}[Kamo and Osuga {\cite{KO}}]\label{DefEb}
Let $b,h\in\omega^\omega$ such that $b\geq1$ and assume  $\lim_{i\to+\infty}\frac{h(i)}{b(i)}=0$.
Define the \textit{$(b,h)$-eventually different real forcing $\Ebb_{b}^{h}$} as the poset whose conditions are of the form $(s,F)$ and satisfy:
\begin{enumerate}[(i)]
  \item $s\in\Seq(b)$,
  \item $F\subseteq \Scal(b,h)$ is finite, and
  \item $|F|h(n)< b(n)$ for each $n\geq |s|.$
\end{enumerate}
It is ordered by $(t,F^{\prime})\leq(s,F)$ iff $s\subseteq t$, $F\subseteq F^{\prime}$ and $\forall{i\in|t|\smallsetminus|s|}\forall{\varphi\in F}(t(i)\not\in \varphi(i))$.

If $S \subseteq \Scal(b,h)$, define $\Ebb^h_{b}(S)=\{ (s,F) \in \Ebb^h_b : F \subseteq S \}$ with the same order as $\Ebb_{b}^{h}$. Denote $\Ebb_b:=\Ebb^1_b$ and $\Ebb_b(S):=\Ebb^1_b(S)$.
\end{definition}

If $G$ is $\Ebb_{b}^{h}(S)$-generic over $V$ and $r:=\bigcup\{s:\exists{F}((s,F)\in G)\}$, then $r\in \prod b$ and $\forall \varphi\in S(\varphi\not\ni^* r)$. Also, $V[G]=V[r]$. In particular, $\Ebb^h_b$ adds an  $r\in\prod b$ such that $\forall\varphi\in\Scal(b,h)\cap V(\varphi\not\ni^* r)$. Hence, when $h\geq^*1$, $r$ is eventually different from the members of $V\cap\prod b$.



The following result is a generalization of \cite[Lemma 9]{KO} about the linkedness of $\Ebb^h_b(S)$

\begin{lemma}\label{genlink}
   Let $b,h\in\omega^\omega$ with $b\geq 1$. Let $\pi,\rho\in\omega^\omega$ and assume that there is a non-decreasing function $f\in\omega^\omega$ going to infinity and an $m^*<\omega$  such that, for all but finitely many $k<\omega$,
   \begin{enumerate}[(i)]
       \item $k\pi(k)h(i)<b(i)$ for all $i\geq f(k)$ and
       \item $k\prod_{i=m^*}^{f(k)-1}((\min\{k,f(k)\}-1)h(i)+1)\leq\rho(k)$.
   \end{enumerate}
   Then, for any $S\subseteq\Scal(b,h)$, $\Ebb^h_b(S)$ is $(\rho,\pi)$-linked.
\end{lemma}
\begin{proof}
   Fix $M>m^*$ such that, for all $k\geq M$, (i) and (ii) holds and $f(k)>0$. Find a non-decreasing function $g:\omega\to\omega$ that goes to infinity such that, for all $k\geq M$, $g(k)<\min\{k,f(k)\}$ and $|\prod_{i<g(k)}b(i)|\leq k$. Choose $M'\geq M$ such that $g(k)\geq M$ for all $k\geq M'$. Put
   \[S_k:=\{s\in\prod_{i<f(k)}b(i):\forall i\in[g(k),f(k))(s(i)\leq g(k)h(i))\}\]
   when $k\geq M'$, otherwise put $S_k:=\emptyset$. Note that, for $k\geq M'$,
   \[
       |S_k|\leq k\prod_{i=g(k)}^{f(k)-1}(g(k)h(i)+1)\leq k\prod_{i=m^*}^{f(k)-1}((\min\{k,f(k)\}-1)h(i)+1)\leq\rho(k),
   \]
   so $|S_k|\leq\rho(k)$ for all $k<\omega$. For each $s\in S_k$ put
   \[Q_{k,s}:=\{(t,F)\in\Ebb_b^h(S):t=s,\forall i\geq f(k)(\pi(k)|F|h(i)<b(i))\}.\]
   Clearly, $Q_{k,s}$ is $\pi(k)$-linked.

   It remains to show that $\la Q_{k,s}:k<\omega,s\in S_k\ra$ satisfies (ii') of Definition \ref{link}. If $(t,F)\in\Ebb^h_b(S)$, choose $N$ such that $|t|+|F|+M'\leq g(N)$. We prove that, for all $k\geq N$, there is some $s\in S_k$ such that $(s,F)\leq(t,F)$ and $(s,F)\in Q_{k,s}$. Extend $t$ to $t'$ so that $|t'|=g(k)$ and $(t',F)$ is a condition in $\Ebb^h_b(S)$ stronger than $(t,F)$. Note that $|\bigcup_{\varphi\in F}\varphi(i)|\leq h(i)g(N)\leq h(i)g(k)$ for each $i\geq g(k)$, thus we can extend $t'$ to an $s\in S_k$ so that $(s,F)$ is a condition stronger than $(t',F)$. As $|F|\pi(k)h(i)\leq g(N)\pi(k)h(i)\leq k\pi(k)h(i)<b(i)$ by (i), $(s,F)\in Q_{k,s}$.
\end{proof}

\begin{Corol}\label{Ebsigmalinked}
   If $S\subseteq \Scal(b,h)$ and $\lim_{n\to+\infty}\frac{h(n)}{b(n)}=0$ then $\Ebb_{b}^{h}(S)$ is $\sigma$-linked. In particular, $\Ebb_{b}^{h}$ is a Suslin ccc poset.
\end{Corol}
\begin{proof}
   As $\lim_{n\to+\infty}\frac{h(n)}{b(n)}=0$, we can find an increasing function $f:\omega\to\omega$ such that $2kh(i)<b(i)$ for all $i\geq f(k)$ and $k<\omega$. By Lemma \ref{genlink}, $\Ebb^h_b(S)$ is $(\rho,2)$-linked where $\rho(k):=k\prod_{i=1}^{f(k)-1}((\min\{k,f(k)\}-1)h(i)+1)$. The conclusion is a consequence of Lemmas \ref{linksigmalink} and \ref{genlink}.
\end{proof}

\begin{Corol}[{\cite[Lemma 9]{KO}}]\label{linked}
Let $b,\pi,h \in \omega^{\omega}$ such that $\pi$ and $h$ are non-decreasing, $b\geq 1$, both $\pi$ and $h$ are $\geq^*1$ and $b\geq^{*}h\pi\mathrm{id}_{\omega}+1$. If $S\subseteq \Scal(b,h)$ then $\Ebb^h_{b}(S)$ is $((\mathrm{id}_{\omega}h)^{\mathrm{id}_{\omega}},\pi)$-linked. In particular, if $h$ is the constant function $1$ then $\Ebb_b(S)$ is $((\mathrm{id}_{\omega})^{\mathrm{id}_{\omega}},\pi)$-linked.
\end{Corol}
\begin{proof}
   Use $f=\id_\omega$ and $m^*=1$ in Lemma \ref{genlink}.
\end{proof}

To finish this section, we review from \cite{BrM} the poset that increases $\blc_{b,h}$.

\begin{definition}[Brendle and Mej\'{\i}a {\cite[Def. 4.1]{BrM}}]\label{DefLocposet}
Let $b,h\in \omega^\omega$ such that $b\geq1$. For $R\subseteq \prod b$, define the poset
\[\LOCbb_{b}^{h}(R):=\{(s,F):s\in \Scal_{<\omega}(b,h), F\subseteq R\text{\ and }  \forall i\geq|s|(|F| \leq h(i)) \}\] ordered by $(s',F')\leq (s,F)$ iff $s\subseteq s'$, $F\subseteq F'$ and $\forall{i\in[|s|,|s'|)}(\{x(i):x\in F\}\subseteq s'(i))$. Put $\LOCbb_{b}^{h}:=\LOCbb_{b}^{h}(\prod b)$, which is a Suslin ccc poset.
\end{definition}

\begin{lemma}[Brendle and Mej\'{\i}a {\cite[Lemma 4.2]{BrM}}]\label{Locgeneric}
Let $b,h\in \omega^\omega$ such that $b\geq1$ and let $R\subseteq\prod b$. If $h$ goes to infinity then $\LOCbb_{b}^{h}(R)$ is $\sigma$-linked and it generically adds a slalom in $\Scal(b,h)$ that localizes all reals in $R$. In particular, $\LOCbb_{b}^{h}$ generically adds a slalom in $\Scal(b,h)$ that localizes all the ground model reals in $\prod b$.
\end{lemma}

\begin{lemma}[Brendle and Mej\'{\i}a {\cite[Lemma 5.10]{BrM}}]\label{Loclinked}
Let $b,h,\pi,\rho \in\omega^\omega$ be non-decreasing functions with $b\geq 1$ and $h$ going to infinity. If $\{m_{k}\}_{k<\omega}$ is a non-decreasing sequence of natural numbers that goes to infinity and, for all but finitely many $k<\omega$, $k\cdot \pi(k)\leq h(m_k)$ and $k\cdot|[b(m_k-1)]^{\leq k}|^{m_k}\leq \rho(k)$, then $\LOCbb_{b}^{h}(R)$ is $(\rho,\pi)$-linked for any $R\subseteq \prod b$.
\end{lemma}

\section{ZFC results}\label{SecZFC}

In this section we prove and review some inequalities between the cardinal invariants associated with Yorioka ideals, the cardinals in Cicho\'n's diagram and localization and anti-localization cardinals. Figure \ref{diagram} at the end of this section, which is taken from \cite{K1}, illustrates some of these inequalities.

\subsection{Localization and anti-localization cardinals}\label{SubSecLoc}

For this subsection, fix a function $b$ with domain $\omega$ and $h\in\omega^\omega$ such that $1\leq^* h$, $b(i)\neq\emptyset$ for every $i<\omega$ and $\forall^\infty i<\omega(h(i)<|b(i)|)$. We show that the localization and anti-localization cardinals are characterized by any other known cardinal invariants when $b(i)$ is infinite for infinitely many $i<\omega$. We also show some results for these cardinals when $b\in\omega^\omega$, mostly when taking limit values like $\sup\{\dlc_{b,h}:b\in\omega^\omega\}$.

When $h$ does not go to infinity, the localization cardinals have simple characterizations.

\begin{theorem}\label{Lochnotinfty}
   If $h$ does not go to infinity then $\dlc_{b,h}\geq \cfrak$ and $\blc_{b,h}=N+1$ where $N$ is the minimum natural number such that $A_N:=\{i<\omega:h(i)=N\}$ is infinite. Moreover, $\dlc_{b,h}=\cfrak$ when $\forall^\infty i<\omega(|b(i)|\leq\cfrak)$.
\end{theorem}
\begin{proof}
   Without loss of generality, we assume that each $b(i)$ is a cardinal number. For each $l\leq N$ define $x_l\in\prod b(i)$ by $x_l(i):=l$ when $i\in A_N$ and $h(i)<b(i)$, or $0$ otherwise. Note that no single slalom in $\Scal(b,h)$ localizes every $x_l$ for $l\leq N$, so $\blc_{b,h}\leq N+1$. Conversely,
   by the definition of $N$, $N\leq^* h$, so $N<\blc_{b,h}$.

   On the other hand, \cite[Lemma 1.11]{GS93} states that $\dlc_{b,h}=\cfrak$ whenever $h$ is constant and $b\in\omega^\omega$. So, by Example \ref{1.46}(3) and (4), $\dlc_{b,h}\geq\cfrak$ holds in our case. In addition, if $\forall{i<\omega}(|b(i)|\leq\cfrak)$ then $\dlc_{b,h}\leq|\Scal(b,h)|=\big|\prod b\big|=\cfrak$, so the ``moreover" part follows by Example \ref{1.46}(3).
\end{proof}

A similar result to the above can be proved for the anti-localization cardinals when the sequence $\big\la\frac{h(i)}{|b(i)|}:i<\omega\big\ra$ does not converge to $0$ (put $\frac{h(i)}{|b(i)|}=0$ when $b(i)$ is an infinite set).

\begin{theorem}\label{aLctrivial}
  If the sequence $\big\la\frac{h(i)}{|b(i)|}:i<\omega\big\ra$ does not converge to $0$ then $\balc_{b,h}=N$ where $N$ is the minimum natural number such that $B_N:=\{i<\omega:|b(i)|\leq N\cdot h(i)\}$ is infinite, and $\dalc_{b,h}=\cfrak$.
\end{theorem}
\begin{proof}
  For all but finitely many $i\in B_N$ there is a partition $\la c_{i,j}:j<N\ra$ of $b(i)$ such that $0<|c_{i,j}|\leq h(i)$ for each $j<N$. Let $\iota:\omega\to B_N$ be the increasing enumeration of $B_N$. Define $\varphi_-:N^\omega\to\Scal(b,h)$ such that, for each $z\in N^\omega$, $\varphi_-(z)(\iota(i))=c_{\iota(i),z(i)}$ and $\varphi_-(z)(i)=\emptyset$ for all $i\notin B_N$; and define $\varphi_+:\prod b\to N^\omega$ such that $\varphi_+(x)(i)=j$ iff $x(\iota(i))\in c_{\iota(i),j}$. Note that the pair $(\varphi_-,\varphi_+)$ witnesses that $\Ed(N)\leqT\aLc(b,h)$, so $\balc_{b,h}\leq\balc_{N,1}$ and $\dalc_{N,1}\leq\dalc_{b,h}$. On the other hand, by Example \ref{ExTukeyorder}(2) and Theorem \ref{Lochnotinfty}, $\balc_{N,1}=\blc_{N,N-1}=N$ and $\dalc_{N,1}=\dlc_{N,N-1}=\cfrak$. Therefore $\balc_{b,h}\leq N$ and $\cfrak\leq\dalc_{b,h}$.

  The converse inequality for $\balc_{b,h}$ follows from the fact that $(N-1)\cdot h(i)<|b(i)|$ for all but finitely many $i<\omega$ (so $N-1<\balc_{b,h}$).

  Put $b'(i)=\min\{\omega,|b(i)|\}$. As $|b'(i)|\leq|b(i)|$ for all $i<\omega$, by Example \ref{ExTukeyorder}(3) $\dalc_{b,h}\leq\dalc_{b',h}$. Note that $\dalc_{b',h}\leq\big|\prod b'\big|\leq\cfrak$.
\end{proof}

\begin{lemma}\label{aLcincrinf}
  Assume that the sequence $\la|b(i)|:i<\omega\ra$ is non-decreasing and $|b(0)|\geq\aleph_0$. Then $\balc_{b,h}\geq\cof([\kappa]^{\aleph_0})$ and $\dalc_{b,h}\leq\add([\kappa]^{\aleph_0})$ where $\kappa=\sup_{i<\omega}\{|b(i)|\}$.
\end{lemma}

Recall that $\cof([\omega]^{\aleph_0})=1$ and that $\add([\omega]^{\aleph_0})$ is undefined. If we interpret $\add([\omega]^{\aleph_0})$ as ``something" that is above all the ordinals, the inequality $\dalc_{b,h}\leq\add([\omega]^{\aleph_0})$ makes sense. On the other hand, $\add([\kappa]^{\aleph_0})=\aleph_1$ when $\kappa$ is uncountable.

\begin{proof}
  Wlog assume that each $b(i)$ is an infinite cardinal. Define $b'(i):=b(i)^{<\omega}$ for each $i<\omega$. As $|b'(i)|=|b(i)|$ then $\aLc(b',h)\eqT\aLc(b,h)$, so we can work with $b'$ instead of $b$. Define $\varphi_-:\Scal(b',h)\to[\kappa]^{\aleph_0}$ such that $\varphi_-(S)$ contains $\{s(j):j<|s|,\ s\in S(i),\ i<\omega\}$, and define $\varphi_+:[\kappa]^{\aleph_0}\to\prod b'$ such that, for any $c\in[\kappa]^{\aleph_0}$, $\varphi_+(c)(i)=\bar{c}\frestr k_{c,i}$ where $\bar{c}:=\la c_j:j<\omega\ra$ is some (chosen) enumeration of $c$ and $k_{c,i}$ is the maximal $k\leq i$ such that $\bar{c}\frestr k\in b'(i)$. Note that $\la k_{c,i}:i<\omega\ra$ is a non-decreasing sequence that goes to infinity.

  It is enough to show that $(\varphi_-,\varphi_+)$ witnesses $\aLc(b,h)\leqT\la[\kappa]^{\aleph_0},[\kappa]^{\aleph_0},\nsupseteq\ra$, that is, for any $S\in\Scal(b,h)$ and $c\in[\kappa]^{\aleph_0}$, if $\varphi_-(S)\nsupseteq c$ then $S\not\ni^*\varphi_+(c)$. Choose $l\in c\smallsetminus\varphi_-(S)$. Hence $\varphi_+(c)(i)\notin S(i)$ for any $i\geq N$, where $N$ is some natural number such that $l=c_{j_0}$ for some $j_0< N$ and $\{c_j:j\leq j_0\}\subseteq b'(N)$.
\end{proof}

\begin{theorem}\label{lcalcinfb}
  \begin{enumerate}[(a)]
      \item If $b(i)$ is infinite for all but finitely many $i<\omega$ then
      \[\balc_{b,h}=\max\{\cof([\kappa]^{\aleph_0}),\non(\Mcal)\}\text{\ and }\dalc_{b,h}=\min\{\add([\kappa]^{\aleph_0}),\cov(\Mcal)\}\]
      where $\kappa=\liminf_{i<\omega}\{|b(i)|\}$. In particular, if $\kappa=\omega$ then $\balc_{b,h}=\non(\Mcal)$ and $\dalc_{b,h}=\cov(\Mcal)$; otherwise, if $\kappa$ is uncountable then $\dalc_{b,h}=\aleph_1$.
      \item If $b(i)$ is infinite for infinitely many $i<\omega$ then
      \[\blc_{b,h}=\min\{\add([\lambda]^{\aleph_0}),\blc_{\omega,h}\}\text{\ and }\dlc_{b,h}=\max\{\cof([\lambda]^{\aleph_0}),\dlc_{\omega,h}\}\]
      where $\lambda=\limsup_{i<\omega}\{|b(i)|\}$. In particular, if $\lambda=\omega$ then $\blc_{b,h}=\blc_{\omega,h}$ and $\dlc_{b,h}=\dlc_{\omega,h}$; if $\lambda$ is uncountable and $h$ goes to infinity then $\blc_{b,h}=\aleph_1$.
  \end{enumerate}
\end{theorem}
\begin{proof}
We first show (a). Wlog, we may assume that $b(i)$ is an infinite cardinal for all $i<\omega$. Find a non-decreasing function $b_0$ from $\omega$ into the infinite cardinals such that $b_0(i)\leq b(i)$ for all $i<\omega$, and $\sup_{i<\omega}\{b_0(i)\}=\kappa$. By Example \ref{ExTukeyorder}(3) and Theorem \ref{Barcharalc}, $\non(\Mcal)\leq\balc_{b_0,h}\leq\balc_{b,h}$ and $\dalc_{b,h}\leq\dalc_{b_0,h}\leq\cov(\Mcal)$. On the other hand, by Lemma \ref{aLcincrinf}, $\cof([\kappa]^{\aleph_0})\leq\balc_{b_0,h}\leq\balc_{b,h}$ and $\dalc_{b,h}\leq\add([\kappa]^{\aleph_0})$, so $\max\{\cof([\kappa]^{\aleph_0}),\non(\Mcal)\}\leq\balc_{b,h}$ and $\dalc_{b,h}\leq\min\{\add([\kappa]^{\aleph_0}),\cov(\Mcal)\}$.

Now we show that $\balc_{b,h}\leq\max\{\cof([\kappa]^{\aleph_0}),\balc_{\omega,h}\}$. Let $C\subseteq[\kappa]^{\aleph_0}$ be a witness of $\cof([\kappa]^{\aleph_0})$. For each $c\in C$ choose a witness $S_c\subseteq\Scal(b_c,h)$ of $\balc_{b_c,h}$ where $b_c(i):=b(i)\cap c$, and put $S:=\bigcup_{c\in C}S_c$. It is clear that $|S|\leq\max\{\cof([\kappa]^{\aleph_0}),\balc_{\omega,h}\}$, so it is enough to show that, for any $x\in\prod b$, there is some $\varphi\in S$ such that $x\in^\infty S$. As $\ran x$ is countable, there is some $c\in C$ such that $\ran x\subseteq c$. Hence $x\in\prod b_c$ and there is some $\varphi\in S_c$ such that $x\in^\infty\varphi$.

To see that $\min\{\add([\kappa]^{\aleph_0}),\cov(\Mcal)\}\leq\dalc_{b,h}$, assume that $Y\subseteq\prod b$ has size less than this minimum and show that there is some $\varphi\in\Scal(b,h)$ such that $y\in^\infty\varphi$ for every $y\in Y$. For each $y\in Y$ choose a $c_y\in[\kappa]^{\aleph_0}$ such that $y\in\prod b_{c_y}$. As $|Y|<\add([\kappa]^{\aleph_0})$, there is some $c^*\in[\kappa]^{\aleph_0}$ such that $\bigcup_{y\in Y}c_y\subseteq c^*$(\footnote{For this to happen, it is clear that either $\kappa=\omega$ or $Y$ is countable.}), so $Y\subseteq\prod b_{c^*}$. Hence, as $|Y|<\cov(\Mcal)=\dalc_{b_{c^*},h}$, there is some $\varphi\in\Scal(b_{c^*},h)$ as desired.

To finish, we show (b). As $b\leq^*\lambda$, by Example \ref{ExTukeyorder}(3) $\blc_{\lambda,h}\leq\blc_{b,h}$ and $\dlc_{b,h}\leq\dlc_{\lambda,h}$. A similar argument as in (a) guarantees that $\min\{\add([\lambda]^{\aleph_0}),\blc_{\omega,h}\}\leq\blc_{\lambda,h}$ and $\dlc_{\lambda,h}\leq\max\{\cof([\lambda]^{\aleph_0}),\dlc_{\omega,h}\}$.

Find an increasing function $\iota:\omega\to\omega$ such that $\la|b(\iota(i)|:i<\omega\ra$ is a non-decreasing sequence of infinite cardinals and $\sup_{i<\omega}\{|b(\iota(i))|\}=\lambda$. Put $b':=b\circ\iota$ and $h'=h\circ\iota$. By Example \ref{ExTukeyorder}(4), (5) and Lemma \ref{aLcincrinf}, $\blc_{b,h}\leq\blc_{b',h'}\leq\dalc_{b',h'}\leq\add([\lambda]^{\aleph_0})$ and $\cof([\lambda]^{\aleph_0})\leq\balc_{b',h'}\leq\dlc_{b',h'}\leq\dlc_{b,h}$, so it remains to show that $\blc_{b,h}\leq\blc_{\omega,h}$ and $\dlc_{\omega,h}\leq\dlc_{b,h}$. This is clear by Lemma \ref{Lochnotinfty} when $h$ does not go to infinity; if $h$ goes to infinity then so does $h'$ and, by Theorem \ref{Barcharloc} and Example \ref{ExTukeyorder}(3), $\blc_{b,h}\leq\blc_{b',h'}\leq\blc_{\omega,h'}=\blc_{\omega,h}$ and $\dlc_{\omega,h}=\dlc_{\omega,h'}\leq\dlc_{b',h'}\leq\dlc_{b,h}$.
\end{proof}

\begin{remark}\label{Remhirrelevant}
   In Theorem \ref{lcalcinfb} the role of $h$ is not too relevant. In (a) we actually have $\balc_{b,h}=\balc_{b,1}$ and $\dalc_{b,h}=\dalc_{b,1}$. In (b), when $h$ goes to infinity $\blc_{b,h}=\min\{\add([\lambda]^{\aleph_0}),\add(\Ncal)\}=\blc_{b,\id_\omega}$ and $\dlc_{b,h}=\max\{\cof([\lambda]^{\aleph_0}),\cof(\Ncal)\}=\dlc_{b,\id_\omega}$. When $h$ does not go to infinity $\dlc_{b,h}=\dlc_{b,1}$ (see Theorem \ref{Lochnotinfty} and Corollary \ref{lcalclargeb}(b)) and, although $\blc_{b,h}$ depends on $h$, it is a finite number already calculated in Theorem \ref{Lochnotinfty}.
\end{remark}

The previous result implies that $\balc_{\omega_n,1}=\max\{\aleph_n,\non(\Mcal)\}$, $\dlc_{\omega_n,\id_\omega}=\max\{\aleph_n,\cof(\Ncal)\}$, and $\dlc_{\omega_n,1}=\max\{\aleph_n,\cfrak\}$ for any $n<\omega$. For larger $\kappa$, the cardinal $\cof([\kappa]^{\aleph_0})$ is quite special, e.g., large cardinals are necessary to prove the consistency of $\cof([\kappa]^{\aleph_0})>\kappa$ for some $\kappa$ of uncountable cofinality (for more on this, see e.g. \cite{eisworth,rinot}).

When $b$ is above the continuum, $\balc_{b,h}$ and $\dlc_{b,h}$ are simpler to calculate.

\begin{Corol}\label{lcalclargeb}
   \begin{enumerate}[(a)]
       \item If $\forall^\infty i<\omega(|b(i)|\geq\cfrak)$ then $\balc_{b,h}=(\liminf_{i<\omega}\{|b(i)|\})^{\aleph_0}$.
       \item If $\exists^\infty i<\omega(|b(i)|\geq\cfrak)$ then $\dlc_{b,h}=(\limsup_{i<\omega}\{|b(i)|\})^{\aleph_0}$.
   \end{enumerate}
\end{Corol}
\begin{proof}
   It is a direct consequence of Theorem \ref{lcalcinfb} and the fact that $\cf([\kappa]^{\aleph_0})=\kappa^{\aleph_0}$ whenever $\kappa\geq\cfrak$. This is clear for $\kappa=\cfrak$. If $\kappa>\cfrak$, as $|\Pcal(c)|=\cfrak$ for any $c\in[\kappa]^{\aleph_0}$, no cofinal family in $[\kappa]^{\aleph_0}$ can have size $<\big|[\kappa]^{\aleph_0}\big|=\kappa^{\aleph_0}$.
\end{proof}

The cases that are not characterized in Theorem \ref{lcalcinfb} are when $b(i)$ is finite for all (but finitely many) $i<\omega$ for the localization cardinals, and when $b(i)$ is finite for infinitely many $i<\omega$ for the anti-localization cardinals. Even more, by the following result (for the case $\kappa=\omega$), the latter case is reduced to the case when $b(i)$ is finite for all $i<\omega$.

\begin{lemma}\label{aLocfin}
  Let $\kappa$ be an infinite cardinal. If the set $F:=\{i<\omega:|b(i)|<\kappa\}$ is infinite and $\iota:\omega\to F$ is the increasing enumeration of $F$, then $\balc_{b,h}=\balc_{b\circ\iota,h\circ\iota}$ and $\dalc_{b,h}=\dalc_{b\circ\iota,h\circ\iota}$.
\end{lemma}
\begin{proof}
  If $\omega\smallsetminus F$ is finite, it is clear that $\aLc(b,h)\eqT\aLc(b\circ\iota,h\circ\iota)$. Otherwise, when $\omega\smallsetminus F$ is infinite, $\aLc(b,h)\eqT\aLc(b\circ\iota,h\circ\iota)\wedge\aLc(b\circ\iota',h\circ\iota')$ where $\iota':\omega\to\omega\smallsetminus F$ is the increasing enumeration of $\omega\smallsetminus F$. On the other hand, 
  $\balc_{b\circ\iota,h\circ\iota}\leq\balc_{\kappa,h\circ\iota}=\balc_{\kappa,h\circ\iota'}\leq\balc_{b\circ\iota',h\circ\iota'}$
  by Example \ref{1.46}(3) and Theorem \ref{lcalcinfb}. In a similar way, $\dalc_{b\circ\iota',h\circ\iota'}\leq\dalc_{b\circ\iota,h\circ\iota}$ Hence, the result follows by Theorem \ref{seqcompinv}.
\end{proof}

\newcommand{\minLc}{\mathrm{minLc}}
\newcommand{\supLc}{\mathrm{supLc}}
\newcommand{\supaLc}{\mathrm{supaLc}}
\newcommand{\minaLc}{\mathrm{minaLc}}

Now, we look at limits of localization and anti-localization cardinals.

\begin{definition}\label{DefminLc}
Define the following cardinal characteristics.
   \[\begin{array}{rlrl}
      \minLc:= & \min\{\blc_{b,\id_\omega}:b\in\omega^\omega\}, &
      \supLc:= & \sup\{\dlc_{b,\id_\omega}:b\in\omega^\omega\},\\
      \supaLc:= & \sup\{\balc_{b,1}:b\in\omega^\omega\}, &
      \minaLc:= & \min\{\dalc_{b,1}:b\in\omega^\omega\}.
   \end{array}\]
\end{definition}

\begin{lemma}\label{Lcsup}
   \begin{enumerate}[(a)]
       \item $\min\{\blc_{b,h}:b\in\omega^\omega\}=\minLc$ and $\sup\{\dlc_{b,h}:b\in\omega^\omega\}=\supLc$ when $h$ goes to infinity.
       \item $\sup\{\balc_{b,h}:b\in\omega^\omega\}=\supaLc$ and $\min\{\dalc_{b,h}:b\in\omega^\omega\}=\minaLc$.
   \end{enumerate}
\end{lemma}
\begin{proof}
   Denote by $\minLc_h:= \min\{\blc_{b,h}:b\in\omega^\omega\}$, $\supLc_h:=  \sup\{\dlc_{b,h}:b\in\omega^\omega\}$, $\supaLc_h:= \sup\{\balc_{b,h}:b\in\omega^\omega\}$, and $\minaLc_h:=  \min\{\dalc_{b,h}:b\in\omega^\omega\}$. To see (a) it is enough to show the following.
   \begin{Claim}
      For any $b,h'\in\omega^\omega$ such that $h\leq h'$ there is some $b'\in\omega^\omega$ such that $\Lc(b,h)\leqT\Lc(b',h')$
   \end{Claim}
   \begin{proof}
      Choose an interval partition $\la I_n:-1\leq n<\omega\ra$ of $\omega$ such that $h(k)\geq h'(n)$ for all $k\in I_n$ (denote $h'(-1):=0$). Put $b'(n):=\prod_{k\in I_n}b(k)$. Define $\varphi_-:\prod b\to\prod b'$ by $\varphi_-(x):=\la x\frestr I_n:0\leq n<\omega\ra$, and define $\varphi_+:\Scal(b',h')\to\Scal(b,h)$ by $\varphi_+(S)(k)=\{s(k):s\in S(n)\}$ whenever $k\in I_n$ (put $S(-1):=\emptyset$). It is easy to show that $(\varphi_-,\varphi_+)$ is the required Tukey connection.
   \end{proof}
   Assume $h\leq h'$. The claim implies that $\minLc_{h'}\leq\blc_{b',h'}\leq\blc_{b,h}$ and $\dlc_{b,h}\leq\dlc_{b',h'}\leq\supLc_{h'}$ for any $b\in\omega$. On the other hand, $\minLc_h\leq\minLc_{h'}$ and $\supLc_{h'}\leq\supLc_h$ by Example \ref{ExTukeyorder}(3), so equality holds. Therefore, (a) follows by using an $h'$ above both $\id_\omega$ and $h$.

   Concerning item (b), we have

   \begin{Claim}
      For any $b,h\in\omega^\omega$ with $h\geq^* 1$ there is some $b'\in\omega^\omega$ such that $\aLc(b',h)\leqT\aLc(b,1)$.
   \end{Claim}
   \begin{proof}
      Let $\la I_n:n<\omega\ra$ be the interval partition of $\omega$ such that $|I_n|=h(n)$ for all $n<\omega$. Put $b'(n):=\prod_{k\in I_n}b(k)$. Define $\varphi_-:\Scal(b',h)\to\prod b$ such that, for any $S\in\Scal(b',h)$, $\varphi_-(S)$ satisfies that $\forall n<\omega\forall t\in S(n)\exists k\in I_n(\varphi_-(S)(k)=t(k))$ (which is fine because $|S(n)|\leq|I_n|$). On the other hand, define $\varphi_+:\prod b\to\prod b'$ by $\varphi_+(y):=\la y\frestr I_n:n<\omega\ra$. It is clear that $(\varphi_-,\varphi_+)$ is the Tukey connection we want.
   \end{proof}
   As in the proof of (a), the claim above can be used to prove (b).
\end{proof}

Other cardinals like $\sup\{\blc_{b,h}:b\in\omega^\omega\}$ are in principle not that interesting, for instance, this supremum above would be $\blc_{1,h}$, which is undefined, or at least if $b$ is restricted to be above $h$, then it would be $\blc_{h+1,h}$. Also, when $h$ does not go to infinity, $\minLc_h$ and $\supLc_h$ are easily characterized by Theorem \ref{Lochnotinfty}.

The following is another characterization of $\add(\Ncal)$ and $\cof(\Ncal)$ in terms of Localization cardinals for $b\in\omega^\omega$.

\begin{lemma}\label{bLoc(b,h)addN}
   $\add(\Ncal)=\min\{\bfrak,\minLc\}$ and $\cof(\Ncal)=\max\{\dfrak,\supLc\}$.
\end{lemma}
\begin{proof}
    In this proof we use the characterization of $\add(\Ncal)$ and $\cof(\Ncal)$ given in Theorem \ref{1.6}

    Assume that $F\subseteq\omega^\omega$ and $|F|<\min\{\bfrak,\minLc\}$. Therefore, there is some $d\in\omega^\omega$ such that, for every $x\in F$, $\forall^\infty i<\omega(x(i)<d(i))$. On the other hand, as $|F|<\blc_{d,\id_\omega}$, we can find an slalom in $\Scal(d,\id_\omega)\subseteq\Scal(\omega,\id_\omega)$ that localizes all the reals in $F$ (just use a family $F'\subseteq\prod d$ of the same size as $F$ such that each member of $F'$ is a finite modification of some member of $F$ and viceversa). Therefore, $\add(\Ncal)\geq\min\{\bfrak,\minLc\})$.

    Now we prove $\cof(\Ncal)\leq\max\{\dfrak,\supLc\}$. Choose a dominating family $D\subseteq\omega^\omega$ of size $\dfrak$ and, for each $d\in D$, choose  a family of slaloms $S_d\subseteq\Scal(d,\id_\omega)$ that witnesses $\dlc_{d,\id_\omega}$. Note that $S:=\bigcup_{d\in D}S_d\subseteq\Scal(\omega,\id_\omega)$ has size $\leq\max\{\dfrak,\supLc\}$ and that every real in $\omega^\omega$ is localized by some slalom in $S$, so $\cof(\Ncal)\leq|S|$.

    It is easy to see that $\Dbf\leqT\Lc(\omega,\id_\omega)$, so
    $\add(\Ncal)\leq\bfrak$ and $\dfrak\leq\cof(\Ncal)$. On the other hand, as $\add(\Ncal)\leq\blc_{b,\id_\omega}$ and $\dlc_{b,\id_\omega}\leq\cof(\Ncal)$ for any $b\in\omega^\omega$ (see Example \ref{1.46}(3)), the converse inequalities follow.
\end{proof}

The version of this lemma for the anti-localization cardinals is Miller's \cite{Mi} known result $\add(\Mcal)=\min\{\bfrak,\minaLc\}$ and its dual $\cof(\Mcal)=\max\{\dfrak,\supaLc\}$, which is proved in Theorem \ref{2.14}. Miller also proved that $\non(\SNcal)=\minaLc$. At the end of Section \ref{SecMain} we explain why no inequality between each pair of these cardinals can be proved in ZFC .


\subsection{Additivity and cofinality of Yorioka ideals}

Though Corollary \ref{YorioaddN} and Theorem \ref{2.6} are stated in \cite{K1}, their proofs seem not to appear in any existing reference. We offer original proofs of both results, even more, we provide the following general result that gives us Corollary \ref{YorioaddN} as a direct consequence. This is also used in Section \ref{4.2} to prove that, consistently, $\add(\Ncal)<\add(\Ical_f)<\cof(\Ical_f)<\cof(\Ncal)$.



\begin{theorem}\label{2.4.1}
    Let $f\in\omega^\omega$ increasing, $b:\omega\to\omega+1\smallsetminus\{0\}$ and $h\in\omega^\omega$. If $2^f\ll b$ and $\exists l^*<\omega(h\leq^*\id_\omega^{l^*})$ then $\min\{\bfrak,\blc_{b,h}\} \leq \add(\mathcal{I}_{f})$  and $\cof(\mathcal{I}_{f}) \leq \max\{\dfrak,\dlc_{b,h}\}$.
\end{theorem}
\begin{proof}
Note that $\exists l^*<\omega(h\leq^*\id_\omega^{l^*})$ is equivalent to $\exists l^*<\omega\forall^\infty n(\sum_{k\leq n}h(k)\leq n^{l^*})$. We use the latter assertion in this proof.

Fix a bijection $r^*:\omega\to2^{<\omega}$ such that, for any $i,j<\omega$, $i\leq j$ implies $|r^*_i|\leq|r^*_j|$. By Theorems \ref{1.6} and \ref{1.48}, it is enough to show that  $\langle \Ical_{f}, \Ical_{f}, \subseteq\rangle\leqT(\langle\omega^{\uparrow\omega},\omega^{\uparrow\omega},\leq^{*}\rangle;\Lc(b,h))$.

Fix $X\in \Ical_{f}$. There exists a $\sigma\in(2^{<\omega})^{\omega}$ with $h_{\sigma}\gg f$ such that  $X\subseteq [\sigma]_\infty$. Define $h_0\in\omega^\omega$ such that $h_{0}(i):=\min \{\lfloor\log_{2}b(i)\rfloor,h_{\sigma}(i)\}$. Note that $h_{0}\gg f+1$ because $2^f\ll b$. Hence, by Lemma \ref{lemosuga}, there is a  $d_{X}\in\omega^{\uparrow\omega}$ such that $g^f_{d_{X}}+1\leq^* h_{0}$. Define
\[\sigma_{X}(i):=\left\{\begin{array}{ll}
    \sigma(i)\frestr g^f_{d_{X}}(i) & \text{if $g^f_{d_{X}}(i)+1\leq h_0(i)$,} \\
    \la\ \ra & \text{otherwise.}
\end{array}\right.\]
Clearly $[\sigma]_\infty\subseteq [\sigma_{X}]_\infty$ and, for some $N_X<\omega$, $g^f_{d_{X}}(i)+1\leq \lfloor\log_{2}b(i)\rfloor$ for all $i\geq N_X$, so $2^{g^f_{d_X}(i)+1}\leq b(i)$. Therefore $2^{\leq g^f_{d_X}(i)}\subseteq\{r^*_j:j<b(i)\}$ for any $i\geq N_X$. Now, define $F_{X}:\omega^{\uparrow \omega} \to  \prod b$ by
\[F_X(e)(i):=\left\{\begin{array}{ll}
 j & \text{if $g^f_{e}(i) \leq  g^f_{d_X}(i)$, $i\geq N_X$, and $r^*_j:=\sigma_{X}(i)\frestr g^f_{e}(i)$,}\\
0 & \text{otherwise.}
\end{array}\right.\]
Note that, if $e\geq^* d_X$, then $[r^{**}_{F_X(e)}]_\infty\in \Jcal_{g_{e}^{f}}\subseteq\Ical_f$ where $r^{**}_{F_X(e)}:=\la r^*_{F_X(e)(i)}:i<\omega\ra$. Besides $[\sigma_X]_\infty\subseteq[r^{**}_{F_X(e)}]_\infty$. Define $\varphi_{-}:\Ical_{f} \to  \omega^{\uparrow \omega}\times (\prod b)^{\omega^{\uparrow\omega}}$ by $\varphi_-(X):=(d_{X},F_{X})$.

Fix $e\in \omega^{\uparrow\omega}$ and $S\in\Scal(b,h)$. Put $S_e(n)=\{j\in S(n):|r^*_j|=g_{e}^{f}(n)\}$ for each $n<\omega$. If $\exists^{\infty}{n<\omega}(S_e(n)\neq \emptyset)$  find the interval partition $\langle I^{S_e}_n : n<\omega\rangle$ of $\omega$ such that $|I^{S_e}_n|=|S_e(n)|$ and enumerate $S_e(n)=\{j^{S_e}_{i}:i \in I^{S_e}_n\}$. Define $\tau_{e,S} \in (2^{<\omega})^\omega$ by $\tau_{e,S}(i):=r^*_{j^{S_e}_{i}}$ when $i\in I^{S_e}_{n}$. We show that $[\tau_{e,S}]_\infty\in\Ical_f$, that is, $h_{\tau_{e,S}}\gg f$. Choose $l^*<\omega$ such that $\forall^\infty n(\sum_{k\leq n}h(k)\leq n^{l^*})$ and fix $k<\omega$. Note that $\forall n\geq e(l^*)(f(n^{l^*})\leq g_{e}^{f}(n))$. For $i \in I^{S_e}_n$ with $n\geq e(l^*)$ large enough, $i<\sum_{k\leq n}|S_e(k)|\leq \sum_{k\leq n}h(k)\leq n^{l^*}$. Hence, for all $i\in I^{S_e}_n$ with $n\geq e(k\cdot l^*)$ large enough, $h_{\tau_{e,S}}(i)=|\tau_{e,S}(i)|=g^{f}_{e}(n)\geq f(n^{k\cdot l^*})\geq f(i^k)$.

Define $\varphi_{+}:\omega^{\uparrow\omega}\times \prod b\to \Ical_{f}$ by $\varphi_+(e,S):=[\tau_{e,S}]_\infty$ if $\exists^{\infty}{n<\omega}(S_e(n)\neq \emptyset)$, or $\varphi_+(e,S):=\emptyset$ otherwise. It remains to show that, if $X\in\Ical_f$, $e\in\omega^{\uparrow\omega}$, $S\in\Scal(b,h)$, $d_{X}\leq^{*}e$, and $F_X(e)\in^{*}S$, then $X\subseteq [\tau_{e,S}]_\infty$. For large enough $n<\omega$, as $F_X(e)(n)\in S(n)$,  $F_X(e)(n)=j^{S_e}_{i_n}$ for some $i_n\in I_n^{S_e}$ by the definition of $S_e(n)$. This implies that $[r^{**}_{F_X(e)}]_\infty\subseteq[\tau_{e,s}]_\infty$. On the other hand, $d_X\leq^* e$ implies that $X\subseteq[\sigma_X]_\infty\subseteq[r^{**}_{F_X(e)}]_\infty$, so $X\subseteq[\tau_{e,s}]_\infty$.
\end{proof}

\begin{Corol}\label{YorioaddN}
   If $f\in\omega^\omega$ is increasing then $\add(\Ncal)\leq\add(\Ical_f)$ and $\cof(\Ical_f)\leq\cof(\Ncal)$.
\end{Corol}
\begin{proof}
   Apply Theorem \ref{2.4.1} with $b=\omega$ and $h=\id_\omega$.
\end{proof}

In the previous result we did not use the fact that $\Ical_f$ is closed under unions, so it implies that $\Ical_f$ is a $\sigma$-ideal (see Theorem \ref{Yoriosigmaid}).

\begin{theorem}[Kamo and Osuga \cite{kamo-osuga}]\label{2.5}
$\add(\mathcal{I}_f)\leq \bfrak$ and $\dfrak\leq \cof(\mathcal{I}_f)$. Even more, $\langle\omega^\omega,\omega^\omega,\leq^*\rangle\preceq_{\mathrm{T}}\langle\Ical_{f},\Ical_{f},\subseteq\rangle.$
\end{theorem}


\begin{theorem}\label{2.6}
Let $f,f^{\prime}\in \omega^\omega$ be increasing. If $\forall^{\infty}{n<\omega}(f(n+1)-f(n)\leq f^{\prime}(n+1)-f^{\prime}(n))$ then $\add(\Ical_{f})\geq \add(\Ical_{f^{\prime}})$ and $\cof(\Ical_{f})\leq\cof(\Ical_{f^{\prime}})$. In particular  $\add(\Ical_{f})\leq \add(\Ical_{\mathrm{id}_{\omega}})$ and $\cof(\Ical_{\mathrm{id}_{\omega}})\leq\cof(\Ical_f)$.
\end{theorem}
\begin{proof}
To fix some notation, for each $h\in\omega^\omega$ define $\Delta h(0):= h(0)$ and $\Delta h(n+1):=h(n+1)-h(n)$ for all $n<\omega$.
By Lemma \ref{f+c=f}, we can assume wlog $f(0)\leq f^{\prime}(0)$ and $\forall{n<\omega}(f(n+1)-f(n)\leq f'(n+1)-f'(n))$, that is, $\Delta f(n)\leq\Delta f'(n)$ for all $n<\omega$. Clearly, $m<n$ implies $f(n)-f(m)\leq f'(n)-f'(m)$ (it can be easily proved by induction on $n$).
As a result $f\leq f^{'}$, $\Delta g_{d}^{f}(n)\leq\Delta g_{d}^{f'}(n)$ for all $n<\omega$ and $g_d^f\leq g_d^{f'}$ for any $d\in\omega^{\uparrow\omega}$.

By Theorems \ref{1.48} and \ref{2.5}, it is enough to show $\la\Ical_{f},\Ical_{f},\subseteq\ra\leqT(\mathbf{D}';
\Bbf)$ where $\mathbf{D}':=\langle\omega^{\uparrow\omega},\omega^{\uparrow\omega},\leq^{*}\rangle$ and $\Bbf:=\langle\Ical_{f'},\Ical_{f'},\subseteq\rangle$. 

Fix $X\in\Ical_{f}$. There are $d_{X}\in\omega^{\uparrow\omega}$ and $\sigma_{X}\in(2^{<\omega})^{\omega}$ such that $X\subseteq[\sigma_{X}]_\infty$ and $h_{\sigma_{X}}=g^{f}_{d_{X}}$. For each $e\in\omega^{\uparrow\omega}$ and $n<\omega$, define $\tau_{X,e}(n)\in2^{g_e^{f'}(n)}$ according to the following cases: if $g_e^f(n)\leq d^f_{X}(n)$, put (\footnote{Intuitively, $\tau_{X,e}(n)$ is the sequence formed by inserting blocks of $0$'s inside $\sigma_X(n)$ in the following way: for each $k\leq n$, insert a block of $0$'s of length $\Delta g^{f'}_e(k)-\Delta g^{f}_e(k)$ between $\sigma_X(n)(g^f_e(k)-1)$ and $\sigma_X(n)(g^f_e(k))$ (when $k=0$ and $g^f_e(0)=0$, just insert the corresponding block behind $\sigma_X(n)(0)$, and when $k=n$, just insert the block after $\sigma_X(n)(g^f_e(n)-1)$).})
\[\tau_{X,e}(n)(i):=\left\{\begin{array}{ll}
            \sigma_X(n)(i-g^{f'}_e(m-1)+g^f_e(m-1)) & \text{if $0\leq i-g^{f'}_e(m-1)<\Delta g^f_e(m)$}\\
             & \text{for some $m\leq n$,}\\
            0 & \text{otherwise}
\end{array}\right.\]
(consider $g^{f}_{e}(-1)=g^{f'}_{e}(-1):=0$ and also note that such $m$ is unique with respect to $i$ because $g^{f'}_e(m-1)\leq i <\Delta g^f_e(m)+g^{f'}_e(m-1)\leq g^{f'}_e(m)$),
otherwise $\tau_{X,e}(n)(i)$ is just some fixed member of $2^{g_{e}^{f'}(n)}$. Note that $h_{\tau_{X,e}}=g_{e}^{f'}$, so $[\tau_{X,e}]_\infty\in\Jcal_{g_{e}^{f'}}\subseteq \Ical_{f'}$. Define $F_X:\omega^{\uparrow\omega}\to\Ical_{f'}$ by $F_{X}(e):=[\tau_{X,e}]_\infty$ and
$\varphi_{-}:\Ical_{f} \to \omega^{\uparrow \omega}\times (\Ical_{f'})^{\omega^{\uparrow\omega}}$ by $\varphi_-(X):=(d_{X},F_{X})$.

Fix $e\in\omega^{\uparrow\omega}$ and $Y\in\Ical_{f'}$. There are $e^{*}\geq e$ in $\omega^{\uparrow\omega}$ and $\tau'_{Y,e}\in (2^{<\omega})^{\omega}$ with $h_{\tau'_{Y,e}}=g_{e^{*}}^{f'}$ such that $Y\subseteq[\tau'_{Y,e}]_\infty$. Define $\rho_{Y,e}:\omega\to2^{<\omega}$ by $\rho_{Y,e}(n)(i)=\tau'_{Y,e}(n)(i+g^{f'}_e(m-1)-g^f_e(m-1))$ when $g^f_e(m-1)\leq i<g^{f}_e(m)$ (\footnote{Intuitively, $\rho_{Y,e}(n)$ is the sequence that results by cutting from $\tau'_{Y,e}(n)$ the blocks that are indexed by $[g^{f'}_e(k-1)+\Delta g^f_e(k),g^{f'}_e(k))$ (for $k<\omega$).}). Note that
\[|\rho_{Y,e}(n)|=
\begin{cases}
g^{f^{}}_{e}(m)  & \text{if $g^{f^{\prime}}_{e^{*}}(n)\geq g^{f^{\prime}}_{e}(m-1)+\Delta g^{f^{}}_{e}(m)$}\\
g^{f}_{e}(m-1)+g^{f'}_{e^{*}}(n)-g^{f'}_{e}(m-1) & \text{otherwise},
\end{cases}\]
where $m$ is the minimum natural number such that $g^{f'}_{e^{*}}(n)\leq g^{f'}_{e}(m)$. Clearly, $m\leqslant n$, so $|\rho_{Y,e}(n)|\leq g_e^f(n)$.

\begin{Claim}\label{11}
$[\rho_{Y,e}]_\infty\in\Ical_f$.
\end{Claim}
\begin{proof}
It is enough to show that $h_{\rho_{Y,e}}\gg f$.

Let $k_{0}=-1$ and $k_{m+1}=\min\{k<\omega: \forall n>k(g^{f'}_{e^{*}}(n)>g^{f'}_{e}(m))\}$. Note that $\{k_{m}\}_{m<\omega}$ is a monotone increasing sequence that goes to $+\infty$. Put $I_{m}=(k_{m},k_{m+1}]$. Note that, for $n\in I_{m}$, $m=\min\{k<\omega:g^{f'}_{e^{*}}(n)\leq g^{f'}_{e}(k)\}$ so $|\rho_{Y,e}(n)|\geq g_{e}^{f}(m-1)$.

Fix $c>0$. Find $N>e^{*}(3c)$ such that, for all $m\geq N$, $k_{m}>e^{*}(3c)$ and $(m-1)^{2}\geq m$.

\begin{Subclaim}
If $m\geq N^{3c}$ and $m\in [e(k),e(k+1))$ then $k_{m+1}\leq m^{\frac{k+10}{3c} }$.
\end{Subclaim}
\begin{proof} If $n>m^{\frac{k+10}{3c}}$ then $n^{3c}>m^{k+10}$, so $g^{f'}_{e^{*}}(n)\geq f'(n^{3c})> f'(m^{k+10})=g^{f'}_{e}(m)$ because $e^{*}(3c)<N\leq N^{k+10}\leq m^{(\frac{k+10}{3c})}< n$.
\end{proof}

Fix $n>k_{N^{3c}}$, so $n\in I_{m}$ for some $m\geq N^{3c}$. By the subclaim, $n \leq k_{m+1}\leq m^{\frac{k+10}{3c}}$ where $k$ satisfies $m\in [e(k),e(k+1))$. Thus $h_{\rho_{Y,e}}(n)\geq g_{e}^{f}(m-1)\geq f((m-1)^{k+9})\geq f(n^{c})$ since $(m-1)^{ k+9}\geq m^{\frac{ k+9}{2}}$ and $\frac{ k+9}{2}\geq \frac{ k+10}{3}$, so $(m-1)^{ k+9}\geq m^{\frac{ k+10}{3}}= m^{(\frac{ k+10}{3c})c}\geq n^{c}$.
\end{proof}

Define $\varphi_{+}:\omega^{\uparrow\omega}\times \Ical_{f'}\to \Ical_{f}$ by $\varphi_+(e,Y)=[\rho_{Y,e}]_\infty$.

To finish the proof we need to check that, for $X\in \Ical_f$, $Y\in\Ical_{f'}$ and $e\in\omega^{\uparrow\omega}$, if $d_{X}\leq^{*}e$ and $F_{X}(e)\subseteq Y$ then $X\subseteq[\rho_{Y,e}]_\infty$. Fix $x\in X\subseteq[\sigma_X]_\infty$, i.e. $\exists^{\infty}{n<\omega}(\sigma_{X}(n)\subseteq x)$. Define $x'\in2^\omega$ by (\footnote{Intuitively, $x'$ results by inserting (infinitely many) blocks of $0$'s inside $x$ like in the definition of $\tau_{X,e}$.})
\[x'(i):=\left\{\begin{array}{ll}
            x(i-g^{f'}_e(m-1)+g^f_e(m-1)) & \text{if $0\leq i-g^{f'}_e(m-1)<\Delta g^f_e(m)$,}\\
            0 & \text{otherwise.}
\end{array}\right.\]
Clearly, $x'\in[\tau_{X,e}]_\infty=F_X(e)\subseteq Y\subseteq[\tau'_{Y,e}]_\infty$. As $\tau'_{Y,e}(n)\subseteq x'$ implies $\rho_{Y,e}(n)\subseteq x$, we conclude that $x\in[\rho_{Y,e}]_\infty$.
\end{proof}


We also look at the following cardinal invariants related to Yorioka ideals:
\[\begin{array}{rcl}
    \minadd&= & \min\{\add(\mathcal{I}_{f}):f \in \omega^\omega\textrm{\ increasing}\},\\
    \supcov&= & \sup\{\cov(\mathcal{I}_{f}):f \in \omega^\omega\textrm{\ increasing}\},\\
    \minnon&= & \min\{\non(\mathcal{I}_{f}):f \in \omega^\omega\textrm{\ increasing}\}\textrm{, and}\\
    \supcof&= & \sup\{\cof(\mathcal{I}_{f}):f \in \omega^\omega\textrm{\ increasing}\}.
\end{array}\]

It is not necessary to refer to $\supadd$, $\mincov$, $\mathrm{supnon}$ and $\mincof$ as they are $\add(\Ical_{\id_\omega})$, $\cov(\Ical_{\id_{\omega}})$, $\non(\Ical_{\id_{\omega}})$ and $\cof(\Ical_{\id_{\omega}})$, respectively. This follows from Theorem \ref{2.6} and the fact that $\la2^\omega,\Ical_f,\in\ra\leqT\la2^\omega,\Ical_{f'},\in\ra$ when $f\leq^* f'$. Even more, $\SNcal=\bigcap\{\Ical_f:f\in\omega^\omega\ \textrm{increasing}\}$ implies $\minadd\leq\add(\SNcal)$, $\supcov\leq\cov(\SNcal)$, $\minnon=\non(\SNcal)$ and $\cof(\SNcal)\leq(\supcof)^\dfrak=2^{\dfrak}$ (see \cite{K1}). It is already known from \cite{Mi} that $\non(\SNcal)=\minaLc$, so $\minnon=\non(\SNcal)$ also follows from Theorem \ref{2.9}.

\subsection{Covering and uniformity of Yorioka ideals}\label{SubSeccovY}

The following results shows a relationship between the cardinals of the relational systems of the form $\la2^\omega,\Jcal_g,\in\ra$ and $\aLc(b,h)$.

\begin{lemma}[Kamo and Osuga \cite{KO}]\label{2.7}
Let $b \in \omega^\omega$ with $b\geq^* 2$. If $g\in\omega^\omega$ and $g \geq^* (\log b)^{+}$, then $\balc_{b,1}\leq\cov(\Jcal_{g})$ and $\non(\Jcal_{g})\leq\dalc_{b,1}$. Moreover, $\aLc(b,1)^\perp\leqT\langle 2^\omega, \mathcal{J}_g,\in \rangle$.
\end{lemma}

\begin{lemma}[Kamo and Osuga \cite{KO}]\label{2.8} Let $ h,b\in \omega^\omega$ and $g\in\omega^\omega$ monotone increasing. If $1\leq^* h\leq^* b$ and $b\geq^* 2^{g\circ(h^{+}-1)}$ then $\langle 2^\omega, \mathcal{J}_g,\in \rangle\leqT\aLc(b,h)^\perp$.
In particular, $\cov(\mathcal{J}_g)\leq\balc_{b,h}$ and $\dalc_{b,h}\leq \non(\mathcal{J}_g)$.
\end{lemma}

As a consequence. the cardinals $\minnon$ and $\supcov$ can be characterized as follows.

\begin{theorem}\label{2.9}
$\supcov=\supaLc$ and $\minnon=\minaLc$.
\end{theorem}
\begin{proof}
We first prove that, for all $b\in \omega^\omega$, $\balc_{b,1}\leqslant\supcov$ and $\minnon\leq \dalc_{b,1}$.
Wlog $b\geq 2$ (By Example \ref{1.46}(3)). Put $f=(\log b)^{+}$. As $\Jcal_{f}\supseteq\Ical_{f}$, $\cov(\Jcal_{f})\leqslant \cov(\Ical_{f})$ and $\non(\Ical_{f})\leqslant \non(\Jcal_{f})$. By Lemma \ref{2.7}, $\balc_{b,1}\leq \cov(\Jcal_{f})$ and $\non(\Jcal_{f})\leq \dalc_{b,1}$.

To prove the converse inequalities, assume that $f\in \omega^{\omega}$ is increasing. Choose some $g\gg f$ monotone increasing, so  $\Jcal_{g}\subseteq \Ical_{f}$. Hence $\cov(\Ical_{f})\leqslant \cov(\Jcal_{g})$ and $\non(\Jcal_{g})\leq\non(\Ical_{f})$. Put $b=2^{g}$. By Lemma \ref{2.8}, $\cov(\Jcal_{g})\leq \balc_{b,1}$ and $\dalc_{b,1}\leq\non(\Jcal_{g})$, so $\cov(\Ical_{f})\leq\balc_{b,1}\leq\supaLc$ and $\minaLc\leq\dalc_{b,1}\leq\non(\Ical_f)$.
\end{proof}

On the other hand, by Example \ref{1.46}(3) and the previous result, we have

\begin{Corol}\label{2.11}
 $\supcov\leq \non(\Mcal)$ and $\cov(\Mcal)\leq\minnon$. In particular $\cov(\Ical_f)\leq \non(\Mcal)$ and $\cov(\Mcal)\leq \non(\mathcal{I}_f)$ for every increasing $f\in\omega^\omega$.
\end{Corol}

We know that $\add(\mathcal{M})= \min\{\bfrak,\cov(\mathcal{M}) \}$ and $\cof(\mathcal{M})= \max\{\dfrak,\non(\mathcal{M}) \}$, but these equalities can be refined as in the following two results. These yield a version of Lemma \ref{bLoc(b,h)addN} for the anti-localization cardinals.



\begin{theorem}[Miller {\cite{Mi}}]\label{2.12}
   $\add(\mathcal{M})= \min\{\bfrak,\minnon\}$.
\end{theorem}




\begin{theorem}\label{2.14}
   $\cof(\mathcal{M})=\max\{\dfrak,\supcov \}$
\end{theorem}
\begin{proof}
  The inequality $\geq$  follows from Corollary \ref{2.11}.

  We prove $\leq$. Let $D\subseteq \omega^{\uparrow\omega}$ be a dominating family of size $\dfrak$. By Theorem \ref{2.9}, $\supcov=\supaLc$. For each $d \in D$ let $E_{d}\subseteq \prod d$ be a witness of $\balc_{d,1}$. Then $E=\bigcup_{d \in D}E_d$ satisfies $\forall{x\in\omega^{\omega}}\exists y\in E(x=^{\infty}y)$ 
  so, by Theorem \ref{1.7}, $\non(\mathcal{M})\leq|E|\leq \dfrak\cdot\supcov$. Hence $\cof(\mathcal{M})=\max\{\dfrak,\non(\mathcal{M}) \}\leq\max\{\dfrak,\supcov \}$.
\end{proof}

As a consequence, by Theorems \ref{2.5}, \ref{2.12} and \ref{2.14},

\begin{Corol}
    $\minadd\leq \add(\mathcal{M})$ and $\cof(\mathcal{M})\leq \supcof$
\end{Corol}

Figure \ref{diagram} summarizes some results of this section. As an application, Yorioka's characterization of $\cof(\SNcal)$ can be reformulated as follows.

\begin{theorem}[Yorioka {\cite[Thm. 3.9]{Yorioka}}]
   If $\minadd=\supcof=\kappa$ then $\cof(\SNcal)=\dfrak_\kappa$. In particular, $\add(\Ncal)=\cof(\Ncal)=\kappa$ implies $\cof(\SNcal)=\dfrak_\kappa$.
\end{theorem}

\begin{figure}
\begin{center}
\includegraphics[scale=0.9]{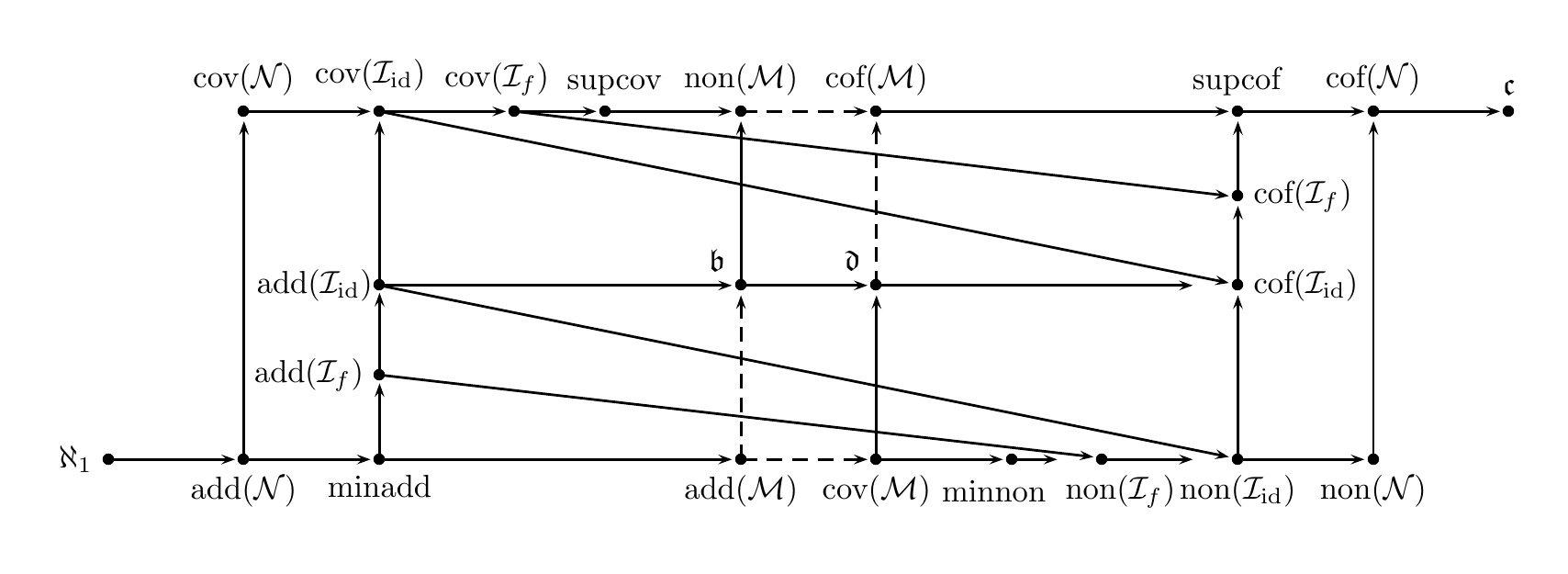}
\caption{Extended Cicho\'n's diagram.}
\label{diagram}
\end{center}
\end{figure}



\section{Preservation properties}\label{SecPreservation}

The preservation properties discussed in this section were developed for FS iterations of ccc posets by Judah-Shelah \cite{JS} and Brendle \cite{Br}, later generalized and summarized in \cite{BJ}, \cite{M} and \cite{FFMM}. We generalize this theory so that preservation properties as in \cite{KO} (see Example \ref{3.18}) become particular cases. Afterwards, we show how to adapt this theory to preserve unbounded reals along FS iterations, which is useful in the context of matrix iterations.

\subsection{The presevation theory}\label{SubSecPresTheory}

Our notation is closer to \cite{FFMM}. The classical preservation theory of Judah-Shelah and Brendle corresponds to the case $|\Omega|=1$ of the definition below. Though the proofs of the facts in this section follow the same ideas as the classical results, the arguments are presented for completeness.

\begin{definition}\label{3.1}
Say that $\Rcal=\langle X,Y,\sqsubset\rangle$ is a \textit{generalized Polish relational system (gPrs)} if
\begin{enumerate}
\item[(I)] $X$ is a Perfect Polish space,
\item[(II)] $Y=\bigcup_{e\in \Omega}Y_e$ where $\Omega$ is a non-empty set and, for some Polish space $Z$, $Y_e$ is non-empty and analytic in $Z$ for all $e\in \Omega$, and
\item[(III)] $\sqsubset=\bigcup_{n<\omega}\sqsubset_{n}$ where $\langle\sqsubset_{n}: n<\omega\rangle$  is some increasing sequence of closed subsets of $X\times Z$ such that, for any $n<\omega$ and for any $y\in Y$,
$(\sqsubset_{n})^{y}=\{x\in X:x\sqsubset_{n}y \}$ is closed nowhere dense.
\end{enumerate}

If $|\Omega|=1$, we just say that $\Rcal$ is a \emph{Polish relational system (Prs)}.

For a set $A$ and $x\in X$ say that $x$ is \textit{$\Rcal$-unbounded over  $A$} if $\forall y \in A\cap Y(x\not\sqsubset y)$.
\end{definition}

Fix, throughout this section, a gPrs $\Rcal=\langle X,Y,\sqsubset\rangle$ as in the previous definition.

\begin{lemma}\label{bRnonM}
$\langle X,\Mcal(X),\in\rangle \leqT \Rcal$. In particular, $\bfrak(\Rcal)\leq \non(\Mcal)$ and $\cov(\Mcal)\leq\dfrak(\Rcal)$.
\begin{proof}
Let $\varphi_-=\id_X$ and $\varphi_{+}:Y\to\Mcal(X)$ defined by $\varphi_{+}(y)=\{x\in X:x \sqsubset y\}$. Clearly, the pair $(\varphi_{-},\varphi_{-})$ witnesses $\langle X,\Mcal(X),\in\rangle \leqT \Rcal$.
\end{proof}
\end{lemma}

\begin{definition} \label{DefRunb}
   Let $\theta$ be a cardinal. A family $F\subseteq X$ is \textit{$\theta$-$\Rcal$-unbounded} if for any $E\subseteq Y$ of size $<\theta$ there is an $x\in F$ which is $\Rcal$-unbounded over $E$; say that $F$ is \textit{strongly $\theta$-$\Rcal$-unbounded} if $|F|\geq \theta$ and $|\{x\in F:x\sqsubset y\}|<\theta$ for all $y\in Y$.
\end{definition}

For $\theta\geq2$, any $\theta$-$\Rcal$-unbounded family is $\Rcal$-unbounded and, for $\theta$ regular, if $F$ is a strongly $\theta$-$\Rcal$-unbounded family then it is $|F|$-$\Rcal$-unbounded. In consequence,

\begin{lemma}\label{bRdR}
  \begin{enumerate}[(a)]
      \item If $\theta\geq2$ and $F\subseteq X$ is $\theta$-$\Rcal$-unbounded then $\bfrak(\Rcal)\leq|F|$ and $\theta\leq\dfrak(\Rcal)$.
      \item If $\theta$ is regular and $F\subseteq X$ is strongly $\theta$-$\Rcal$-unbounded then $\bfrak(\Rcal)\leq|F|\leq\dfrak(\Rcal)$.
  \end{enumerate}
\end{lemma}

The following are useful properties to preserve (strongly) $\theta$-$\Rcal$-unbounded families in forcing generic extensions. In this context, $X,Z$ and $\sqsubset$ are interpreted in transitive models of ZFC as Polish spaces, while $Y$ is interpreted as $Y^{M}=\bigcup_{e\in\Omega}Y^{M}_{e}$ for such a model $M$ containing the information to define $Y$. As in the case of Polish spaces, we also omit the upper indices $M$ on the interpretation of $Y$.

\begin{definition}
Let $\Pbb$ be a  forcing notion and $\theta$ a cardinal.
\begin{enumerate}[(1)]
\item $\Pbb$ is \textit{$\theta$-$\Rcal$-good} if, for any $\Pbb$-name $\dot{h}$ for a member of $Y$, there exists  a non-empty $H\subseteq Y$ (in the ground model) of size $<\theta$ such that, for any $x\in X$, if $x$ is $\Rcal$-unbounded over  $H$ then $\Vdash x\not\not\sqsubset \dot{h}$.
\item $\Pbb$ is  \textit{$\theta$-$\Rcal$-nice} if, for all $e\in\Omega$ and for any $\Pbb$-name $\dot{h}$ for a member of $Y_{e}$, there exists a non-empty $H\subseteq Y$ of size $<\theta$ such that, for any $x\in X$, if $x$ is $\Rcal$-unbounded over  $H$ then $\Vdash x\not\not\sqsubset \dot{h}$.
\end{enumerate}
Say that $\Pbb$ is \textit{$\Rcal$-good} (\textit{$\Rcal$-nice}) if it is $\aleph_1$-$\Rcal$-good ($\aleph_1$-$\Rcal$-nice).
\end{definition}

Note that $\theta<\theta^{\prime}$ implies that any $\theta$-$\Rcal$-good poset is $\theta^{\prime}$-$\Rcal$-good. Also, if $\Pbb\lessdot\Qbb$ and $\Qbb$ is $\theta$-$\Rcal$-good, then $\Pbb$ is $\theta$-$\Rcal$-good. Similar facts hold for niceness. It is clear that any  $\theta$-$\Rcal$-good forcing notion is $\theta$-$\Rcal$-nice. The converse holds in some cases as below.

\begin{lemma}\label{3.7}
   Let $\theta$ be a regular cardinal. If either $\Pbb$ is $\theta$-cc or $|\Omega|<\theta$,  then $\Pbb$ is $\theta$-$\Rcal$-nice iff it is $\theta$-$\Rcal$-good.
\end{lemma}
\begin{proof} Assume that  either  $\Pbb$ is $\theta$-cc or $|\Omega|<\theta$. Let $\dot{h}$ an $\Pbb$-name for a member of $Y$. Choose a maximal antichain $A$ in $\Pbb$ and $\{ e_p :p\in A\}\subseteq \Omega$ such that  $p\Vdash \dot{h}\in Y_{e_{p}}$ for all $p\in A$. Put $\Gamma:=\{e_{p}:p \in A\}$. By hypothesis, $\Gamma$ has size $<\theta$. For each $\beta \in \Gamma$ define $A_{\beta}=\{p\in A:e_{p}=\beta\}$. As $p\Vdash \dot{h}\in Y_\beta$ for any $p\in A_\beta$, we can find a $\Pbb$-name $\dot{y}_{\beta}$ of a member of $Y_\beta$ such that $p\Vdash \dot{h}=\dot{y}_\beta$ for any $p\in A_\beta$.

As $\Pbb$ is $\theta$-$\Rcal$-nice, for each $\beta\in \Gamma$ there exists a non-empty $H_{\beta}\subseteq Y$ of size $<\theta$ that witnesses niceness for $\dot{y}_\beta$.
Put $H=\bigcup_{\beta\in \Gamma}H_{\beta}$ which has size $<\theta$ because 
$\theta$ is regular. Assume that  $x\in X$ is $\Rcal$-unbounded over $H$. Given $p\in A$, there is a $\beta\in \Gamma$ such that $p\in A_{\beta}$. As $x\in X$ is $\Rcal$-unbounded over $H_{\beta}$, $p\Vdash x\not\sqsubset\dot{y}_{\beta}$ and, on the other hand, $p\Vdash \dot{h}=\dot{y}_{\beta}$, so $p\Vdash x\not\sqsubset \dot{h}$. As $A$ is a maximal antichain, $\Vdash x\not\sqsubset \dot{h}$.
\end{proof}

\begin{lemma}\label{3.8}
Let $\theta$ be a regular cardinal, $\lambda\geq \theta$ a cardinal and let $\Pbb$ be a $\theta$-$\Rcal$-good poset.
\begin{enumerate}[(a)]
\item If $F\subseteq X$ is $\lambda$-$\Rcal$-unbounded, then $\Pbb$ forces that it is  $\dot{\lambda}'$-$\Rcal$-unbounded where, in the $\Pbb$-extension, $\dot{\lambda}'$ is the smallest cardinal $\geq\lambda$.
\item If $\cf(\lambda)\geq\theta$ and $F\subseteq X$ is strongly $\lambda$-$\Rcal$-unbounded then $\Vdash$``if $\lambda$ is a cardinal then $F$ is strongly $\lambda$-$\Rcal$-unbounded".
\item If $\dfrak(\Rcal)\geq\lambda$ then $\Pbb$ forces that $\dfrak(\Rcal)\geq\dot{\lambda}'$.
\end{enumerate}
\end{lemma}
\begin{proof}
\begin{enumerate}[(a)]
\item It is enough to consider sets of $\Pbb$-names for members of $Y$ of the form $\{\dot{g}_{\alpha}\}_{\alpha<\eta}$ for some $\eta<\lambda$. For $\alpha<\eta$, let $H_{\alpha}\subseteq Y$ of size $<\theta$ that witnesses the goodness of $\Pbb$ for $\dot{g}_{\alpha}$. Put $H=\bigcup_{\alpha<\eta}H_\alpha$. As  $|H|<\lambda$, there is some $x\in F$ $\Rcal$-unbounded over $H$. Thus $\Vdash x\not\sqsubset \dot{g}_{\alpha}$ for any $\alpha<\eta$.


\item Repeat the argument above with $\eta=1$ and find $H$. Hence $\Vdash\{x\in F:x\sqsubset\dot{g}_0\}\subseteq\bigcup_{h\in H}\{x\in F:x\sqsubset h\}$. That union has size $<\lambda$ in the ground model. On the other hand, $F$ is forced to have size $\geq|\lambda|$.

\item Consequence of (a) because $V\models$``$X$ is $\lambda$-$\Rcal$-unbounded".
\end{enumerate}
\end{proof}

We now aim to prove that $\theta$-$\Rcal$-goodness is respected in FS iterations of $\theta$-cc posets.

\begin{definition} Let $\Pbb$ be a forcing notion and let $\dot{z}$ be a $\Pbb$-name for a real in $\omega^\omega$. A  pair $(\la p_{n}\ra_{n<\omega},g)$ is called an \textit{interpretation of $\dot{z}$ in $\Pbb$} if $g\in\omega^{\omega}$ and, for all $n<\omega$,
\begin{enumerate}[(i)]
\item $p_n\in\Pbb$, $p_{n+1}\leq p_{n}$, and
\item $p_{n}\Vdash \dot{z}\frestr n=g\frestr n$.
\end{enumerate}
Say that this interpretation is \textit{below $p\in\Pbb$} if, additionally, $p_{0}\leq p$.
\end{definition}


\begin{lemma}\label{3.10.2.1}
Assume that $\Pbb$ is a poset , $e\in \Omega$, $f:\omega^\omega\to Y_e$ is a continuous function, $\dot{z}$ is a $\Pbb$-name for a real in $\omega^\omega$ and $(\la p_{n}\ra_{n<\omega},g)$ is an interpretation of $\dot{z}$ in $\Pbb$. If $x\in X$, $n<\omega$ and $x\not\sqsubset_n f(g)$, then there is a $k<\omega$ such that $p_{k}\Vdash x \not\sqsubset_{n} f(\dot{z})$.
\end{lemma}
\begin{proof} As $\{y\in Y_e:x\sqsubset_{n}y\}$ is closed in $Y_{e}$ (see Definition \ref{3.1}) and $f:\omega^\omega\to Y_e$ is continuous, $f^{-1}[\{y\in Y_e:x\sqsubset_{n}y\}]$ is closed in $\omega^\omega$. Define $C_{x}:=\{w\in \omega^\omega:x\sqsubset_{n}f(w)\}$ and note that $f^{-1}[\{y\in Y_e:x\sqsubset_{n}y\}]=C_{x}$. If $x\not\sqsubset_n f(g)$ then there is a $k<\omega$ such that $[g\frestr k] \cap C_{x}=\emptyset$. On the other hand $p_{k}\Vdash \dot{z}\frestr k=g\frestr k$, so $p_{k}\Vdash [\dot{z}\frestr k] \cap C_{x}=\emptyset$. Hence $p_{k}\Vdash\dot{z} \notin C_{x}$, that is, $p_{k}\Vdash x\not\sqsubset_n f(\dot{z})$.
\end{proof}

\begin{lemma}\label{3.10}
  If $\theta$ is a cardinal then any poset of size $<\theta$ is $\theta$-$\Rcal$-nice.
  Moreover, if $\theta$ is regular then any such poset is $\theta$-$\Rcal$-good. In particular, $\Cbb$ is $\Rcal$-good.
\end{lemma}
\begin{proof}
Put $\Pbb=\{p_\alpha:\alpha<\mu\}$ where $\mu:=|\Pbb|<\theta$. Let $e\in \Omega$ and $\dot{h}$ be a $\Pbb$-name for a member of $Y_e$. Choose a continuous and surjective function $f:\omega^\omega\to Y_e$ and a $\Pbb$-name for a real $\dot{z}$ in $\omega^\omega$ such that $\Pbb$ forces that $f(\dot{z})=\dot{h}$. For each $\alpha<\mu$, choose an interpretation $(\la p_{\alpha,n}\ra_{n<\omega},z_{\alpha})$ of $\dot{z}$ below $p_{\alpha}$. We prove that, if $x\in X$ and $\forall{\alpha<\mu}(x\not\sqsubset f(z_{\alpha}))$, then $\Vdash x\not\sqsubset \dot{h}$. Fix $p\in \Pbb$ and $m<\omega$, so there exists an $\alpha<\mu$ such that $p=p_{\alpha}$.
By  Lemma \ref{3.10.2.1} there exists a $k<\omega$ such that $p_{\alpha,k}\Vdash x\not\sqsubset_{m} f(\dot{z})$. Therefore $p_{\alpha,k}\Vdash x\not\sqsubset_{m}\dot{h}$ and $p_{\alpha,k}\leq p_{\alpha}=p$.

The `moreover' part follows by Lemma \ref{3.7}.
\end{proof}

\begin{lemma}\label{3.10.2}
Let $\theta$ be a regular cardinal, $\Pbb$ a poset and $\dot{\Qbb}$ a $\Pbb$-name for a poset. If $\Pbb$ is $\theta$-$\mathrm{cc}$, $\theta$-$\Rcal$-good and it forces that $\dot{\Qbb}$ is $\theta$-$\Rcal$-good, then $\Pbb\ast\dot{\Qbb}$ is $\theta$-$\Rcal$-good
\end{lemma}
\begin{proof}
Let $\dot{h}$ be a $\Pbb\ast\dot{\Qbb}$-name for a member of $Y$. Wlog $\Pbb$ forces that $\dot{h}$ is a $\dot{\Qbb}$-name for a member for $Y$. As $\Pbb$ forces that $\dot{\Qbb}$ is $\theta$-$\Rcal$-good, $\Pbb$ forces that there is a nonempty $\dot{H}\subseteq Y$ of size $<\theta$ such that, for any $x\in X$, if $x$ $\Rcal$-unbounded over $\dot{H}$, then $\Vdash_{\dot{\Qbb}} x \not\sqsubset \dot{h}$. As $\Pbb$ is $\theta$-cc, we can find $\nu<\theta$ in the ground model so that $\dot{H}=\{\dot{y}_{\alpha}:\alpha<\nu\}$. For each $\alpha<\nu$ let $B_\alpha$ be a witness of goodness of $\Pbb$ for $\dot{y}_{\alpha}$. Put $B:=\bigcup_{\alpha<\nu}B_\alpha$, which has size $<\theta$. It is easy to see that, if $x\in X$ is $\Rcal$-unbounded over $B$, then $\Vdash_{\Pbb\ast\dot{\Qbb}}x\not\sqsubset \dot{h}$. 
\end{proof}

We show that goodness is preserved along direct limits in quite a general way so that the theory of this section can be applied to template iterations as in \cite[Sect. 5]{mejia-temp}. Say that a partial order $I$ is \emph{directed} if, for any $i,j\in I$, there is a $k\in I$ such that $i,j\leq k$. A system  $\langle\Pbb_{i}\rangle_{i\in I}$ of posets indexed by a directed partial order $I$ is called a \emph{directed system of posets} if $\Pbb_i$ is a complete subposet of $\Pbb_j$ for all $i\leq j$ in $I$. For such a directed system, define its \emph{direct limit} by $\limdir_{i\in I}\Pbb_i:=\bigcup_{i\in I}\Pbb_i$.

\begin{theorem}\label{3.12.1}
Let $\theta$ be a regular cardinal, $\langle\Pbb_{i}\rangle_{i\in I}$ a directed system of posets and $\Pbb:=\limdir_{i\in I}\Pbb_i$. If $|I|<\theta$ and $\Pbb_i$ is $\theta$-$\Rcal$-nice for all $i\in I$, then $\Pbb$ is $\theta$-$\Rcal$-nice.
\end{theorem}
\begin{proof}
Let $e\in\Omega$ and let $\dot{h}$ be a $\Pbb$-name for member in $Y_e$. Choose a continuous and surjective function $f:\omega^\omega\to Y_e$ and a $\Pbb$-name for a real $\dot{z}$ in $\omega^\omega$ such that $\Pbb$ forces that $f(\dot{z})=\dot{h}$. For each $i\in I$, find a $\Pbb_i$-name for a real $\dot{z}_i$ in $\omega^\omega$ and a sequence $\langle \dot{p}_{i,k}\ra_{k<\omega}$ of $\Pbb_i$-names such that $\Pbb_i$ forces that $(\la\dot{p}_{i,k}\ra_{k<\omega},\dot{z}_i)$ is an interpretation of $\dot{z}$ in $\Pbb/\Pbb_i$. Choose $H_i\subseteq Y$ of size $<\theta$ such that it witnesses goodness of $\Pbb_i$ for $f(\dot{z}_i)$. Put $H=\bigcup_{i\in I}H_i$, which has size $<\theta$ since $|I|<\theta$ and $\theta$ is regular.

We prove that, if $x\in X$ is $\Rcal$-unbounded over $H$, then $\Vdash_{\Pbb} x\not\sqsubset \dot{h}$. Assume towards a contradiction that there are $p\in\Pbb$ and $n<\omega$ such that $p\Vdash_{\Pbb}x\sqsubset_{n}\dot{h}$. Choose $i\in I$ such that
$p\in\Pbb_i$. Let $G$ be a $\Pbb_i$-generic over the ground model $V$ with $p\in G$. By the choice of $H_i$, $x\not\sqsubset f(\dot{z}_i[G])$, in particular, $x\not\sqsubset_{n} f(\dot{z}_i[G])$.
By  Lemma \ref{3.10.2.1}, there is a $k<\omega$ such that  $\dot{p}_{i,k}[G]\Vdash_{\Pbb/\Pbb_i}x\not\sqsubset_{n}f(\dot{z})=\dot{h}$. On the other hand, by hypothesis, $p\Vdash_{\Pbb_i}``\Vdash_{\Pbb/\Pbb_i}x\sqsubset_n \dot{h}$'', a contradiction.
\end{proof}

\begin{remark}
   We can replace the hypothesis $|I|<\theta$ by $\cf(I)<\theta$ in the previous result. This is because $\limdir_{i\in I}\Pbb_i=\limdir_{i\in C}\Pbb_i$ for any cofinal $C\subseteq I$.
\end{remark}



\begin{Corol}\label{3.11}
Let $\theta$ be an uncountable regular cardinal and $\mathbb{P}_{\delta}=\langle\mathbb{P}_{\alpha},\dot{\mathbb{Q}}_{\alpha}\rangle_{\alpha<\delta}$ a FS iteration of $\theta$-$\mathrm{cc}$ forcing notions. If, for each $\alpha<\delta$, $\Pbb_{\alpha}$ forces that $\dot{\mathbb{Q}}_{\alpha}$ is $\theta$-$\Rcal$-good, then $\Pbb_\delta$ is $\theta$-$\Rcal$-good.
\end{Corol}
\begin{proof}
   We prove that $\Pbb_\alpha$ is $\theta$-$\Rcal$-good by induction on $\alpha\leq\delta$. The step $\alpha=0$ follows by Lemma \ref{3.10} and the successor step follows by Lemma \ref{3.10.2}. Assume that $\alpha$ is a limit ordinal. If $\cf(\alpha)<\theta$ then $\Pbb_\alpha$ is $\theta$-$\Rcal$-good by Theorem \ref{3.12.1} and Lemma \ref{3.7}.

   Assume that $\mathrm{cf}(\alpha)\geq \theta$. Let $\dot{h}$ be a $\Pbb_{\alpha}$-name for a member of $Y$. By $\theta$-cc-ness, there exists a $\xi<\alpha$ such that $\dot{h}$ is a $\Pbb_{\xi}$-name. As $\Pbb_{\xi}$ is $\theta$-$\Rcal$-good, there is a non-empty $H\subseteq Y$ of size $<\theta$ that witnesses goodness of $\Pbb_{\xi}$ for $\dot{h}$. It is clear that $H$ also witnesses goodness of $\Pbb_{\alpha}$ for $\dot{h}$.
\end{proof}

Recall that $c\in X$ is a \emph{Cohen real over a model $M$} if $c$ does not belong to any Borel meager set coded in $M$. It is clear that Cohen forcing adds such a real over the ground model. Indeed, given a metric $d$ on $\omega$ such that $X$, as a complete metric space, is a completion of $\la\omega,d\ra$, consider $\Cbb^d:=\{t\in\omega^{<\omega}:\forall{i<|t|-1}(d(t_i,t_{i+1})<2^{-(i+2)})\}$ ordered by end-extension. This is a countable atomless poset (because $\la\omega,d\ra$ is perfect), so it is forcing equivalent to $\Cbb$. It is not hard to see that $\Cbb^d$ adds a Cauchy-sequence that converges to a Cohen real in $X$ over the ground model.

By Definition \ref{3.1}(III), any Cohen real in $X$ over the ground model is $\Rcal$-unbounded over the ground model. Hence, it is possible to add (strongly) $\Rcal$-unbounded families with Cohen reals through FS iterations. 

\begin{lemma}\label{3.12}
If $\nu$ is a cardinal with uncountable cofinality and $\mathbb{P}_{\nu}=\langle\mathbb{P}_{\alpha},\dot{\mathbb{Q}}_{\alpha}\rangle_{\alpha<\nu}$ is a FS iteration of non-trivial $\cf(\nu)$-cc posets, then $\Pbb_{\nu}$ adds a strongly $\nu$-$\Rcal$-unbounded family of size $\nu$.
\end{lemma}
\begin{proof} The Cohen reals (in $X$) added at the limit steps of the iteration form a strongly $\nu$-$\Rcal$-unbounded family of size $\nu$.
\end{proof}

\begin{theorem}\label{3.13}
Let $\theta$ be an uncountable regular cardinal, $\delta\geq\theta$ an ordinal, and let $\mathbb{P}_{\delta}=\langle\mathbb{P}_{\alpha},\dot{\mathbb{Q}}_{\alpha}\rangle_{\alpha<\delta}$ be a FS iteration such that, for each $\alpha<\delta$, $\dot{\Qbb}_{\alpha}$ is a $\Pbb_{\alpha}$-name of a  non-trivial $\theta$-$\Rcal$-good $\theta$-cc poset. Then:
\begin{enumerate}[(a)]
    \item For any cardinal $\nu\in[\theta,\delta]$ with $\cf(\nu)\geq\theta$, $\Pbb_\nu$ adds a strongly $\nu$-$\Rcal$-unbounded family of size $\nu$ which is still strongly $\nu$-$\Rcal$-unbounded in the $\Pbb_\delta$-extension.
    \item For any cardinal $\lambda\in[\theta,\delta]$, $\Pbb_\lambda$ adds a $\lambda$-$\Rcal$-unbounded family of size $\lambda$ which is still $\lambda$-$\Rcal$-unbounded in the $\Pbb_\delta$-extension.
    \item $\mathbb{P}_{\delta}$ forces that $\bfrak(\Rcal)\leq\theta$ and  $|\delta|\leq\dfrak(\Rcal)$.
\end{enumerate}
\end{theorem}
\begin{proof}
   (a) is a direct consequence of Lemmas \ref{3.8}, \ref{3.12} and the fact that $\Pbb_\delta/\Pbb_\nu$, the remaining part of the iteration from stage $\nu$, is $\theta$-$\Rcal$-good (by Corollary \ref{3.11}). On the other hand, by Lemma \ref{bRdR}, $\bfrak(\Rcal)\leq\theta$ follows from (a) for $\nu=\theta$ and $\dfrak(\Rcal)\geq|\delta|$ follows from (b) for $\lambda=|\delta|$.

   It remains to prove (b) for the case when $\lambda$ is singular (for $\lambda$ regular it just follows from (a)). Work in $V_\lambda=V^{\Pbb_\lambda}$. Let $\{\delta_\xi:\xi<\lambda\}$ be the increasing enumeration of 0 and the limit ordinals below $\lambda$ and, for each $\xi<\lambda$, denote by $c_\xi$ the Cohen real in $X$ over $V_{\delta_\xi}$ added by $\Pbb_{\delta_{\xi+1}}$. As explained in the proof of Lemma \ref{3.12}, for each $\nu\in[\theta,\lambda)$ regular, $\{c_\xi:\xi<\nu\}\in V_\nu$ is (strongly) $\nu$-$\Rcal$-unbounded in $V_\nu$, and also in $V_\lambda$ by (a). Thus $\{c_\xi:\xi<\lambda\}$ is $\lambda$-$\Rcal$-unbounded. Indeed, if $A\subseteq Y$ has size $<\lambda$ then it has size $<\nu$ for some regular $\nu\in[\theta,\lambda)$, so there is some $\xi<\nu$ such that $c_\xi$ is $\Rcal$-unbounded over $A$.

   As $\Pbb_\delta/\Pbb_\lambda$ is $\theta$-$\Rcal$-good, then $\{c_\xi:\xi<\lambda\}$ is still $\lambda$-$\Rcal$-unbounded in the $\Pbb_\delta$-extension by Lemma \ref{3.8}.
\end{proof}

\subsection{Examples of preservation properties}\label{3.2}

We start with examples of the classical framework, that is, with instances of Polish relational systems.

\begin{Example}\label{ExmPrs}
  \begin{enumerate}[(1)]
     \item For every $b:\omega\to\omega+1\smallsetminus\{0\}$ with $b\not\leq^*1$, the relational system $\Ed(b)$ (Example \ref{ExmCichonrelsys}(3)) is a Prs. Indeed, $\neq^*=\bigcup_{n<\omega}\neq^*_n$ where $\neq^*_n$ is defined as $x\neq^*_n y$ iff $x(i)\neq y(i)$ for all $i\geq n$. If $b:\omega\to\omega$ and $\nu<\theta$ are infinite cardinals, then any $\nu$-centered poset is $\theta$-$\Ed(b)$-good (similar to Lemma \ref{3.21} when $h=1$).

     \item For every $b:\omega\to\omega+1\smallsetminus\{0\}$ and $h\in\omega^\omega$ with $h\geq^*1$, the relational system $\aLc(b,h)$ (Definition \ref{Defantiloc}) is a Prs. Indeed, $\not\ni^*=\bigcup_{n<\omega}\not\ni^*_n$ where $\not\ni^*_n$ is defined as $\varphi\not\ni^*_n y$ iff $y(i)\notin\varphi(i)$ for all $i\geq n$. When $b:\omega\to\omega$, a similar proof as Lemma \ref{3.21} yields that any $\nu$-centered poset is $\theta$-$\aLc(b,h)$-good when $\nu<\theta$ are infinite cardinals.

     \item The relational system $\mathbf{D}=\la\omega^\omega,\omega^\omega,\leq^*\ra$ is a Prs. Clearly, any $\omega^\omega$-bounding poset is $\mathbf{D}$-good. Miller \cite{Mi} proved that $\Ebb$, the standard $\sigma$-centered poset that adds an eventually different real in $\omega^\omega$, is $\mathbf{D}$-good. A similar proof yields that $\Ebb^h_b$ is $\mathbf{D}$-good for any $b:\omega\to\omega\smallsetminus\{0\}$ and $h\in\omega^\omega$ with $\lim_{n\to+\infty}\frac{h(n)}{b(n)}=0$. Likewise, $\LOCbb^h_b$ is $\Dbf$-good for any $b:\omega\to\omega\smallsetminus\{0\}$ and $h\in\omega^\omega$ that goes to infinity.

     \item For $\mathcal{H}\subseteq\omega^\omega$ denote $\Lc(\omega,\mathcal{H}):=\langle\omega^\omega,\Scal(\omega,\mathcal{H}),\in^{*}\rangle$ where
     $\Scal(\omega,\mathcal{H}):=\bigcup_{h\in\mathcal{H}}\Scal(\omega,h)$.
     If $\mathcal{H}$ is countable and non-empty then $\Lc(\omega,\mathcal{H})$ is a Prs because $\Scal(\omega,\mathcal{H})$ is $F_\sigma$ in $([\omega]^{<\omega})^\omega$. In addition, if $\mathcal{H}$ contains a function that goes to infinity then $\bfrak(\Lc(\omega,\mathcal{H}))=\add(\Ncal)$ and $\dfrak(\Lc(\omega,\mathcal{H}))=\cof(\Ncal)$ (as a consequence of Theorem \ref{1.6}). If $\nu<\theta$ are infinite cardinals and $\theta$ is regular then any $\nu$-centered poset is $\theta$-$\Lc(\omega,\mathcal{H})$-good (\cite{JS}, see also \cite[Lemma 6]{Br}). Moreover, if all the members of $\Hcal$ go to infinity then any Boolean algebra with a strictly positive finitely additive measure is $\Lc(\omega,\Hcal)$-good (\cite{Ka}). In particular, any subalgebra of random forcing is $\mathbf{Lc}(\omega,\Hcal)$-good.
  \end{enumerate}
\end{Example}




\begin{Example}\label{Exam}
If $\Hcal$ is non-empty, then $\Lc(\omega,\Hcal)$ is a gPrs where $\Omega=\mathcal{H}$, $Z=([\omega]^{<\omega})^{\omega}$ and $Y_{h}=\Scal(\omega,h)$ for each $h\in\mathcal{H}$. Like in Example \ref{ExmPrs}(4), if $\nu<\theta$ are infinite cardinals, then any $\nu$-centered poset is $\theta$-$\Lc(\omega,\Hcal)$-nice (see \cite[Lemma 5.13]{BrM}) and, by Lemma \ref{3.7}, it is $\theta$-$\Lc(\omega,\Hcal)$-good when $\theta$ is regular (because any $\nu$-centered poset is $\theta$-cc).
\end{Example}

\begin{lemma}[Brendle and Mej\'ia {\cite[Lemma 5.14]{BrM}}]\label{linkedpresadd(N)}
   For any $\pi,\rho,g_0\in\omega^\omega$ with $\pi$ and $g_0$ going to $+\infty$, there is a $\leq^*$-increasing sequence $\la g_n:n<\omega\ra$ such that any $(\rho,\pi)$-linked poset is 2-$\Lc(\omega,\{g_n:n<\omega\})$-nice (hence $\Lc(\omega,\{g_n:n<\omega\})$-good).
\end{lemma}



\newcommand{\Slm}{\mathbf{aLc}^*}

\begin{Example}[Kamo and Osuga {\cite{KO}}]\label{3.18}
Fix a family $\mathcal{E}\subseteq\omega^\omega$ of size $\aleph_1$ of non-decreasing functions which satisfies
\begin{enumerate}[(i)]
\item $\forall e \in \mathcal{E}(e \leq \text{id}_\omega)$,
\item $\forall e \in \mathcal{E}($ $\lim_{n\to+\infty}e(n)=+\infty$ and $\lim_{n\to+\infty}(n-e(n))=+\infty$),
\item $\forall e \in \mathcal{E} \exists^{}e^{'} \in \mathcal{E}(e+1\leq^{*} e^{'})$ and
\item $\forall \mathcal{E}^{'} \in [\mathcal{E}]^{\leq \aleph_0} \exists e\in \mathcal{E}\forall e^{'}\in\mathcal{E}^{'}(e^{'}\leq^{*}e)$.
\end{enumerate}
The existence of the family $\mathcal{E}$ is a consequence of Lemma \ref{genbounding} applied to $H:=\id_\omega+1$ and $g:=\id_\omega$.

For $b,h \in \omega^\omega$ such that $b>0$ and $h\geq^*1$, we define $\hat{\Scal}(b,h)=\hat{\Scal}_{\mathcal{E}}(b,h)$ by
\[\hat{\Scal}(b,h):=\bigcup_{e\in\mathcal{E}}\Scal(b,h^e)=\Big\{\varphi \in \prod_{n<\omega}\mathcal{P}(b(n)):\exists e\in \mathcal{E}\forall n<\omega (|\varphi(n)|\leq h(n)^{e(n)})\Big\}\]

Let $n<\omega$. For $\psi, \varphi :\omega\to[\omega]^{<\omega}$,
define the relation $\psi\blacktriangleright_{n}\varphi$ iff $\forall{k\geq n}(\psi(k)\nsupseteq\varphi(k))$, and define
$\psi\blacktriangleright\varphi$ iff $\forall^{\infty}{k<\omega}(\psi(k)\nsupseteq\varphi(k))$, i.e., $\blacktriangleright=\bigcup_{n<\omega}\blacktriangleright_{n}$. Put $\Slm(b,h):=\langle \Scal(b,h^{\mathrm{id}_{\omega}}), \hat{\Scal}(b,h), \blacktriangleright\rangle$ which is a gPrs where $\Omega=\mathcal{E}$, $Z=\Scal(b,h^{\id_\omega})$ and $Y_e=\Scal(b,h^{e})$ for each $e\in\mathcal{E}$. Note that $Y_e$ is closed in $Z$.

The property $\theta$-$\Slm(b,h)$-good is what Kamo and Osuga \cite[Def. 6]{KO} denote by $(\ast^{<\theta}_{b,h})$. However, they use $\hat{\Scal}(b,h)$ instead of $\Scal(b,h^{\id_\omega})$ for the first coordinate of $\Slm(b,h)$ (implicitly), which we think does not work for the second claim in the proof of \cite[Thm. 1]{KO}. We believe this can be corrected with the gPrs $\Slm(b,h)$ we propose here.

When $h=1$, $\Slm(b,1)=\la\Scal(b,1),\Scal(b,1),\blacktriangleright\ra\eqT\Ed(b)$, so $\Slm(b,1)$ can be described as a Prs. Even more, $\bfrak(\Slm(b,1))=\balc_{b,1}$ and $\dfrak(\Slm(b,1))=\dalc_{b,1}$.
\end{Example}

\begin{Example}\label{ExLc*}
   Define $\Lc^*(b,h):=\la \prod b,\hat{\Scal}(b,h),\in^* \ra$. When $h\geq^*1$ and $b>^* h^e$ for any $e\in\mathcal{E}$, it is clear that $\Lc^*(b,h)$ is a gPrs.
\end{Example}

\begin{lemma}\label{Slmcard}
\begin{enumerate}[(a)]
    \item $\Slm(b,h)\leqT\aLc(b,h^{\id_\omega})$. In particular, $\balc_{b,h^{\mathrm{id}_{\omega}}}\leq\bfrak(\Slm(b,h))$ and $\dfrak(\Slm(b,h))\leq\dalc_{b,h^{\mathrm{id}_{\omega}}}$.
    \item $\Lc^*(b,h)\leqT\Lc(b,h)$. In particular, $\blc_{b,h}\leq\bfrak(\Lc^*(b,h))$ and $\dfrak(\Lc^*(b,h))\leq\dlc_{b,h}$.
\end{enumerate}
\end{lemma}
\begin{proof}
   Put $\varphi_-:=\id_{\Scal(b,h^{\id_\omega})}$ and $\psi_-=\id_{\prod b}$. Define $\varphi_{+}:\prod{b}\to \hat{\Scal}(b,h)$ by $\varphi_+(y)(i)=\{y(i)\}$ whenever $h(i)>0$, and, for some fixed $e_0\in\Ecal$, define $\psi_+:\Scal(b,h)\to\hat{\Scal}(b,h)$ by $\psi_+(S)(i)=S(i)$ whenever $e_0(i)>0$. It is clear that $(\varphi_-,\varphi_+)$ is a Tukey-Galois connection for (a), and $(\psi_-,\psi_+)$ is one for (b).
\end{proof}



By the previous result, we can use the preservation property of $\Slm(b,h)$ to decide the values of $\balc_{b,h^{\mathrm{id}_{\omega}}}$ and $\dalc_{b,h^{\mathrm{id}_{\omega}}}$ by forcing, likewise for $\Lc^*(b,h)$ and the localization cardinals.

\begin{lemma}\label{1.19} Let $n<\omega$ and $B\subseteq \Pbb$ be $n$-linked. If $F\in V$ has size $\leq n$ and $\dot{a}$ is a $\Pbb$-name for a member of $F$, then there exists a $c \in F$ such that no $p \in B$ forces $\dot{a}\neq c$.
\end{lemma}
\begin{proof}
   See e.g. \cite[Lemma 7]{KO}.
\end{proof}

\begin{lemma}\label{3.20}
If $h,b \in \omega^\omega$ with $b>0$, $b\not\leq^*1$ and $h\geq^*1$, then any $(h,b^{h^{\text{id}_\omega}})$-linked poset is both $2$-$\Slm(b,h)$-good and $2$-$\Lc^*(b,h)$-good.
\end{lemma}
\begin{proof}
   This was proved for the gPrs $\Slm(b,h)$ in \cite[Lemma 10]{KO}. The case of $\Lc^*(b,h)$ follows from the same proof, which we include for completeness. Let $\Pbb$ be a poset and let $\la Q_{n,j}:n<\omega, j<h(n)\ra$ be a sequence that witnesses that $\Pbb$ is $(h,b^{h^{\text{id}_\omega}})$-linked. Assume that $\dot{\varphi}$ is a $\Pbb$-name of a member of $\hat{\Scal}(b,h)$. As $\Pbb$ is ccc (by Lemma \ref{linksigmalink}), we can find an $e\in\Ecal$ such that $\Pbb$ forces that $\varphi\in\Scal(b,h^e)$. Furthermore, choose $e'\in\Ecal$ and $N<\omega$ such that $h(n)\cdot e(n)>0$ and $e(n)+1\leq e'(n)$ for all $n\geq N$, so we can find a $\Pbb$-name $\dot{\varphi}'$ of a member of $\Scal(b,h^{e})$ such that $\Pbb$ forces $\dot{\varphi}(n)\subseteq\dot{\varphi}'(n)\neq\emptyset$ for all $n\geq N$.

   For each $n\geq N$ and $j<h(n)$, as $Q_{n,j}$ is $b(n)^{h(n)^n}$-linked and $[b(n)]^{\leq h(n)^{e(n)}}\smallsetminus\{\emptyset\}$ has size $\leq b(n)^{h(n)^n}$, by Lemma \ref{1.19} there is an $a_{n,j}\subseteq b(n)$ of size $\leq h(n)^{e(n)}$ such that $p\nVdash a_{n,j}\neq\dot{\varphi}'(n)$ for all $p\in Q_{n,j}$. Note that $\psi(n):=\bigcup_{j<h(n)}a_{n,j}$ has size $\leq h(n)^{e(n)+1}\leq h(n)^{e'(n)}$, which gives us a $\psi\in\Scal(b,h^{e'})$. It is clear that
   \begin{enumerate}[(i)]
       \item if $\vartheta\in\Scal(b,h^{\id_\omega})$ and $\vartheta\not\blacktriangleright\psi$ then $\Vdash\vartheta\not\blacktriangleright\dot{\varphi}'$ (which implies $\Vdash\vartheta\not\blacktriangleright\dot{\varphi}$), and
       \item if $x\in\prod b$ and $\neg(x\in^*\psi)$ then $\Vdash\neg(x\in^*\dot{\varphi}')$ (so $\Vdash\neg(x\in^*\dot{\varphi})$).
   \end{enumerate}
   This concludes the proof.
\end{proof}

\begin{lemma}\label{3.21} If $\mu<\theta$ are infinite cardinals then any $\mu$-centered poset is both $\theta$-$\Slm(b,h)$-nice and $\theta$-$\Lc^*(b,h)$-nice. In addition, if $\theta$ is regular, then any $\mu$-centered poset is both $\theta$-$\Slm(b,h)$-good and $\theta$-$\Lc^*(b,h)$-good.
\end{lemma}
\begin{proof}
The latter part is a consequence of  Lemma \ref{3.7}.

Let $\Pbb$ be a poset such that $\Pbb=\bigcup_{\alpha<\mu}P_\alpha$ where each $P_\alpha$ is centered. Fix $e\in\mathcal{E}$ and a $\Pbb$-name $\dot{\varphi}$ for a member of $\Scal(b,h^e)$. For each $\alpha<\mu$ and $m\in \omega$, by Lemma \ref{1.19} find a $\psi_{\alpha}(m)\in[b(m)]^{\leq(h(m)^{e(m)})}$ such that no $p\in P_\alpha$ forces  $\psi_{\alpha}(m)\neq \dot{\varphi}(m)$. Put $H:=\{\psi_{\alpha}:\alpha<\mu\}$, which is a subset of $\Scal(b,h^e)$. Assume that $\vartheta \in\Scal(b,h^{\mathrm{id}_{\omega}})$ and $\vartheta\not \blacktriangleright\psi_\alpha$ for all $\alpha<\mu$. Fix $p\in \Pbb$ and $m<\omega$. Choose $\alpha<\mu$ such that $p\in P_\alpha$ and find a $k\geq m$ such that $\vartheta(k)\supseteq \psi_{\alpha}(k)$. As $p\not\Vdash\psi_{\alpha}(k)\neq\dot{\varphi}(k)$, there is a $q\leq p$ that forces $\vartheta(k)\supseteq \psi_{\alpha}(k)=\dot{\varphi}(k)$.

A similar argument yields that, whenever $x\in\prod b$ and $\neg(x\in^*\psi_\alpha)$ for any $\alpha<\mu$, $\Pbb$ forces that $\neg(x\in^*\dot{\varphi})$.
\end{proof}


\newcommand{\Fr}{\mathbf{Fr}}

The following example abbreviates many facts about the main preservation result in \cite{BrM}. This will not be used in any other part of this text.

\begin{Example}[Brendle and Mej\'{\i}a {\cite{BrM}}]\label{ExmBM}
   Let $\bar{a}=\la a_i:i<\omega\ra$ be a partition of $\omega$ into non-empty finite sets and $\bar{L}=\la L_n:n<\omega\ra$ a partition of $\omega$ into infinite sets. For each $i<\omega$ let $\varphi_i:\Pcal(a_i)\to[0,+\infty)$ be a submeasure. Fix $h\geq^*1$ in $\omega^\omega$ and let $b_{\bar{a}}(i):=\Pcal(a_i)$ for each $i<\omega$.

   Define $P_m(i):=\{a\subseteq a_i:\varphi_i(a)\leq m\}$ for each $i,m<\omega$. Put $\Omega:=\omega\times\omega\times\mathcal{E}$ ($\mathcal{E}$ as in Example \ref{3.18}) and, for each $(m,n,e)\in\Omega$, put $Y_{m,n,e}:=\{n\}\times\Scal(P_m,h^e)$, which is closed in $Z:=\omega\times\Scal(b_{\bar{a}},h^{\id_\omega})$.

   For each $k<\omega$ define the relation $\blacktriangleright'_k\subseteq([\omega]^{<\omega})^\omega\times (\omega\times([\omega]^{<\omega})^\omega))$ by $\vartheta\blacktriangleright'_k(n,\psi)$ iff $\vartheta(i)\nsupseteq\psi(i)$ for all $i\in L_n\smallsetminus k$. Put $\blacktriangleright':=\bigcup_{k<\omega}\blacktriangleright'_k$. It is not hard to see that $\Fr(\bar{a},\bar{\varphi},\bar{L},h):=\la\Scal(b_{\bar{a}},h^{\id_\omega}),Y,\blacktriangleright'\ra$ is a gPrs where $Y:=\bigcup_{p\in\Omega}Y_p$.

   The property $\theta$-$\Fr(\bar{a},\bar{\varphi},\bar{L},h)$-good was studied in \cite[Sect. 5]{BrM} to preserve Rothberger gaps through FS iterations. Some of its results can be simplified by the theory presented in Subsection \ref{SubSecPresTheory}. Note that $\theta$-$\Fr(\bar{a},\bar{\varphi},\bar{L},h)$-goodness corresponds to \cite[Def. 5.5]{BrM}.
\end{Example}

\newcommand{\sig}{\boldsymbol{\Sigma}}
\newcommand{\cosig}{\boldsymbol{\Pi}}

\subsection{Preservation of $\Rcal$-unbounded reals}\label{3.3}


All the results of this subsection are versions of the contents of \cite[Sect. 4]{M} in the context of gPrs. Though the proofs are similar, we still present them for completeness. Fix, throughout this section, transitive models $M$ and $N$ of (a sufficient large finite fragment of) ZFC with $M\subseteq N$.

\begin{definition}
Given two posets $\Pbb\in M$  and $\Qbb$ (not necessarily in $M$) say that \textit{$\Pbb$ is a complete suborder of $\Qbb$ with respect to $M$}, denoted by $\Pbb \lessdot_{M}\Qbb$, if $\Pbb$ is a suborder of $\Qbb$ and every maximal antichain in $\Pbb$ that belongs to $M$ is also a maximal antichain in $\Qbb$.
\end{definition}

Clearly, if $\Pbb\lessdot_{M} \Qbb$ and $G$ is $\Qbb$-generic over $N$, then $G \cap \Pbb$ is $\Pbb$-generic over $M$ and $M[G \cap \Pbb]\subseteq N[G]$. Recall that, if $\Sbb$ is a Suslin ccc poset coded in $M$, then $\Sbb^{M}\lessdot_{M}\Sbb^{N}$. Also, if $\Pbb\in M$ is a poset, then $\Pbb\lessdot_{M}\Pbb$.

For the following results, fix a gPrs $\Rcal=\langle X,Y,\sqsubset\rangle$ coded in $M$ (in the sense that
all its components are coded in $M$).  We are interested in preserving $\Rcal$-unbounded reals between forcing extensions of $M$ and $N$.

\begin{lemma}[{\cite[Theorem 7]{M}}]\label{3.24} Let $\Sbb$  be a Suslin ccc poset coded in $M$. If $M\models$``$\Sbb$ is $\Rcal$-good" then, in $N$, $\Sbb^N$ forces that every $c\in X^N$ that is $\Rcal$-unbounded over $M$ is $\Rcal$-unbounded over $M^{\Sbb^{M}}$.
\end{lemma}
\begin{proof}
Let $Z'$ be the Polish space where $\Sbb$ is defined, and recall the Polish space $Z$ that contains $Y$ (see Definition \ref{3.1}). Choose a metric space $\la\eta,d\ra$ with $\eta\leq\omega$ such that $Z$, as a complete metric space with metric $d^*$, is a completion of $\la\eta,d\ra$. Note that any (good) $\Sbb$-name of a member of $Z$ can be seen as a name of a Cauchy sequence $\la\dot{k}_m:m<\omega\ra$ in $\la\eta,d\ra$ such that $d(\dot{k}_m,\dot{k}_{m+1})<2^{-(m+2)}$ for all $m<\omega$. This can be coded by a member of $(Z'\times\eta)^{\omega\times\omega}$. Therefore, ``$\dot{h}$ is a $\Sbb$-name of a member of $Z$" is a conjunction between a $\sig^1_1$-statement and a $\cosig^1_1$-statement in $(Z'\times\eta)^{\omega\times\omega}$. Indeed, a $\Sbb$-name $\dot{h}$ of a member of $Z$ is a function $\dot{h}=(p^{\dot{h}},k^{\dot{h}}):\omega\times\omega\to Z'\times\eta$ such that, for each $m<\omega$,
\begin{enumerate}[(i)]
    \item  $\{p^{\dot{h}}(m,n):n<\omega\}$ is a maximal antichain in $\Sbb$ (the point is that $p^{\dot{h}}(m,n)$ decides that the $m$-th term of the Cauchy-sequence in $\la\eta,d\ra$ converging to $\dot{h}$ is $k^{\dot{h}}(m,n)$), and
    \item for each $n,n'<\omega$, if $p^{\dot{h}}(m,n)$ and $p^{\dot{h}}(m+1,n')$ are compatible in $\Sbb$ then
    \[d(k^{\dot{h}}(m,n),k^{\dot{h}}(m+1,n'))<2^{-(m+2)}\]
    (which means that $\dot{h}$ is a name of a Cauchy-sequence as described before).
\end{enumerate}
The statement in (i) can be expressed as a conjunction between a $\sig^1_1$-statement and a $\cosig^1_1$-statement, while (ii) is a $\cosig^1_1$-statement. Even more, if $\Sbb$ is a Borel subset of $Z'$ then (i) is a $\cosig^1_1$-statement.

\begin{Claim}[{\cite[Claim 1]{M}}]\label{claim1} The statement $\Psi(\dot{h},\bar{y})$ that says ``$\dot{h}$ is a $\Sbb$-name for a member of $Z$ and, for all $x\in X$, if $x\not\sqsubset y_{n}$ for each $n<\omega$ then $\Vdash_{\Sbb} x\not\sqsubset\dot{h}$" is a conjunction of a $\sig^{1}_{1}$-statement with a $\cosig^{1}_{1}$-statement  in $(Z^{'}\times\eta)^{\omega\times\omega}\times Z^{\omega}$. Even more, if $\Sbb$ is Borel in $Z'$, then the statement is $\cosig_{1}^{1}$.
\end{Claim}
\begin{proof}
    It is just enough to look at the complexity of ``$\Vdash_{\Sbb}x\not\sqsubset\dot{h}$". This is equivalent to say that ``for every $p\in\Sbb$ and $l<\omega$ there are positive rational numbers $r,\varepsilon$ and $k,m,n<\omega$ such that $p$ is compatible with $p^{\dot{h}}(m,n)$, $B(k,r)\cap\{z\in Z:x\sqsubset_l z\}=\emptyset$ and $d(k,k^{\dot{h}}(m,n))<r-2^{-(m+1)}-\varepsilon$" where $B(k,r):=\{z\in Z:d^*(k,z)<r\}$. This statement can be written in the form $\forall p\in Z(p\notin\Sbb\text{\ or }\Theta(p,x,\dot{h}))$ where $\Theta(p,x,\dot{h})$ is $\cosig^1_1$ (recall that compatibility in $\Sbb$ is a $\cosig^1_1$-relation in $Z$ and that ``$B(k,r)\cap\{z\in Z:x\sqsubset_l z\}=\emptyset$" is also $\cosig^1_1$). Hence, as ``$p\in\Sbb$" is $\sig^1_1$, the whole statement is $\cosig^1_1$. On the other hand, as discussed before the claim, ``$\dot{h}$ is a $\Sbb$-name for a member of $Z$" is a conjunction of a $\sig^1_1$-statement with a $\cosig^1_1$-statement.
\end{proof}
In $M$, fix a $\Sbb$-name $\dot{h}$ for a real in $Y$ and a countable $H\subseteq Y$ that witnesses the goodness of $\Sbb$ for $\dot{h}$. Enumerate $H=\{y_{n}:n <\omega\}$. Now, as $\Psi(\dot{h},\la y_n:n<\omega\ra)$ is true in $M$, it is also true in $N$ by Claim \ref{claim1} and $\cosig^1_1$-absoluteness. In $N$, as $c$ is $\Rcal$-unbounded over $M$, $\forall{n<\omega}(c\not\sqsubset y_n)$, so  $\Vdash^N_{\Sbb} c\not\sqsubset \dot{h}$.
\end{proof}

\begin{lemma}[{\cite[Lemma 11]{BrF}}]\label{3.25}
   Assume that $\Pbb\in M$ is a poset. Then, in $N$,  $\Pbb$ forces that every $c\in X^N$ that is $\Rcal$-unbounded over $M$ is $\Rcal$-unbounded over $M^{\Pbb}$.
\end{lemma}
\begin{proof} Work within $M$. Let $e\in \Omega$ and $\dot{h}$ be a $\Pbb$-name for a member of $Y_e$. Fix $p\in \Pbb$ and $n<\omega$. Choose a continuous and surjective function $f:\omega^\omega\to Y_e$ and a $\Pbb$-name $\dot{z}$ for a real in $\omega^\omega$ such that $\Pbb$ forces that $f(\dot{z})=\dot{h}$. Choose an interpretation $(\la p_k\ra_{k<\omega},g)$ of $\dot{z}$ below $p$. In $N$, as $c$ is $\Rcal$-unbounded over $M$, then $c\not\sqsubset f(g)$, so $c\not\sqsubset_n f(g)$. By Lemma \ref{3.10.2.1}, there is a $k<\omega$ such that $p_k\Vdash^{N}_\Pbb c \not\sqsubset_n f(\dot{z})=\dot{h}$.
\end{proof}

Let $\Pbb_0,\Pbb_1,\Qbb_0,\Qbb_1$ be partial orders with $\Pbb_0,\Pbb_1\in M$. Recall that $\la\Pbb_0,\Pbb_1,\Qbb_0,\Qbb_1\ra$ is \emph{correct with respect to $M$} if $\Pbb_0$ is a complete subposet of $\Pbb_1$, $\Qbb_0$ is a complete subposet of $\Qbb_1$, $\Pbb_i\lessdot_M\Qbb_i$ for each $i<2$ and, whenever $p_0\in\Pbb_0$ is a reduction of $p_1\in\Pbb_1$, then $p_0$ is a reduction of $p_1$ with respect to $\Qbb_0,\Qbb_1$ (see \cite[Def. 2.8]{mejia-temp}).

\begin{lemma}[{\cite[Lemma 5.14]{mejia-temp}}]\label{presunbdirlimit}
   Let $I \in M$ be a directed partial order, $\langle \Pbb_i\rangle_{i\in I}\in M$ and $\langle \Qbb_i\rangle_{i\in I}\in N$ directed systems of posets such that, for any $i,j\in I$, $\Pbb_i\lessdot_{M}\Qbb_i$ and $\la\Pbb_i,\Pbb_j,\Qbb_i,\Qbb_j\ra$ is correct with respect to $M$ whenever $i\leq j$. Assume that $c\in X^N$ is $\Rcal$-unbounded over $M$ and that, for each $i\in I$, $\Qbb_i$ forces (in $N$) that $c$ is $\Rcal$-unbounded over $M^{\Pbb_i}$. If $\Pbb=\limdir_{i\in I}\Pbb_i$ and $\Qbb=\limdir_{i\in I}\Pbb_i$ then $\Pbb\lessdot_M\Qbb$, $\la\Pbb_i,\Pbb,\Qbb_i,\Qbb\ra$ is correct with respect to $M$ for each $i\in I$, and $\Qbb$ forces (in $N$) that $c$ is $\Rcal$-unbounded over $M^{\Pbb}$.
\end{lemma}
\begin{proof}
By \cite[Lemma 2.15]{mejia-temp}, $\Pbb\lessdot_M\Qbb$ and $\la\Pbb_i,\Pbb,\Qbb_i,\Qbb\ra$ is correct with respect to $M$ for each $i\in I$.

Work within $M$. Let $e\in \Omega$ and $\dot{h}$ be a $\Pbb$-name for a member in $Y_e$. Choose a continuous and surjective function $f:\omega^\omega\to Y_e$ and a $\Pbb$-name $\dot{z}$ for a real in $\omega^\omega$ such that $\Pbb$ forces that $f(\dot{z})=\dot{h}$.  Work in $N$. Assume, towards a contradiction, that there are $q\in \Qbb$ and $n<\omega$ such that $q\Vdash^{N}_{\Qbb}c\sqsubset_{n}\dot{h}$. Choose $i\in I$ such that $q\in \Qbb_i$.

Let $G$ be $\Qbb_i$-generic over $N$ such that $q\in G$. By assumption, $\Vdash^{N[G]}_{\Qbb/\Qbb_i}c\sqsubset_n \dot{h}$. In $M[G\cap \Pbb_i]$, find and interpretation $(\la p_{k}\ra_{k<\omega},g)$ of $\dot{z}$ in $\Pbb/\Pbb_i$. In $N[G]$, as $c\not\sqsubset f(g)$, by Lemma \ref{3.10.2.1} there is a $k<\omega$  such that $p_k\Vdash^{N[G]}_{\Qbb/\Qbb_i}c \not\sqsubset_{n} f(\dot{z})$ (by \cite[Lemma 2.13]{mejia-temp} $\Pbb/\Pbb_i\lessdot_{M[G\cap \Pbb_i]}\Qbb/\Qbb_i$, so the previous interpretation of $\dot{z}$ in $\Pbb/\Pbb_i$ is also an interpretation in $\Qbb/\Qbb_i$). Thus $p_{k}\Vdash^{N[G]}_{\Qbb/\Qbb_i}c \not\sqsubset_{n}\dot{h}$ which contradicts $\Vdash^{N[G]}_{\Qbb/\Qbb_i}c\sqsubset_n\dot{h}$.
\end{proof}

\begin{Corol}\label{3.26}
   Let $\delta$ be an ordinal, $c\in X^N$, $\langle\mathbb{P}_{\alpha}^{0},\dot{\mathbb{Q}}_{\alpha}^{0}\rangle_{\alpha<\delta}\in M$ and $\langle\mathbb{P}_{\alpha}^{1},\dot{\mathbb{Q}}_{\alpha}^{1}\rangle_{\alpha<\delta}\in N$ both FS iterations such that, for any $\alpha<\delta$, if $\mathbb{P}_{\alpha}^{0}\lessdot_{M}\mathbb{P}_{\alpha}^{1}$ then, in $N$, $\mathbb{P}_{\alpha}^{1}$ forces that $\dot{\mathbb{Q}}_{\alpha}^{0}\lessdot_{M^{\mathbb{P}_{\alpha}^{0}}}\dot{\mathbb{Q}}_{\alpha}^{1}$ and that $\dot{\mathbb{Q}}^1_\alpha$ forces that $c$ is $\Rcal$-unbounded over $M^{\Pbb^0_{\alpha+1}}$. Then,  $\mathbb{P}_{\alpha}^{0}\lessdot_{M}\mathbb{P}_{\alpha}^{1}$ for all $\alpha\leq\delta$ and $\Pbb^1_\delta$ forces, in $N$, that $c$ is $\Rcal$-unbounded over $M^{\Pbb^0_\delta}$.
\end{Corol}
\begin{proof}
   Fix $\alpha\leq\delta$. When $\Pbb^0_\alpha\lessdot_M\Pbb^1_\alpha$ it is not hard to see by induction on $\beta\in[\alpha,\delta]$ that $\la\Pbb^0_\alpha,\Pbb^0_\beta,\Pbb^1_\alpha,\Pbb^1_\beta\ra$ is correct with respect to $M$ (the limit step follows by Lemma \ref{presunbdirlimit}, the successor step by \cite[Lemma 13]{BrF}). The result follows by the case $\alpha=0$ and by Lemma \ref{presunbdirlimit}.
\end{proof}


\section{Consistency results}\label{4.2}\label{SecMain}


Before we prove the consistency results of this section, we review some facts about matrix iterations, including one result about preservation in the context of general Polish relational systems.

\newcommand{\mbf}{\mathbf{m}}
\newcommand{\Qnm}{\dot{\mathbb{Q}}}

\begin{definition}[Blass and Shelah {\cite{B1S}}]\label{4.1.1}
A \textit{matrix iteration} $\mbf$ consists of 
\begin{enumerate}[(I)]
\item two ordinals $\gamma^\mbf$ and $\delta^\mbf$,
\item for each $\alpha\leq\gamma^\mbf$, a FS iteration $\Pbb^\mbf_{\alpha,\delta^\mbf}=\la\Pbb^\mbf_{\alpha,\xi},\Qnm^\mbf_{\alpha,\xi}:\xi<\delta^\mbf\ra$ such that, for any $\alpha\leq\beta\leq\gamma^\mbf$ and $\xi<\delta^\mbf$, if $\Pbb^\mbf_{\alpha,\xi}\lessdot\Pbb^\mbf_{\beta,\xi}$ then $\Pbb^\mbf_{\beta,\xi}$ forces $\Qnm^\mbf_{\alpha,\xi}\lessdot_{V^{\Pbb^\mbf_{\alpha,\xi}}}\Qnm^\mbf_{\beta,\xi}$.
\end{enumerate}
According to this notation, $\Pbb^\mbf_{\alpha,0}$ is the trivial poset and $\Pbb^\mbf_{\alpha,1}=\Qnm^\mbf_{\alpha,0}$. 
By Corollary \ref{3.26}, $\Pbb^\mbf_{\alpha,\xi}$ is a complete suborder of $\Pbb^\mbf_{\beta,\xi}$ for all $\alpha\leq\beta\leq\gamma$ and $\xi\leq\delta$.

We drop the upper index $\mbf$ when it is clear from the context. If $G$ is $\mathbb{P}_{\gamma,\delta}$-generic over $V$ we denote  $V_{\alpha,\xi}=V[G\cap\Pbb_{\alpha,\xi}]$ for all $\alpha\leq\gamma$ and $\xi\leq\delta$ . Clearly, $V_{\alpha,\xi}\subseteq V_{\beta,\eta}$ for all $\alpha\leq\beta\leq\gamma$ and $\xi\leq\eta\leq\delta$. The idea of such a construction is to obtain a matrix $\langle V_{\alpha,\xi}:\alpha\leq\gamma,\xi\leq\delta\rangle$ of generic extensions as illustrated in Figure \ref{matrix}.
\end{definition}


\begin{figure}
\begin{center}
\includegraphics[scale=1]{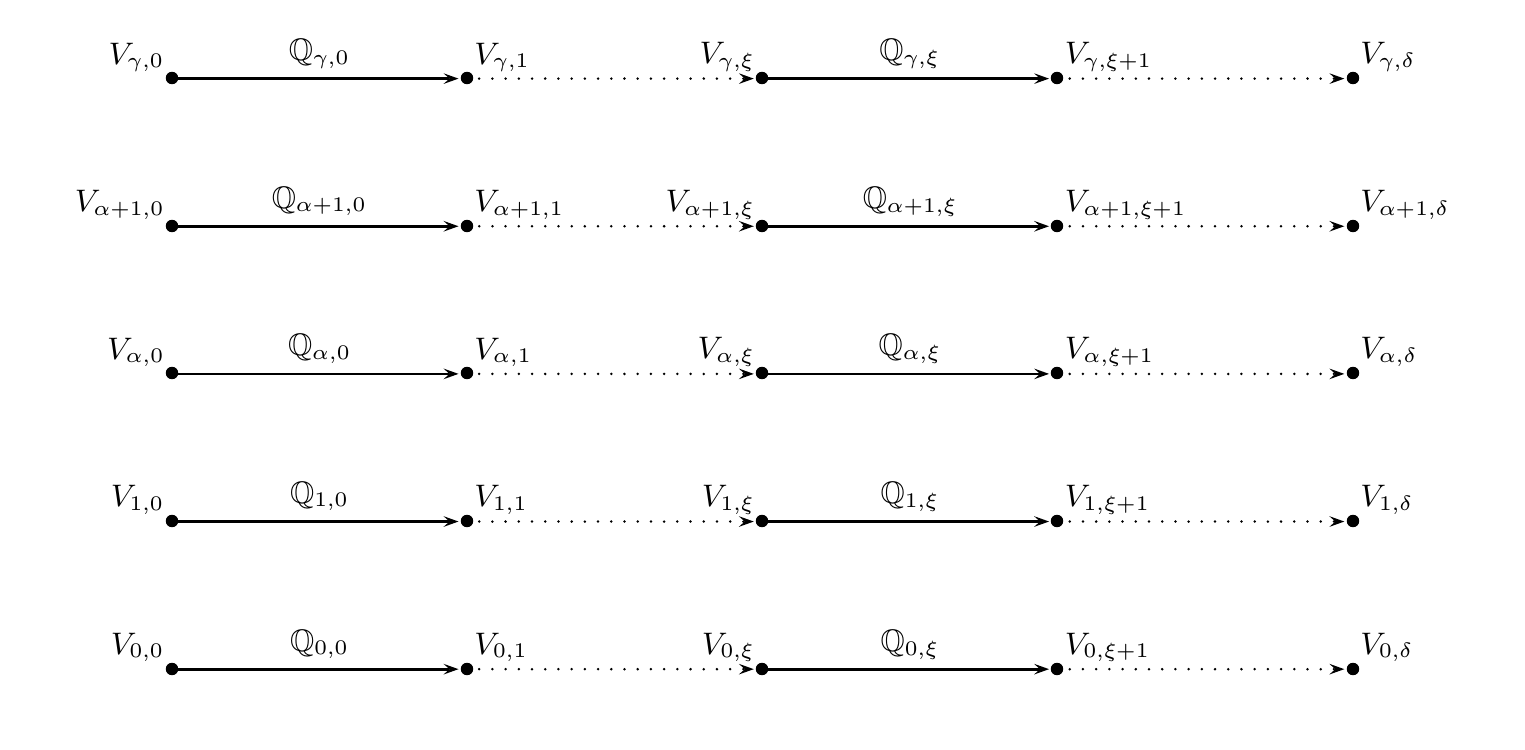}
\caption{Matrix iteration}
\label{matrix}
\end{center}
\end{figure}


The construction of the matrix iterations in our consistency results corresponds to the following particular case.

\newcommand{\Tnm}{\dot{\mathbb{T}}}

\begin{definition}\label{Defstandmat}
A matrix iteration $\mbf$ is \textit{standard} if
\begin{enumerate}[(I)]
 \item each $\Pbb^\mbf_{\alpha,1}$ ($\alpha\leq\gamma^\mbf$) is ccc,
 \item it consists, additionally, of
  \begin{enumerate}[(i)]
      \item a partition $\la S^\mbf,T^\mbf\ra$ of $[1,\delta^\mbf)$,
      \item a function $\Delta^{\mathbf{m}}:T^{\mathbf{m}}\to\{\alpha\leq\gamma^{\mathbf{m}}:\alpha\text{\ is not limit}\}$,
      \item a sequence $\langle \Sbb_{\xi}^{\mathbf{m}}:\xi\in S^{\mathbf{m}}\rangle$ of Suslin ccc posets coded in the ground model $V$,
      \item a sequence $\la\Tnm^\mbf_\xi:\xi\in T^\mbf\ra$ such that each $\Tnm^\mbf_\xi$ is a $\Pbb^\mbf_{\Delta^\mbf(\xi),\xi}$-name of a poset such that it is forced by $\Pbb^\mbf_{\gamma,\xi}$ to be ccc, and
  \end{enumerate}
\item for each $\alpha\leq\gamma$ and $1\leq\xi<\delta$,

\[\dot{\mathbb{Q}}_{\alpha,\xi}^{\mathbf{m}}:=
\begin{cases}
    (\Sbb^\mbf_\xi)^{V_{\alpha,\xi}} & \textrm{if $\xi\in S^\mbf$,}\\
    \Tnm^\mbf_\xi & \textrm{if $\xi\in T^\mbf$ and $\alpha\geq\Delta^\mbf(\xi)$,}\\
    \mathds{1}         & \textrm{otherwise.}
\end{cases}\]
\end{enumerate}
In practice, $\la\Pbb^\mbf_{\alpha,\xi}:\alpha\leq\gamma\ra$ is constructed by induction on $\xi\leq\delta$ (before the $\xi+1$-th step for $\xi\in T^\mbf$, $\Tnm_\xi^\mbf$ is defined). So it is clear that the constructed system $\mbf$ is a matrix iteration. Moreover, each $\Pbb^\mbf_{\alpha,\xi}$ is ccc. Again, when there is no place for confusion, we omit the upper index $\mathbf{m}$.
\end{definition}


\begin{lemma}[{\cite[Lemma 5]{BrF}}, see also {\cite[Cor. 3.9]{FFMM}}]\label{4.3}
Let $\mbf$ be a standard matrix iteration. Assume that
\begin{enumerate}[(i)]
   \item $\gamma_0\leq\gamma^\mbf$ has uncountable cofinality and
   \item $\Pbb_{\gamma_0,1}=\limdir_{\alpha<\gamma_0}\Pbb_{\alpha,1}$
\end{enumerate}
Then, for any $\xi\leq\delta^\mbf$, $\Pbb_{\gamma_0,\xi}=\limdir_{\alpha<\gamma_0}\Pbb_{\alpha,\xi}$. In particular, for any Polish space $X$ coded in $V$ (by a countable metric space), $X^{V_{\gamma_0,\xi}}=\bigcup_{\alpha<\gamma_0}X^{V_{\alpha,\xi}}$.
\end{lemma}


\begin{theorem}[{\cite[Thm. 10 \& Cor. 1]{M}}]\label{4.4}
   Let $\mbf$ be a standard matrix iteration and $\Rcal=\la X,Y,\sqsubset\ra$ a gPrs coded in $V$. Assume that, for each $\alpha<\gamma$,
   \begin{enumerate}[(i)]
       \item $\Pbb_{\alpha+1,1}$ adds a real $\dot{c}_\alpha\in X$ that is $\Rcal$-unbounded over $V_{\alpha,1}$ and
       \item for each $\xi\in S^\mbf$, $\Pbb_{\alpha,\xi}$ forces that $\Sbb_\xi^{V_{\alpha,\xi}}$ is $\Rcal$-good.
   \end{enumerate}
   Then $\Pbb_{\alpha+1,\delta}$ forces that $\dot{c}_{\alpha}$ is $\Rcal$-unbounded over $V_{\alpha,\delta}$. Even more, if $\gamma$ has uncountable cofinality and $\Pbb_{\gamma,1}=\limdir_{\alpha<\gamma}\Pbb_{\alpha,1}$, then $\mathbb{P}_{\gamma,\delta}$ forces $\bfrak(\Rcal)\leq\mathrm{cf}(\gamma)\leq \dfrak(\Rcal)$.
\end{theorem}
\begin{proof}
The first statement is a direct consequence of Lemmas \ref{3.24}, \ref{3.25} and Corollary \ref{3.26}.
For the second statement, given an increasing cofinal sequence $\{\alpha_{\zeta}:\zeta<\mathrm{cf}(\gamma)\}\in V$ in $\gamma$, by Lemma \ref{4.3} $\Pbb_{\gamma,\delta}$ forces that $\{\dot{c}_{\alpha_{\zeta}}:\zeta<\mathrm{cf}(\gamma)\}$ is strongly $\mathrm{cf}(\gamma)$-$\Rcal$-unbounded of size $\cf(\gamma)$, so $\bfrak(\Rcal)\leq\cf(\gamma)\leq\dfrak(\Rcal)$ by Lemma \ref{bRdR}. 
\end{proof}

For the reader convenience, before we prove our main results we summarize some facts about preservation from the previous sections. We assume that $b,h,\pi,\rho\in\omega^\omega$ are increasing.
\begin{enumerate}[({P}1)]
    \item If, for all but finitelly many $k<\omega$, $k\cdot\pi(k)\leq h(k)$ and $k\cdot|[b(k-1)]^{\leq k}|^k\leq\rho(k)$ then $\LOCbb^h_b(R)$ is $(\rho,\pi)$-linked for any $R\subseteq\prod b$ (Lemma \ref{Loclinked}).
    \item If $b>^*\pi\id_\omega$ then $\Ebb_b(S)$ is $((\id_\omega)^{\id_\omega},\pi)$-linked for any $S\subseteq\Scal(b,1)$ (Lemma \ref{linked}).
    \item Any $(h,b^{h^{\id_\omega}})$-linked poset is both $2$-$\aLc^*(b,h)$-good and $2$-$\Lc^*(b,h)$-good (Lemma \ref{3.20}).
    \item If $\theta$ is uncountable regular and $\mu<\theta$ is an infinite cardinal, then any $\mu$-centered poset is both $\theta$-$\aLc^*(b,h)$-good and $\theta$-$\Lc^*(b,h)$-good (Lemma \ref{3.21}).
    \item $\blc_{b,h}\leq\bfrak(\Lc^*(b,h))$ and $\dfrak(\Lc^*(b,h))\leq\dlc_{b,h}$ (Lemma \ref{Slmcard}).
    \item $\balc_{b,h^{\id_\omega}}\leq\bfrak(\aLc^*(b,h))$ and $\dfrak(\aLc^*(b,h))\leq\dalc_{b,h^{\id_\omega}}$ (Lemma \ref{Slmcard}).
\end{enumerate}

In relation to (P1) we consider the following structure discussed in \cite[Sect. 6]{BrM}  (see Lemma \ref{opRrho*}(d) below). Let $H,g\in\omega^\omega$ such that $H>\id_\omega$. Define $\Rbb^g_H:=\{x\in\omega^\omega:\forall k<\omega(H^{(k)}\circ x\leq^* g)\}$ where $H^{(0)}=\id_\omega$ and $H^{(k+1)}=H\circ H^{(k)}$, and denote $\Dbf(H,g):=\la\Rbb^g_H,\Rbb^g_H,\leq^*\ra$.

\begin{lemma}[Brendle and Mej\'ia {\cite[Lemma 6.5]{BrM}}]\label{genbounding}
    $\Dbf(H,g)\eqT\Dbf$ whenever $\Rbb^g_H\neq\emptyset$.
\end{lemma}

We use the following particular case. Fix the operation $\sigma^*:\omega\times\omega\to\omega$ such that $\sigma^*(m,0)=1$ and $\sigma^*(m,n+1)=m^{\sigma^*(m,n)}$. Define $\rho^*\in\omega^\omega$ such that $\rho^*(0)=2$ and $\rho^*(i+1)=\sigma(\rho^*(i),i+3)$. Denote $\Rbb^{\rho^*}:=\Rbb^g_H$ where $g(i):=\rho^*(i+1)$ and $H(i):=2^{i}$, and put $\Dbf^{\rho^*}:=\Dbf(H,g)$.

\begin{lemma}[cf. {\cite[Lemma 6.4]{BrM}}]\label{opRrho*}
\begin{enumerate}[(a)]
    \item $\id_\omega,\rho^*\in\Rbb^{\rho^*}$.
    \item $\Dbf^{\rho^*}\eqT\Dbf$.
    \item If $x,y\in\Rbb^{\rho^*}$ then $x+y$, $x\cdot y$, $x^y$ are in $\Rbb^{\rho^*}$.
    \item For any $x\in\Rbb^{\rho^*}$, $\forall^\infty i<\omega(i\cdot\big|[x(i-1)]^{\leq i}\big|^i\leq\rho^*(i))$.
\end{enumerate}
\end{lemma}

Now, we are ready to prove the main results of this paper.

\begin{theorem}\label{main}
Let $\mu\leq\nu\leq\kappa$ be uncountable regular cardinals and let $\lambda\geq\kappa$ be a cardinal such that $\lambda^{<\mu}=\lambda$. If
\begin{enumerate}[(I)]
    \item $\pi$ is an ordinal with $|\pi|\leq\lambda$,
    \item $\la\theta_\zeta:\zeta<\pi\ra$ is a non-decreasing sequence of regular cardinals in $[\mu,\nu]$,
    \item $\lambda^{<\theta_\zeta}=\lambda$ for all $\zeta<\pi$, and
    \item $\zeta^*\leq\pi$,
\end{enumerate} 
then there is a ccc poset forcing
\begin{enumerate}[(A)]
\item $\mathrm{MA}_{<\mu}$,  $\add(\Ncal)=\add(\Ical_{f'})=\cov(\Ical_{\id_\omega})=\mu$ and $\non(\Ical_{\id_\omega})=\cof(\Ical_{f'})=\cfrak=\lambda$ for all increasing $f'\in\omega^\omega$, and
\item there are sets $\{c_\zeta:\zeta<\pi\}$, $\{h_\zeta:\zeta<\zeta^*\}$ and $\{f_\zeta:\zeta^*\leq\zeta<\pi\}$ of increasing functions in $\omega^\omega$ such that
\begin{enumerate}[({B}1)]
   \item $\blc_{c_\zeta,h_\zeta}=\theta_\zeta$ for all $\zeta<\zeta^*$,
   \item $\balc_{c_\zeta,1}=\cov(\Ical_{f_\zeta})=\balc_{c_\zeta,H}=\theta_\zeta$ for all $\zeta\in[\zeta^*,\pi)$ where $H:=(\rho^*)^{\id_\omega}$, and
   \item there is an increasing function $f_\pi$ such that $\cov(\Ical_f)=\supcov=\add(\Mcal)=\cof(\Mcal)=\nu$ and $\minnon=\non(\Ical_f)=\kappa$ for any increasing $f\geq^*f_\pi$.
\end{enumerate}
\end{enumerate}
\end{theorem}

We are allowed to use $\pi=0$ in this theorem, in which case we have Theorem \ref{mainA} (see the Introduction) and $f_\pi$ could be found in the ground model. When $\pi>0$, we additionally obtain many values for cardinals of the type $\blc_{b,h}$ and $\balc_{b,h}$, even allowing at most $\lambda$-many repetitions for each value (as the continuum is forced to be $\lambda$, no more that $\lambda$-many repeated values are allowed).

\begin{proof}
Fix a bijection $g=(g_0,g_1,g_2):\lambda\to\kappa\times(\pi+1)\times\lambda$, and consider $t:\kappa\nu\to\kappa$ such that $t(\kappa\delta+\alpha)=\alpha$ for all $\delta<\nu$ and $\alpha<\kappa$. The ccc poset required is of the form $\Dbb_{\pi+1}\ast\dot{\Pbb}$ where $\Dbb_{\pi+1}$ is the FS iteration of length $\pi+1$ of Hechler forcing $\Dbb$ and $\dot{\Pbb}$ is a $\Dbb_{\pi+1}$-name of a ccc poset constructed by a matrix iteration as defined in step 2 below. First, in the ground model, fix $g_{-1}:=\id_\omega^{\id_\omega}$ and $b_{-1}:=\max\{H+1,2^{g_{-1}\circ(H^+-1)}\}$. Note that $H,g_{-1},b_{-1}\in\Rbb^{\rho^*}$ by Lemma \ref{opRrho*} and, by Lemma \ref{2.8}, $\cov(\Ical_{\id_\omega})\leq\balc_{b_{-1},H}$ and $\dalc_{b_{-1},H}\leq\non(\Ical_{\id_\omega})$ (this will be used to show (A)).

\emph{Step 1.} For each $\zeta<\zeta^*$, let $h_\zeta\in\mathbb{R}^{\rho^*}\cap V^{\Dbb_{\zeta+1}}$ be an increasing dominating real over $\mathbb{R}^{\rho^*}\cap V^{\Dbb_{\zeta}}$ (recall that $\Dbf\eqT\Dbf^{\rho^*}$). In $V^{\Dbb_{\zeta+1}}$, choose $c_\zeta\in\Rbb^{\rho^*}$ such that $c_\zeta>h_\zeta^{\id_\omega}$ (this guarantees that $\Lc^*(c_\zeta,h_\zeta)$ is a gPrs).

For each $\zeta\in[\zeta^*,\pi]$, let $c_\zeta \in \omega^\omega\cap V^{\Dbb_{\zeta+1}}$ be an increasing dominating real over $V^{\Dbb_{\zeta}}$, and define $f_\zeta, g_\zeta,b_\zeta\in\omega^\omega\cap V^{\Dbb_{\zeta+1}}$, all increasing, such that $f_{\zeta}\geq (\log c_\zeta)^{+}$, $g_\zeta\gg f_\zeta$ and $b_\zeta\geq^* 2^{g_\zeta\circ(H^+-1)}$. Note that, By Lemmas \ref{2.7} and \ref{2.8}, $\balc_{c_\zeta,1}\leq\cov(\Ical_{f_\zeta})\leq\balc_{b_\zeta,H}$ and $\non(\Ical_{f_\zeta})\leq\dalc_{c_\zeta,1}$.

\emph{Step 2.} Work in $V_{0,0}:=V^{\Dbb_{\pi+1}}$. According to Definition \ref{Defstandmat}, construct a standard matrix iteration $\mathbf{m}$ that satisfies (i)-(viii) below.
\begin{enumerate}[(i)]
   \item $\gamma^{\mathbf{m}}=\kappa$ and $\delta^{\mathbf{m}}=\lambda\kappa\nu$ (as a product of ordinal numbers).

   \item $\Pbb_{\alpha,1}=\Cbb_\alpha$ for each $\alpha\leq\kappa$.

   \item $S=S^\mathbf{m}=\{\lambda\rho:0<\rho<\kappa\nu\}$ and $T=T^\mbf=[1,\delta^\mbf)\smallsetminus S$,

   \item $\Sbb_{\xi}=\Dbb$ for all $\xi\in S$.

   \item If $\xi=\lambda\rho+1$ for some $\rho<\kappa\nu$, put $\Delta^\mbf(\xi)=t(\rho)+1$ and let $\dot{\Tbb}_{\xi}$ be a $\Pbb_{t(\rho)+1,\xi}$-name for $(\Ebb_{c_\pi})^{V_{t(\rho)+1,\xi}}=\Ebb_{c_\pi}(\Scal(c_\pi,1)\cap V_{t(\rho)+1,\xi})$.
\end{enumerate}

For each $\alpha<\kappa$ and $\rho<\kappa\nu$,
\begin{enumerate}[(1)]
    \item let $\langle\dot{\Qbb}_{\alpha,\gamma}^{\rho}:\gamma<   \lambda\ra$ be an enumeration of \emph{all} the \emph{nice} $\Pbb_{\alpha,\lambda\rho}$-names for \emph{all} the posets which underlining set is a subset of $\mu$ of size $<\mu$ and $\Vdash_{\Pbb_{\kappa,\lambda\rho}}$``$\dot{\Qbb}_{\alpha,\gamma}^{\rho}$ is ccc" (possible because $|\Pbb_{\kappa,\lambda\rho}|\leq\lambda=\lambda^{<\mu}$) and

    \item for all $\zeta<\pi$, let $\langle \dot{F}_{\alpha,\zeta,\gamma}^{\rho}:\gamma<\lambda\rangle$ be an enumeration of \emph{all} the \emph{nice} $\Pbb_{\alpha,\lambda\rho}$-names for \emph{all} subsets of $\prod c_\zeta$  of size $<\theta_\zeta$.
\end{enumerate}
If $\xi=\lambda\rho+2+\varepsilon$ for some $\rho<\kappa\nu$ and $\varepsilon<\lambda$, put $\Delta^\mbf(\xi)=g_0(\varepsilon)+1$,
\begin{enumerate}[(i)]
    \setcounter{enumi}{5}
    \item whenever $g_1(\varepsilon)=\pi$ put $\dot{\Tbb}_{\xi}=\dot{\Qbb}_{g_0(\varepsilon),g_2(\varepsilon)}^{\rho}$,

    \item whenever $g_1(\varepsilon)<\zeta^*$ put $\dot{\Tbb}_{\xi}=\LOCbb_{c_{g_1(\varepsilon)}}^{h_{g_1(\varepsilon)}}(\dot{F}_{g(\varepsilon)}^{\rho})$, and

    \item whenever $\zeta^*\leq g_1(\varepsilon)<\pi$ put $\dot{\Tbb}_{\xi}=\Ebb_{c_{g_1(\varepsilon)}}(\dot{F}_{g(\varepsilon)}^{\rho})$.
\end{enumerate}

Put $\Pbb:=\Pbb_{\kappa,\lambda\kappa\nu}. $We prove that $V_{\kappa,\lambda\kappa\nu}$ satisfies the statements of this theorem.

\noindent\textbf{(A)} We first show that, for each $0\xi<\lambda\kappa\nu$, $\Pbb_{\kappa,\xi}$ forces that $\Qnm_{\kappa,\xi}$ is $\mu$-$\Slm(b_{-1},\rho^*)$-good. The case $\xi=0$ follows by (P4) ($\Qnm_{\kappa,0}=\Cbb_\kappa$ is a FS iteration of countable posets); when $\xi=\lambda\rho$ for some $\rho<\kappa\nu$, it is clear by (P4); when $\xi=\lambda\rho+2+\varepsilon$ for some $-1\leq\varepsilon<\lambda$, we split into three subcases: when $g_1(\varepsilon)=\pi$ it is clear by (P4); when $g_1(\varepsilon)<\zeta^*$, as $h_\zeta\geq^*\id_\omega b_{-1}^{H}$ (because $h_\zeta\in\Rbb^{\rho^*}$ is dominating), it follows by (P1), (P3) and Lemma \ref{opRrho*}; and when either $\varepsilon=-1$ or $\zeta^*\leq g_1(\varepsilon)<\pi$, as $c_\pi,c_\zeta>^*\id_\omega b_{-1}^{H}$, it follows by (P2) and (P3).

Therefore, by Theorem \ref{3.13} and (P6), $\Pbb$ forces $\balc_{b_{-1},H}\leq\bfrak(\Slm(b_{-1},\id_{\omega}^{\id_{\omega}}))\leq \mu$ and $\lambda\leq \dfrak(\Slm(b_{-1},\id_{\omega}^{\id_{\omega}}))\leq\dalc_{b_{-1},H}$. As $\cov(\Ical_{\id_\omega})\leq\balc_{b_{-1},H}$ and $\dalc_{b_{-1},H}\leq\non(\Ical_{\id_\omega})$, $\Pbb$ forces $\add(\Ncal)\leq\add(\Ical_{f'})\leq\cov(\Ical_{\id_\omega})\leq \mu$ and $\lambda\leq\non(\Ical_{\id_\omega})\leq\cof(\Ical_{f'})\leq\cof(\Ncal)$ for any increasing $f'\in \omega^\omega$. On the other hand, since $|\Pbb|=\lambda$, $\Pbb$ forces $\cfrak=\lambda$.

It remains to show that $\mathrm{MA}_{<\mu}$ holds in $V_{\kappa,\lambda\kappa\nu}$  (which implies $\add(\Ncal)\geq\mu$). 
Let $\Qbb$ be a ccc poset of size $<\mu$, wlog its underlining set is a subset of $\mu$, and let $\mathcal{D}$ be a family of size $<\mu$ of dense subsets of $\Qbb$. By Lemma \ref{4.3}, $\Qbb,\mathcal{D}\in V_{\alpha,\lambda\rho}$ for some $\alpha<\kappa$ and $\rho<\kappa\nu$. As $\Qbb$ is ccc in $V_{\kappa,\lambda\rho}$, there is some $\gamma<\lambda$ such that $\Qbb=\Qbb^\rho_{\alpha,\gamma}=\mathbb{T}_\xi$ where $\xi=\lambda\rho+2+\varepsilon$ and $\varepsilon=g^{-1}(\alpha,\pi,\gamma)$. It is clear that, in $V_{\alpha+1,\xi+1}$, there is a $\Qbb$-generic set over $V_{\alpha+1,\xi}$, so this generic set intersects all the members of $\mathcal{D}$.

\noindent\textbf{(B1)} Fix $\zeta<\zeta^*$. For the inequality $\bfrak_{c_{\zeta},h_\zeta}^{\Lc}\geq\theta_\zeta$: Let $F\subseteq\prod c_\zeta\cap V_{\kappa,\lambda\kappa\nu}$ be a family of size $<\theta_\zeta$. By Lemma \ref{4.3}, there are $\alpha<\kappa$ and $\rho<\kappa\nu$ such that $F\in V_{\alpha,\lambda\rho}$, so there is some $\gamma<\lambda$ such that $F=F_{\alpha, \zeta, \gamma}^{\rho}$. Hence, the generic slalom added by $\LOCbb_{c_{g_1(\varepsilon)}}^{h_{g_1(\varepsilon)}}(F_{\alpha, \zeta, \gamma}^{\rho})=\mathbb{T}_\xi$, where $\xi=\lambda\rho+2+\varepsilon$ and $\varepsilon=g^{-1}(\alpha,\zeta,\gamma)$, localizes all the reals in $F$.

For the converse, we first show that the iterands in the FS iteration $\langle\Pbb_{\kappa,\xi},\Qnm_{\kappa,\xi}:{\xi<\lambda\kappa\nu}\rangle$ are $\theta_\zeta$-$\Lc^{*}(c_\zeta,h_\zeta)$-good posets. Indeed, as $\Dbb$ is $\sigma$-centered, by (P4), $\Sbb_\xi$ is $\Lc^*(c_{\zeta},h_\zeta)$-good for each $\xi\in S$; as $c_{\pi}>^* c_{\zeta}^{h_{\zeta}^{\id_{\omega}}}\id_{\omega}$ (because $c_\pi$ is dominating over $V^{\Dbb_{\pi}}$ and $c_{\zeta}^{h_{\zeta}^{\id_{\omega}}}\id_{\omega}\in V^{\Dbb_{\zeta+1}}$), for each $\rho<\kappa\nu$, $\Pbb_{\kappa,\lambda\rho+1}$ forces that $\Tnm_{\lambda\rho+1}$ is $(h_\zeta,c_{\zeta}^{h_{\zeta}^{\id_{\omega}}})$-linked by (P2) so, by (P3), it is forced to be $2$-$\Lc^*(c_{\zeta},h_\zeta)$-good; now we analyse the iterands from (vi)-(viii): for (vi), as $\Tnm_\xi$ has size $<\mu\leq\theta_\zeta$, $\Tnm_{\xi}$ is $\theta_\zeta$-$\Lc^*(c_\zeta,h_\zeta)$-good by (P4); for (vii), when $g_1(\varepsilon)\leq \zeta$,  $\Tnm_{\xi}$ is $\theta_\zeta$-$\Lc^*(c_\zeta,h_\zeta)$-good by (P4) because it has size $<\theta_{\zeta}$, else, when $\zeta<g_1(\varepsilon)$, as $h_{g_1(\varepsilon)}$ is dominating over $\Rbb^{\rho^*}\cap V^{\Dbb_{\zeta+1}}$ and $c_{\zeta}^{h_{\zeta}^{\id_{\omega}}}\id_{\omega}\in \Rbb^{\rho^*}\cap V^{\Dbb_{\zeta+1}}$, $h_{g_1(\varepsilon)}>^*  c_{\zeta}^{h_{\zeta}^{\id_{\omega}}}\id_{\omega}$ so, by Lemma \ref{opRrho*}, (P1) and (P3),
$\dot{\Tbb}_{\xi}$ is $2$-$\Lc^*(c_\zeta,h_\zeta)$-good; finally, for (viii), as $g_1(\varepsilon)>\zeta$, $c_{g_1(\varepsilon)}$ is dominating over $V^{\Dbb_{\zeta+1}}$, so $c_{g_1(\varepsilon)}>^* c_{\zeta}^{h_{\zeta}^{\id_{\omega}}}\id_{\omega}$, and $\dot{\Tbb}_{\zeta}$ is $2$-$\Lc^*(c_\zeta,h_\zeta)$-good by (P2) and (P3).

Therefore, $\Pbb$ forces $\bfrak_{c_{\zeta},h_\zeta}^{\Lc}\leq \bfrak(\Lc^*(c_\zeta,h_\zeta))\leq \theta_\zeta$ by Theorem \ref{3.13} and (P5).

\noindent\textbf{(B2)} Fix $\zeta^*\leq \zeta<\pi$. It is enough to show $\theta_\zeta\leq\balc_{c_\zeta,1}$ and $\balc_{c_\zeta,H}\leq\theta_\zeta$. The former inequality is proved by a similar argument as for $\theta_\zeta\leq\bfrak_{c_{\zeta},h_\zeta}^{\Lc}$ in (B1). For the latter, we show that the posets we use in the $\kappa$-th FS iteration are $\theta_\zeta$-$\Slm(c_\zeta,\rho^*)$-good. For each $\xi\in S$, $\Sbb_\xi$ is $\Slm(c_\zeta,\rho^*)$-good by (P4) because $\Dbb$ is $\sigma$-centered; by (P2), as $c_{\pi}>^* c_{\zeta}^H\id_{\omega}$, 
for each $\rho<\kappa\nu$, $\Pbb_{\kappa,\lambda\rho+1}$ forces that $\Tnm_{\lambda\rho+1}$ is $(\rho^*,c_\zeta^H)$-linked (note that $\rho^*\geq^* \id_{\omega}^{\id_\omega}$) so, by (P3), it is forced to be $2$-$\Slm(c_\zeta,\rho^*)$-good; for $\rho<\kappa\nu$ and $\varepsilon<\lambda$ we have following cases: when $g_1(\varepsilon)=\pi$ or $g_1(\varepsilon)\leq\zeta$,
as $|\Tnm_\xi|< \theta_{\zeta}$, $\Tnm_{\xi}$ is $\theta_\zeta$-$\Slm(c_\zeta,\rho^*)$-good by (P4); 
when $\zeta<g_1(\varepsilon)$, $c_{g_1(\varepsilon)}>^* c_{\zeta}^{H}\id_{\omega}$ because $c_{g_1(\varepsilon)}$ is dominating over $V^{\Dbb_{\zeta+1}}$ and $c_{\zeta}^{H}\id_{\omega}\in V^{\Dbb_{\zeta+1}}$, so, by (P2), $\dot{\Tbb}_{\xi}$ is $(\rho^*,c_\zeta^H)$-linked and, by (P3), it is forced to be $2$-$\Slm(c_\zeta,\rho^*)$-good.

Therefore, by Theorem \ref{3.13} and (P6), $\Pbb$ forces that $\balc_{c_\zeta,H}\leq\bfrak(\Slm(c_\zeta,\rho^*))\leq \theta_{\zeta}$.

\noindent\textbf{(B3)} Note that $\Pbb$, as a FS iteration of length $\lambda\kappa\nu$, adds a $\nu$-scale and $\nu$-cofinally many Cohen reals. Therefore, it forces $\add(\Mcal)=\cof(\Mcal)=\nu$.

For $\rho<\kappa\nu$ denote by $r^{\rho}\in V_{t(\rho)+1,\lambda\rho+2}\cap\prod c_{\pi}$ the generic real added by $\dot{\Qbb}_{t(\rho)+1,\lambda\rho+1}=(\Ebb_{c_\pi})^{V_{t(\rho)+1,\lambda\rho+1}}$ over $V_{t(\rho)+1,\lambda\rho+1}$. This real is eventually different from all the members of $ V_{t(\rho)+1,\lambda\rho+1}\cap\prod{c}_{\pi}$. Hence $\nu\leq \balc_{c_{\pi},1}$ is a consequence of the following.

\begin{Claim}\label{claim} In $V_{\kappa,\lambda\kappa\nu}$, for any $F\subseteq \prod c_{\pi}$ of size $<\nu$, there is some  $r^{\rho}$ eventually different from all the members of $F$.
\end{Claim}
\begin{proof}
By Lemma \ref{4.3} there are $\alpha<\kappa$ and $\delta<\kappa\nu$ such that  $F\subseteq V_{\alpha,\lambda\delta}$. By the definition of $t$, find a $\rho\in[\delta,\kappa\nu)$ such that $t(\rho)=\alpha$. Clearly $F\subseteq V_{\alpha,\lambda\rho}$, so their members are all eventually different from $r^{\rho}$.
\end{proof}

On the other hand, $\{r_{}^{\rho}:\rho<\kappa\nu\}$ is a family of reals of size $\leq\kappa$ and, by Claim \ref{claim}, any member of $V_{\kappa,\lambda\kappa\nu}\cap\prod{c_{\pi}} $ is eventually different from some $r_{}^{\rho}$. Hence $\dalc_{c_\pi,1}\leq \kappa$.

Fix $f\in\omega^\omega$ increasing such that $f\geq f_\pi$. Then $\nu\leq\balc_{c_\pi,1}\leq\cov(\Ical_{f_\pi})\leq\cov(\Ical_f)$ and $\kappa\geq\dalc_{c_\pi,1}\geq\non(\Ical_{f_\pi})\geq\non(\Ical_f)$, in fact, $\balc_{c_\pi,1}=\cov(\Ical_f)=\nu$ because $\cov(\Ical_f)\leq\supcov\leq\non(\Mcal)=\nu$.

To finish the proof it remains to show that $\kappa\leq\minnon$. For any $b'\in\omega^\omega$, as $\Dbb$ is $\sigma$-centered and thus $\Slm(b',1)$-good by (P4), $\kappa\leq\dfrak(\Slm(b',1))\leq\dalc_{b',1}$ by Theorem \ref{4.4} and (P6). Therefore, by Theorem \ref{2.9}, $\kappa\leq \minnon=\min\{\dalc_{b',1}:b'\in\omega^\omega\}$.
\end{proof}

\begin{remark}\label{remnoHechler}
   If the matrix iteration construction of Theorem \ref{main} is modified so that $S^\mbf=\emptyset$ (i.e. no use of Hechler forcing) then the cardinal invariants associated with many Yorioka ideals can still be separated. Concretely, the final model still satisfies (A), (B1) and (B2), and also satisfies $\cov(\Ical_f)=\non(\Mcal)=\nu$ and $\cov(\Mcal)=\non(\Ical_f)=\kappa$ for all increasing $f\geq^*f_\pi$. However, the values of $\bfrak$ and $\dfrak$ are unknown because it is unclear whether the restricted versions of $\LOCbb_{c_\zeta,h_\zeta}$ and $\Ebb_{c_\zeta}$ ($\zeta\leq\pi$) used in the construction add dominating reals.
\end{remark}

\begin{remark}\label{rem3D}
   In the hypothesis of Theorem \ref{main} assume, additionally, that $\mu'$ is a regular cardinal and $\mu\leq\mu'\leq\nu$ and, instead of $\kappa$ being regular, just assume that $\kappa^{<\mu'}=\kappa$. As in \cite{mejiavert}, the forcing construction in Theorem \ref{main} can be modified so that the matrix iteration allows vertical support restrictions, that is, $\Pbb_{A,\xi}$ can be defined for all $A\subseteq\kappa$ and $\xi\leq\pi$. The final model of this construction still satisfies (A), (B1) and (B2), and also satisfies $\cov(\Ical_f)=\supcov=\mu'$, $\add(\Mcal)=\cof(\Mcal)=\nu$ and $\minnon=\non(\Ical_f)=\kappa$ (the latter not necessarily regular) for all increasing $f\geq^*f_\pi$. Though this is stronger than Theorem \ref{main}, we do not get to separate more of the cardinal invariants associated with Yorioka ideals.
\end{remark}

\begin{remark}\label{RemRoth}
   In the context of \cite{BrM}, Theorem \ref{main} could be modified so that, in (B1), it can be forced that $\blc_{c_\zeta,h_\zeta}=\bfrak(\Ical_\zeta)=\theta_\zeta$ for $\zeta<\zeta^*$ where $\la\Ical_\zeta:\zeta<\zeta^*\ra$ is some sequence of gradually fragmented ideals on $\omega$ and $\bfrak(\Ical_\zeta)$ denotes the \emph{Rothberger number} of $\Ical_\zeta$.
\end{remark}

The next result guarantees the consistency of $\add(\Ncal)<\add(\Ical_f)<\cof(\Ical_f)<\cof(\Ncal)$ for any fixed $f$.

\begin{theorem}\label{addN<addIf}
Let $\mu\leq\nu\leq\kappa$ be uncountable regular cardinals and let $\lambda\geq\kappa$ be a cardinal such that $\lambda^{<\mu}=\lambda$. If $f\in\omega^\omega$ is increasing then there is a ccc poset forcing that $\add(\Ncal)=\mu$, $\add(\Ical_f)=\bfrak=\non(\Mcal)=\nu$, $\cov(\Mcal)=\dfrak=\cof(\Ical_f)=\kappa$ and $\cof(\Ncal)=\cfrak=\lambda$.
\end{theorem}
\begin{proof} Fix a function $b\in\omega^\omega$ such that $b\gg 2^f$, a bijection $g=(g_0,g_1):\lambda\to\kappa\times\lambda$, and fix $t:\kappa\nu\to\kappa$ such that $t(\kappa\delta+\alpha)=\alpha$ for $\delta<\nu$ and $\alpha<\kappa$. Put $h:=\id_\omega$.   According to Defintion \ref{Defstandmat}, construct a standard matrix iteration $\mathbf{m}$ such that:

\begin{enumerate}[(i)]
   \item $\gamma^{\mathbf{m}}=\kappa$ and $\delta^{\mathbf{m}}=\lambda\kappa\nu$,
   \item $\Pbb_{\alpha,1}=\Cbb_\alpha$ for each $\alpha\leq\kappa$,
   \item If $\xi=\lambda\rho>0$ for some $\rho<\kappa\nu$, put $\Delta^\mbf(\xi)=t(\rho)+1$ and let $\dot{\Tbb}_{\xi}$ be a $\Pbb_{t(\rho)+1,\xi}$-name for $\Dbb^{V_{t(\rho)+1,\xi}}$.
   \item If $\xi=\lambda\rho+1$ for some $\rho<\kappa\nu$, put $\Delta^\mbf(\xi)=t(\rho)+1$, and let $\dot{\Tbb}_{\xi}$ be a $\Pbb_{t(\rho)+1,\xi}$-name for $(\LOCbb_{b}^{h})^{V_{t(\rho)+1,\xi}}$.
\end{enumerate}

For each $\alpha<\kappa$ and $\rho<\kappa\nu$, let $\langle\dot{\Qbb}_{\alpha,\gamma}^{\rho}:\gamma<\lambda\ra$ be an enumeration of \emph{all} the \emph{nice} $\Pbb_{\alpha,\lambda\rho}$-names for \emph{all} the posets whose underlining set is a subset of $\mu$ of size $<\mu$ and $\Vdash_{\Pbb_{\kappa,\lambda\rho}}$``$\dot{\Qbb}_{\alpha,\gamma}^{\rho}$ is ccc" (possible because $|\Pbb_{\kappa,\lambda\rho}|\leq\lambda=\lambda^{<\mu}$).
\begin{enumerate}[(i)]
\setcounter{enumi}{4}
   \item If $\xi=\lambda\rho+2+\varepsilon$ for some $\rho<\kappa\nu$ and $\varepsilon<\lambda$, put $\Delta^\mbf(\xi)=g_0(\varepsilon)+1$ and  $\dot{\Tbb}_{\xi}=\dot{\Qbb}_{g(\varepsilon)}^{\rho}$.
\end{enumerate}

By (P1) we can find increasing $\pi, \rho \in \omega^\omega$ such that $\LOCbb_{b}^{h}(F)$ is $(\pi,\rho)$-linked for any $F\subseteq \prod b$. Therefore, by Lemma \ref{linkedpresadd(N)}, there is a $\leq^*$-increasing sequence $\mathcal{G}=\la g_n:n<\omega\ra$ such that $\LOCbb_{b}^{h}(F)$ is $\Lc(\omega, \mathcal{G})$-good for any $F\subseteq\prod b$. Also, $\Dbb$ is $\Lc(\omega, \mathcal{G})$-good (see Example \ref{Exam}) and $\dot{\Tbb}_{\xi}$ is $\mu$-$\Lc(\omega, \mathcal{G})$-good when $\xi=\lambda\rho+2+\varepsilon$ for some $\rho<\kappa\nu$ and $\varepsilon<\lambda$. Therefore, in $V_{\kappa,\lambda\kappa\mu}$, we have that $\add(\Ncal)\leq \mu$ and $\lambda\leq \cof(\Ncal)$ by Theorem \ref{3.13}. The converse inequalities are similar to the proof of (A) of Theorem \ref{main}.

We now show $\dfrak,\dfrak_{b,h}^{\Lc}\leq \kappa$ and $\bfrak, \bfrak_{b,h}^{\Lc}\geq \nu$. For each $\rho<\kappa\nu$ denote by $d^{\rho}\in V_{t(\rho)+1,\lambda\rho+1}\cap \omega^\omega$ the Hechler generic real added by $\Qnm_{t(\rho)+1,\lambda\rho}=\Qnm_{\kappa,\lambda\rho}=\Dbb^{V_{t(\rho)+1,\lambda\rho}}$ over $V_{t(\rho)+1,\lambda\rho}$, and by $\psi^{\rho}\in V_{t(\rho)+1,\lambda\rho+2}\cap\Scal(b,h)$ the generic slalom added by   $\Qnm_{t(\rho)+1,\lambda\rho+1}=\Qnm_{\kappa,\lambda\rho+1}=(\LOCbb_{b}^{h})^{V_{t(\rho)+1,\lambda\rho+1}}$ over $V_{t(\rho)+1,\lambda\rho+1}$.

\begin{Claim}\label{claimb}
In $V_{\kappa,\lambda\kappa\nu}$, each family of reals of size $<\nu$ is dominated by some $d^{\rho}$.
\end{Claim}
\begin{proof} Let $F$ be such a family. By Lemma \ref{4.3} there are $\alpha<\kappa$ and $\delta<\kappa\nu$ such that  $F\subseteq V_{\alpha,\lambda\delta}$. By the definition of $t$, find a $\rho\in[\delta,\kappa\nu)$ such that $t(\rho)=\alpha$. Clearly, $F\subseteq V_{\alpha,\xi}$ where $\xi=\lambda\rho$, so their members are dominated by $d^{\rho}$.
\end{proof}

As a direct consequence, $\nu\leq\bfrak$. On the other hand, $\{d_{}^{\rho}:\rho<\kappa\nu\}$ is a family of reals of size $\leq\kappa$ and, by Claim \ref{claimb}, any member of $V_{\kappa,\lambda\kappa\nu}\cap\omega^\omega $ is dominated by some $d_{}^{\rho}$. Hence $\dfrak\leq \kappa$. By a similar argument, we can prove:

\begin{Claim}
Each family of reals in $V_{\kappa,\lambda\kappa\nu}$ of size $<\nu$ is localized by some slalom $\psi^{\rho}$.
\end{Claim}

By Theorem \ref{2.4.1},  $\nu\leq\min\{\bfrak,\bfrak_{b,\id_{\omega}}^{\Lc}\}\leq \add(\Ical_f)$ and $\cof(\Ical_f)\leq \max\{\dfrak,\dfrak_{b,\id_\omega}^{\Lc}\}\leq\kappa$ (since $2^f\ll b$), so $\nu\leq \add(\Ical_f)$ and $\cof(\Ical_f)\leq \kappa$. On the other hand, as $\Pbb_{\kappa,\lambda\kappa\nu}$ is obtained by a FS iteration of cofinality $\nu$, $\Pbb_{\kappa,\lambda\kappa\nu}$ adds a $\nu$-$\Ed(\omega)$-unbounded family of Cohen reals of size $\nu$ by Lemma \ref{3.12}, so it forces $\non(\Mcal)=\bfrak(\Ed(\omega))\leq\nu$. Also, by Theorem \ref{4.4}, $\cov(\Mcal)=\dfrak(\Ed(\omega))\geq \kappa$. Hence, by Corollary \ref{2.11},  $\add(\Ical_f)\leq\cov(\Ical_f)\leq \non(\Mcal)\leq \nu$ and $\kappa\leq\cov(\Mcal)\leq \non(\Ical_f)\leq  \cof(\Ical_f)$.
\end{proof}

In the previous model, it is clear that $\minLc\leq\bfrak$ and $\dfrak\leq\supLc$, so $\add(\Ncal)=\minLc$ and $\cof(\Ncal)=\supLc$ by Theorem \ref{bLoc(b,h)addN}.

To finish this section, we show the consistency of $\bfrak<\minLc$ and $\dfrak<\supLc$. As the converse strict inequalities hold in the model of Theorem \ref{main}, each pair of cardinals are independent. In particular, the characterizations $\add(\Ncal)=\min\{\bfrak,\minLc\}$ and $\cof(\Ncal)=\max\{\dfrak,\supLc\}$ (see Theorem \ref{bLoc(b,h)addN}) are optimal. Note that $\bfrak\leq\minLc$ implies $\add(\Ncal)=\minadd=\add(\Ical_{\id_\omega})=\bfrak$, and $\supLc\leq\dfrak$ implies $\cof(\Ncal)=\supcof=\cof(\Ical_{\id_\omega})=\dfrak$.

\begin{theorem}\label{largeminLc}
   Let $\mu\leq\nu$ be uncountable regular cardinals and let $\lambda\geq\nu$ be a cardinal with $\lambda=\lambda^{<\mu}$. Then, there is a ccc poset that forces $\add(\Ncal)=\bfrak=\mu$, $\minLc=\supLc=\nu$ and $\dfrak=\cfrak$.
\end{theorem}
\begin{proof}
  Construct $\Pbb$ by a FS iteration of length $\lambda\nu$ of $\LOCbb^{\id_\omega}_b$ for every $b\in\omega^\omega$ (including those that appear in intermediate extensions), and of all the ccc posets of size $<\mu$ with underlying set an ordinal $<\mu$ (like in the previous results). We also demand that $\LOCbb^{\id_\omega}_b$ for each $b\in\omega^\omega$ is used cofinally often. By this construction, $\Pbb$ is $\mu$-$\Dbf$-good (by Lemma \ref{3.10} and Example \ref{ExmPrs}(3)) and it forces $\mathrm{MA}_{<\mu}$, so it forces $\add(\Ncal)=\bfrak=\mu$ and $\dfrak=\cfrak=\lambda$ by Theorem \ref{3.13}. On the other hand, it is clear that $\Pbb$ forces $\blc_{b,\id_\omega}=\non(\Mcal)=\cov(\Mcal)=\dlc_{b,id_\omega}=\nu$ for any $b\in\omega^\omega$ (also thanks to the Cohen reals added at limit stages), so $\minLc=\supLc=\nu$.
\end{proof}

With respect to the pairs $\bfrak,\minaLc$ and $\dfrak,\supaLc$, it is clear that a FS iteration of big size but short cofinality of Hechler posets forces $\supaLc<\bfrak$ and $\dfrak<\minaLc$ ($\bfrak<\dfrak$ can be additionally forced with a matrix iteration as in Theorem \ref{addN<addIf} but using only (i)-(iii) for the construction). If after this iteration we force with a large random algebra, then $\minaLc=\non(\Ncal)=\aleph_1<\bfrak$ and $\dfrak<\cof(\Ncal)=\supaLc$ is satisfied in the final extension.

\section{Discussions and Open Questions}\label{SecQ}

Though we constructed a model of ZFC where the four cardinal invariants associated with many Yorioka ideals are pairwise different, we still do not know how to construct a model where we can separate those cardinals for \emph{all} Yorioka ideals.

\begin{Question}\label{Q1}
Is there a model of ZFC satisfying $\add(\Ical_f)<\cov(\Ical_f)<\non(\Ical_f)<\cof(\Ical_f)$ for all increasing $f:\omega\to\omega$?
\end{Question}

It is not even known whether there is such a model for just $f=\mathrm{id}_{\omega}$.

There are many ways to attack this problem. One would be to find a gPrs $\Rcal$ that satisfies $\add(\Ical_{\id_\omega})\leq\bfrak(\Rcal)$ and $\dfrak(\Rcal)\leq\cof(\Ical_{\id_\omega})$ and such that any poset of the form $\Ebb_b$ is $\Rcal$-good. If such a gPrs can be found, then a construction as in 
Theorem \ref{main} works. Other way is to adapt the techniques for not adding dominating reals from \cite{GMS} in matrix iterations so that a construction as in Remark \ref{remnoHechler} could force $\bfrak\leq\mu$ and $\lambda\leq\dfrak$. This is quite possible because $\Ebb_b$ is $\mathbf{D}$-good (see Example \ref{ExmPrs}(3)), though its restriction to some internal model may add dominating reals (see \cite{Paw-Dom}). Success with this method would actually solve the following problem.

\begin{Question}\label{Q4}
Is there a model of ZFC satisfying $\add(\Mcal)<\cov(\Mcal)<\non(\Mcal)<\cof(\Mcal)$?
\end{Question}

Such a model exists assuming strongly compact cardinals (see \cite{GKS}), but it is still open modulo ZFC alone.

Other way to attack Question \ref{Q1} is by large products of creature forcing (see \cite{Kell,KellS,FGKS}). Quite recently, A. Fischer, Goldstern, Kellner and Shelah \cite{FGKS} used this technique to prove that 5 cardinal invariants on the right of Cicho\'n's diagram are pairwise different. Such method would also work to solve the dual of Question \ref{Q1}, that is,

\begin{Question}\label{Q5}
Is there a model of $\add(\Ical_f)<\non(\Ical_f)<\cov(\Ical_f)<\cof(\Ical_f)$ for all increasing $f:\omega\to\omega$?
\end{Question}


There are still many open questions in ZFC about the relations between the cardinal invariants associated with Yorioka ideals and the cardinals in Cicho\'n's diagram. For instance, it is unknown whether $\minadd\leq\cov(\Ncal)$ (or $\non(\Ncal)\leq\supcof$) is provable in ZFC. In general, very little is known about the additivity and the cofinality of Yorioka ideals.

\begin{Question}\label{Q2}
Is it consistent with ZFC that infinitely many cardinal invariants of the form $\add(\Ical_f)$ are pairwise different?
\end{Question}

\begin{Question}\label{Q3}
   Is it true in ZFC that $\add(\Ical_{f})=\add(\Ical_{\mathrm{id}_{\omega}})$ (and $\cof(\Ical_f)=\cof(\Ical_{\mathrm{id}_{\omega}})$) for all (or some) increasing $f$?
\end{Question}

Question \ref{Q2} is related with finding a gPrs associated to $\add(\Ical_f)$ and $\cof(\Ical_f)$. 

\begin{Question}\label{Qminnadd}
  Is it true in ZFC that $\add(\Ncal)=\minadd$ (and $\cof(\Ncal)=\supcof$)?
\end{Question}

Theorem \ref{bLoc(b,h)addN} indicates that the method of Theorem \ref{addN<addIf} cannot be used to increase \emph{every} cardinal of the form $\add(\Ical_f)$ while preserving $\add(\Ncal)$ small.


Concerning the problem of separating the four cardinal invariants associated with an ideal, the consistency of $\add(\Ncal)<\cov(\Ncal)<\non(\Ncal)<\cof(\Ncal)$ is known from \cite{M}.

\begin{Question}\label{Q6}
Is there a model of ZFC satisfying $\add(\Ncal)<\non(\Ncal)<\cov(\Ncal)<\cof(\Ncal)?$
\end{Question}

The same question for $\Mcal$ is also open.

Though in Theorem \ref{main} we were able to separate (infinitely) many localization and anti-localization cardinals, the $\bfrak$-localization cardinals appear below the $\bfrak$-anti-localization ones. The reason of this is that the preservation methods we use for the localization cardinals relies on the structure $\Rbb^{\rho^*}$, so after including one $\bfrak$-anti-localization cardinal with functions quite above $\rho^*$ (to preserve the previous localization cardinals), a new forcing to increase some other localization cardinal may not preserve the previous anti-localization ones. In view of this, it would be interesting to improve our preservation methods to allow both type of cardinals appear alternatively.


\bibliography{prop}
\bibliographystyle{alpha}

\end{document}